\documentclass[a4paper,11pt]{amsart}
\usepackage{amssymb}
\usepackage{amscd}
\usepackage{comment}
\usepackage{amsmath,amsthm}
\usepackage[colorlinks=true]{hyperref}
\usepackage{enumerate}
\usepackage{booktabs,multirow}
\usepackage{tikz}
\usepackage{rotating}
\usepackage{ytableau}
\usetikzlibrary{patterns}
\usetikzlibrary{decorations.pathreplacing}
\usetikzlibrary{calc,through}

\allowdisplaybreaks[1]
\numberwithin{figure}{section}
\numberwithin{equation}{section}

\title{Promotion and rowmotion in rational Catalan combinatorics}
\author[K.~Shigechi]{Keiichi~Shigechi}
\email{k1.shigechi AT gmail.com}
\date{\today}

\newcommand\tikzpic[2]{
\raisebox{#1\totalheight}{
\begin{tikzpicture}
#2
\end{tikzpicture}
}}

\newtheorem{theorem}[figure]{Theorem}
\newtheorem{example}[figure]{Example}
\newtheorem{lemma}[figure]{Lemma}
\newtheorem{defn}[figure]{Definition}
\newtheorem{prop}[figure]{Proposition}
\newtheorem{cor}[figure]{Corollary}

\newtheorem{remark}[figure]{Remark}
\begin{document}

\begin{abstract}
We study four bijections, which are promotion, evacuation, rowmotion, and rowvacuation, 
on generalized Dyck paths in rational Catalan combinatorics.
We define the maps on generalized Dyck paths, which have their origins in maps on Dyck paths and non-crossing partitions.
They include rotation, Kreweras complement map, Simion--Ullman involution on non-crossing partitions, and Lalanne--Kreweras involution
on Dyck paths.
These maps have an expression in terms of the four combinatorial bijections.
By extending the bijection studied by D.~Armstrong, C.~Stump, and H.~Thomas on one hand, and the correspondence 
of RSK type studied by B.~Adenbaum and S.~Elizalde on the other, we present the equivalence 
between the two bijections, promotion and rowmotion, on generalized Dyck paths
through these bijection and correspondence. 
For this purpose, we provide an alternative description of the correspondence of RSK type 
in terms of Dyck tilings.
\end{abstract}

\maketitle

\setcounter{tocdepth}{1}
\tableofcontents

\section{Introduction}
{\it Evacuation} was introduced by M.-P. Sch\"utzenberger \cite{Schu61} in the study 
of the Robinson--Schensted--Knuth (RSK) correspondence, which relates a permutation 
in the symmetric group to a pair of Young tableaux of the same shape. 
In subsequent papers, Sch\"utzenberger generalized the definition of evacuation
to linear extensions of a finite poset \cite{Schu72,Schu76}.
Evacuation can be described by another combinatorial operation called {\it promotion}, which 
is simpler than evacuation.
Promotion and evacuation are fundamental bijections to study the set of 
linear extensions of a finite poset.
A.~Brouwer and A.~Schrijver \cite{BroSch74} defined an action on a poset, which is called 
{\it rowmotion} in \cite{StrWil12}, by generalizing an action on hypergraphs \cite{Duc74}.
As in the case of promotion and evacuation, we have another operation called {\it rowvacuation}.
These four operations are described by a composition of involutions called {\it toggles} by a general 
theory (see e.g. \cite{CamFon95,Sta09}).
The similarities between promotion and evacuation, and rowmotion and rowvacuation by the general theory 
suggest a unified treatment of these four operations.

A Dyck path and a non-crossing partition are fundamental combinatorial objects in Catalan combinatorics. 
The number of these objects of size $n$ is given by the $n$-th Catalan number.
Let $a$ and $b$ be coprime positive integers.
A generalized Dyck path, or an $(a,b)$-Dyck path, is a natural generalization of a Dyck path, and 
belongs to the rational Catalan combinatorics.
The combinatorial object corresponding to non-crossing partitions in the rational Catalan 
combinatorics is a $k$-chain of non-crossing partitions if $(a,b)=(1,k)$.
In this paper, we study the four combinatorial operations in a unified manner 
on the poset of generalized Dyck paths and $k$-chains.
We define and study various maps on generalized Dyck paths, which are generalizations of the maps 
on Dyck paths, and describe them in terms of promotion, evacuation, rowmotion and rowvacuation.

Let us first review the case of classical Dyck paths.
Following \cite{AdeEli23}, we introduce a correspondence between a Dyck path
and a $321$-avoiding permutation, and another correspondence using the RSK correspondence.
By using the correspondence between a Dyck path and a $321$-avoiding permutation,
we introduce a bijection on Dyck paths, known as {\it Lalanne--Kreweras} involution \cite{Kre70,Lal92}.  
This involution has been studied in terms of rowvacuation in \cite{HopJos22,Pan09}.
In \cite{AdeEli23}, the correspondence of RSK type is shown to be equivalent 
to another map, the Armstrong--Stump--Thomas bijection, which we call {\it matching map}, on Dyck paths given in \cite{ArmStuTho13}.

We study the above mentioned four combinatorial operations on the poset of rational 
Dyck paths, which we call $(a,b)$-Dyck paths with $a$ and $b$ coprime.
We mainly focus on the case $(a,b)=(1,k)$ with $k\ge2$.
The number of the $k$-Dyck paths of size $n$ is given by the well-known number, the $n$-th Fuss--Catalan
number.
To study $k$-Dyck paths, we introduce the notion of non-crossing partitions following \cite{Kre72}.
By \cite{Ede80,Ede82}, the number of $k$-chains in the lattice of non-crossing partitions is also 
given by the Fuss--Catalan number. 
In \cite{Ede80,Ede82}, Edelman gave a bijection between a $k$-Dyck path and a $k$-chain 
of non-crossing partitions by use of $k$-ary trees.
On the other hand, C.~Stump introduce a bijection between Dyck paths and non-crossing partitions
in the case of $k=1$ in \cite{Stu13}. 
We adopt the bijection studied by Stump since it reflects a combinatorial structure of a $k$-chain 
of non-crossing partition, and generalize it to general $k\ge2$.
To have the bijection between $k$-Dyck paths and $k$-chains of non-crossing partitions in hand,
the maps on non-crossing partitions can be generalized to define the maps 
on $k$-chains of non-crossing partitions.
In what follows, we focus on four maps on non-crossing partitions. 

There are several distinguished actions on non-crossing partitions.
Four of them are the rotation, the {\it Kreweras complement map} \cite{Kre72}, {\it Simion--Ullman involution} \cite{SimUll91},
and {\it Lalanne--Kreweras involution} \cite{Kre70,Lal92}.
The first three maps come from the circular representation of non-crossing partitions.
The last involution comes from the Lalanne--Kreweras involution on Dyck paths. 
The key is that we make use of the conjugation of Lalanne--Kreweras involution on Dyck paths 
by the correspondence of RSK type to define the Lalanne--Kreweras involution on
non-crossing partitions.
This construction allows us to define the Lalanne--Kreweras involution on non-crossing partitions 
that is compatible with the correspondence between Dyck paths and non-crossing partitions.
Under the correspondence between a Dyck path and a non-crossing partition, 
we show that these four maps on non-crossing partitions have equivalent descriptions 
in terms of promotion and evacuation (Proposition \ref{prop:NCpro}).

To study $k$-Dyck paths with $k\ge2$, we will generalize the bijection by Stump to the case of $k\ge2$.
Since a $k$-Dyck path corresponds to a $k$-chain of non-crossing partitions, we can 
define analogues of the four maps in the previous paragraph.
We show that three maps except the Kreweras complement map can be described in terms of 
promotion and evacuation under the correspondence between $k$-Dyck paths and $k$-chains 
(Proposition \ref{prop:kDyckNC}). 
In the case of $k=1$, it is well-known that the Kreweras complement map is equivalent to the promotion. 
However, this is not true for $k\ge2$.
We introduce a new map on $k$-chains, which we call {\it lift}, by using the circular representation
of $k$-chains, and show that the lift corresponds to the promotion (Proposition \ref{prop:Lift}).

Another interesting map in Catalan combinatorics is the matching map on Dyck paths introduced by Armstrong--Stump--Thomas 
in \cite{ArmStuTho13}. 
We generalize the notion of the matching map to a general $(a,b)$ with $a$ and $b$ coprime, and 
study it in detail in the case of $(a,b)=(1,k)$ with $k\ge2$.
In \cite{AdeEli23}, they show the equivalence between the correspondence $\widehat{RSK}$ of RSK type 
and the matching map $\mathtt{Mat}$ for $k=1$.  

For $k=1$, it is well-known that a Dyck path corresponds to a two-row Young tableau. 
However, a $k$-Dyck path corresponds 
to two-row $(1,k)$-Young tableau, which is a variant of a set-valued two-row Young 
tableau (see Section \ref{sec:Bg}). By using another realization of a $k$-Dyck path called a perfect matching, 
a two-row $(1,k)$-Young tableau can be interpreted as a Young tableau with $k+1$ rows which 
satisfies a certain condition.
At the moment, we have no analogue of the correspondence of RSK type even for $(1,k)$ with $k\ge2$.
The equivalence in the previous paragraph implies that one can define an analogue of 
the correspondence of RSK type by using the matching map for general $(a,b)\neq(1,1)$.

We take a closer look at this equivalence.
We first study the relation between the correspondence $\widehat{RSK}$ of RSK type and the 
matching map for $k=1$ (Theorem \ref{thrm:RSKMat}, and see also Theorem 5.2 in \cite{AdeEli23}).
The map $\widehat{RSK}$ can be expressed in terms of promotion $\partial$ and the matching map $\mathtt{Mat}$:
\begin{align}
\label{eq:RSKMat0}
\widehat{RSK}=\partial^{-(n-1)}\circ\mathtt{Mat},
\end{align}
where $n$ is the size of a Dyck path.

In Section \ref{sec:MatRSK}, we study another description of $\widehat{RSK}$ for $k=1$ by using 
a perfect matching. This description uses the property of a $321$-avoiding permutation, therefore 
it is not straightforward to generalize this description to the case $k\ge2$.
On the other hand, we have alternative description of $\widehat{RSK}^{-1}$ by using another combinatorial 
object which we call a {\it Dyck tiling}.
This description depends only on the Young diagram which is determined by a Dyck path, and 
we do not use the property of a $321$-avoiding permutation.
This allows us to define $\widehat{RSK}^{-1}$ in terms of a Young diagram and a Dyck tiling on it.
We have no analogue of $\widehat{RSK}$ for $k\ge2$ as already mentioned, 
but we can make use of $k$-Dyck tilings, which is a generalization of Dyck tilings, 
to define an analogue of $\widehat{RSK}^{-1}$ for $k\ge2$.
We show that $\widehat{RSK}^{-1}$ for $k\ge2$ defined through $k$-Dyck tilings satisfies 
the same relation as Eq. (\ref{eq:RSKMat0}) (Proposition \ref{prop:RSKMatk}).
This implies that one can define an analogue of $\widehat{RSK}$ for general $(a,b)$
through the relation (\ref{eq:RSKMat0}).

The matching map $\mathtt{Mat}$ or equivalently $\widehat{RSK}$ plays the role 
to relate promotion and evacuation to rowmotion and rowvacuation, and vice versa.
More precisely, if we write promotion and rowmotion as $\partial$ and $\delta$, we have 
\begin{align}
\label{eq:Xprorow}
\partial^{-1}\circ X=X\circ\delta,
\end{align}
where $X\in\{\widehat{RSK},\mathtt{Mat}\}$ (see Propositions \ref{prop:cd1} and \ref{prop:RSKk}).
Recall that we have several maps on $k$-chains of non-crossing partitions.
By using the correspondence between $k$-Dyck paths and $k$-chains, 
we show that the rotation, Simion--Ullman involution, and Lalanne--Kreweras involution
are described in terms of rowmotion and rowvacuation through the matching map (Proposition \ref{prop:Matdelta}).

For general $(a,b)$, we define the matching map, and show that it satisfies the same 
relation as Eq. (\ref{eq:Xprorow}).
If $(a,b)\neq(1,k)$ with $k\ge1$, we have no analogue of a description of $(a,b)$-Dyck paths in 
terms of non-crossing partitions. However, we have the actions of promotion, evacuation, rowmotion and rowvacuation 
on the poset of $(a,b)$-Dyck paths by a general theory. 
To find a nice combinatorial description of $(a,b)$-Dyck paths 
similar to non-crossing partitions remains as an open problem.
If such description exists, then one can define analogues of the maps on non-crossing partitions and 
describe them in terms of both promotion and rowmotion.
Further, the existence of the matching map and $(a,b)$-Dyck tiling implies the existence of 
the correspondence of RSK type. To find a correspondence of RSK type for $(a,b)\neq(1,1)$ is also
an interesting problem.

The paper is organized as follows.
In Section \ref{sec:Bg}, we briefly review the notions of $(a,b)$-Dyck paths, perfect matchings,
the Lalanne--Kreweras involution on Dyck paths, $321$-avoiding permutations, and the map $\widehat{RSK}$.
In Section \ref{sec:proeva}, we collect basic results on promotion and evacuation, and study 
their actions on Dyck paths.
In Section \ref{sec:row}, we review rowmotion and rowvacuation, and study their actions on Dyck paths.
In Section \ref{sec:MatRSK}, we study the relation between the map of RSK type and the matching map.
A description of the inverse map of RSK type in terms of Dyck tilings is given.
We study the weighted non-crossing partitions, which are equivalent to $k$-chains of non-crossing partitions,
in Section \ref{sec:NCPTL}.
We study the maps on $k$-chains: the rotation, the Kreweras complement map, the Simion--Ullman involution, 
and the Lalanne--Kreweras involution. 
A description of the maps on $k$-chains of non-crossing partitions in terms of promotion and evacuation
is given. 
Similarly in Section \ref{sec:prorow}, by studying the matching map, a description of the maps on $k$-chains in terms of 
rowmotion and rowvacuation is given.
In Section \ref{sec:ex}, we give an example where $(a,b)=(1,2)$ and $n=3$.

\section{Background}
\label{sec:Bg}
\subsection{Dyck paths and rational Dyck paths}
Let $a,b$ be coprime positive integers. 
A {\it rational Dyck path}, or equivalently an $(a,b)$-Dcyk path, of size $n$
is a lattice path $p$ from $(0,0)$ to $(bn,an)$ which satisfies the following 
conditions.
\begin{enumerate}
\item The path $p$ consists of up steps $(0,1)$ and right steps $(1,0)$. 
\item The path $p$ is above $y=ax/b$.
\end{enumerate}
By definition, we have $an$ up steps and $bn$ right steps in a rational 
Dyck path of size $n$.

\begin{defn}
We denote by $\mathtt{Dyck}_{(a,b)}(n)$ the set of $(a,b)$-Dyck paths 
of size $n$.
\end{defn}

When $(a,b)=(1,1)$, rational $(1,1)$-Dyck paths coincide with the standard 
notion of the Dyck paths.
When $(a,b)=(1,k)$ with $k\ge2$, we call $(1,k)$-Dyck paths $k$-Dyck paths.

\begin{defn}
An $(a,b)$-Dyck path $p$ is said to be prime if $p$ touches the line $y=ax/b$
twice, i.e., $p$ is strictly above the line except two points $(0,0)$ and $(bn,an)$.
\end{defn}

The total number of $(a,b)$-Dyck paths of size $n$ is given by Grossman's 
formula \cite{Biz54,Gro50}:
\begin{align*}
|\mathtt{Dyck}_{(a,b)}(n)|
=
\sum \prod_{j\ge1}\genfrac{}{}{}{}{A_{j}^{k_{j}}}{k_{j}!},
\end{align*}
where 
\begin{align*}
A_{j}=\genfrac{}{}{}{}{1}{j(a+b)}\genfrac{(}{)}{0pt}{}{(a+b)j}{aj},
\end{align*}
and the sum is taken overall positive integral $k_{j}$ such that 
$k_j\ge0$ and $\sum j k_{j}=n$.
When $(a,b)=(1,1)$, the formula is reduced to Catalan number 
\begin{align*}
|\mathtt{Dyck}_{(1,1)}(n)|=\genfrac{}{}{}{}{1}{n+1}\genfrac{(}{)}{0pt}{}{2n}{n}.
\end{align*}
Similarly, when $(a,b)=(1,k)$, we have Fuss--Catalan number 
\begin{align*}
|\mathtt{Dyck}_{(1,k)}(n)|=\genfrac{}{}{}{}{1}{kn+1}\genfrac{(}{)}{0pt}{}{(k+1)n}{n}.
\end{align*}

By definition of rational Dyck paths, we have $an$ up steps in an $(a,b)$-Dyck path.
Suppose that the $i$-th step from $(0,0)$ is the $j$-th up step.
Then, we define the step sequence $\mathbf{u}:=(u_1,\ldots, u_{an})$
by $u_{j}=i$.
The step sequence satisfies the following conditions: $\mathbf{u}$ is 
strictly increasing positive integers in $[1,(a+b)(n-1)+1]$, and $1\le u_j\le \lfloor(j-1)b/a\rfloor+j$, where 
$\lfloor x\rfloor$ is  the floor function.

\begin{figure}[ht]
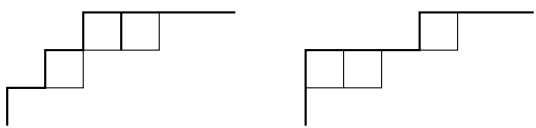

\tikzpic{-0.5}{[scale=0.5]
\draw(0,0)--(0,1)--(2,1)--(2,2)--(4,2)--(4,3)--(6,3);
\draw[thick](0,0)--(0,1)--(1,1)--(1,2)--(2,2)--(2,3)--(6,3)(3,2)--(3,3);
}
\quad
\tikzpic{-0.5}{[scale=0.5]
\draw(0,0)--(0,1)--(2,1)--(2,2)--(4,2)--(4,3)--(6,3)(1,1)--(1,2);
\draw[thick](0,0)--(0,2)--(3,2)--(3,3)--(6,3);
}
\caption{Examples of  $2$-Dyck paths}
\label{fig:GDP}
\end{figure}
Two examples of generalized Dyck paths of size $3$ with $k=2$ are shown 
in Figure \ref{fig:GDP}.
The step sequences for these generalized Dyck paths are $(1,3,5)$ and 
$(1,2,6)$ respectively.
Similarly, in the case of $(a,b)=(2,3)$ and $n=2$, we have $23$ rational 
Dyck paths. The step sequences of the lowest path and the top path are 
$(1,3,6,8)$ and $(1,2,3,4)$ respectively.

For simplicity, we often write a step sequence $\mathbf{u}=u_1\ldots u_{an}$ as a one-line notation.

It is well-known that a Dyck path of size $n$ is bijective to a Young tableau of 
shape $2\times n$.
Let $\mathbf{u}(p)=(u_1,\ldots,u_n)$ be the step sequence of a Dyck path $p$.
The Young tableau $Y$ can be constructed from $\mathbf{u}(p)$ as follows.
The contents of the first row of $Y$ is $\{u_1,\ldots,u_n\}$ and 
the contents of the second row of $Y$ is $[1,2n]\setminus\{u_1,\ldots,u_n\}$.

\begin{example}
\label{ex:Dyck}
We have five Dyck paths for $n=3$.
They are 
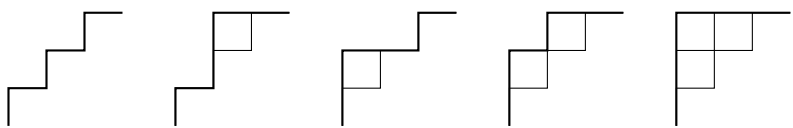
\begin{align*}
\tikzpic{-0.5}{[scale=0.5]
\draw[thick](0,0)--(0,1)--(1,1)--(1,2)--(2,2)--(2,3)--(3,3);
}\quad
\tikzpic{-0.5}{[scale=0.5]
\draw(0,0)--(0,1)--(1,1)--(1,2)--(2,2)--(2,3)--(3,3);
\draw[thick](0,0)--(0,1)--(1,1)--(1,3)--(3,3);
}\quad
\tikzpic{-0.5}{[scale=0.5]
\draw(0,0)--(0,1)--(1,1)--(1,2)--(2,2)--(2,3)--(3,3);
\draw[thick](0,0)--(0,2)--(2,2)--(2,3)--(3,3);
}\quad
\tikzpic{-0.5}{[scale=0.5]
\draw(0,0)--(0,1)--(1,1)--(1,2)--(2,2)--(2,3)--(3,3);
\draw[thick](0,0)--(0,2)--(1,2)--(1,3)--(3,3);
}\quad
\tikzpic{-0.5}{[scale=0.5]
\draw(0,0)--(0,1)--(1,1)--(1,2)--(2,2)--(2,3)--(3,3);
\draw(0,2)--(1,2)--(1,3);
\draw[thick](0,0)--(0,3)--(3,3);
}
\end{align*}
The step sequences of five Dyck paths are 
\begin{align*}
135 \quad 134 \quad 125 \quad 124 \quad 123.
\end{align*}
The Young tableaux corresponding to these Dyck paths are 
\begin{align*}
\begin{ytableau}
1 & 3 & 5  \\
2 & 4 & 6 
\end{ytableau}\quad
\begin{ytableau}
1 & 3 & 4  \\
2 & 5 & 6 
\end{ytableau}\quad
\begin{ytableau}
1 & 2 & 5  \\
3 & 4 & 6 
\end{ytableau}\quad
\begin{ytableau}
1 & 2 & 4  \\
3 & 5 & 6 
\end{ytableau}\quad
\begin{ytableau}
1 & 2 & 3 \\
4 & 5 & 6
\end{ytableau}
\end{align*}
\end{example}

We generalize the correspondence between Dyck paths and Young tableaux with two rows 
to the case of rational $(a,b)$-Dyck paths.
For this purpose, we generalize a Young tableau as follows.
An {\it $(a,b)$-Young diagram} of size $n$ is a diagram consisting of $(a+b)n$ cells 
such that there are $an$ cells of shape $1\times b$ in the first row and 
$bn$ cells of shape $1\times a$ in the second row.
An {\it $(a,b)$-Young tableau} is a filling of an $(a,b)$-Young diagram by integers 
in $\{1,\ldots,(a+b)n\}$.
The integers are increasing from left to right and top to bottom.
Let $p$ be an $(a,b)$-Dyck path of size $n$ and suppose $\mathbf{u}(p)=(u_1,\ldots,u_{an})$ 
is the step sequence of $p$.
An $(a,b)$-Young tableau $T$ for $p$ is given in the same way as a Dyck path.
The contents of the cells in the first row of $T$ are $\{u_1,\ldots,u_{an}\}$, and 
the contents of the cells in the second row are $\{1,\ldots,(a+b)n\}\setminus\{u_1,\ldots,u_{an}\}$.

\begin{defn}
We denote by $\mathtt{SYT}_{(a,b)}(n)$ the set of standard $(a,b)$-Young tableau of size $n$.
\end{defn}

The following proposition is a direct consequence of a construction of 
$(a,b)$-Young tableaux and the definition of $(a,b)$-Dyck paths, i.e, 
we have a natural bijection between an $(a,b)$-Young tableau and an $(a,b)$-Dyck path.
\begin{prop}
We have $|\mathtt{SYT}_{(a,b)}(n)|=|\mathtt{Dyck}_{(a,b)}(n)|$.
\end{prop}

Two generalized Dyck paths in Figure \ref{fig:GDP} correspond to the following 
$(1,2)$-Young tableaux of size $3$:
\begin{align}
\label{eq:12Yt135}
\tikzpic{-0.5}{[scale=0.5]
\draw(0,0)--(6,0)--(6,-2)--(0,-2)--(0,0);
\draw(0,-1)--(6,-1)(2,0)--(2,-2)(4,0)--(4,-2);
\draw(1,-1)--(1,-2)(3,-1)--(3,-2)(5,-1)--(5,-2);
\draw(1,-0.5)node{$1$}(3,-0.5)node{$3$}(5,-0.5)node{$5$};
\draw(0.5,-1.5)node{$2$}(1.5,-1.5)node{$4$}(2.5,-1.5)node{$6$}(3.5,-1.5)node{$7$}(4.5,-1.5)node{$8$}
(5.5,-1.5)node{$9$};
}\quad
\tikzpic{-0.5}{[scale=0.5]
\draw(0,0)--(6,0)--(6,-2)--(0,-2)--(0,0);
\draw(0,-1)--(6,-1)(2,0)--(2,-2)(4,0)--(4,-2);
\draw(1,-1)--(1,-2)(3,-1)--(3,-2)(5,-1)--(5,-2);
\draw(1,-0.5)node{$1$}(3,-0.5)node{$2$}(5,-0.5)node{$6$};
\draw(0.5,-1.5)node{$3$}(1.5,-1.5)node{$4$}(2.5,-1.5)node{$5$}(3.5,-1.5)node{$7$}(4.5,-1.5)node{$8$}
(5.5,-1.5)node{$9$};
}
\end{align}

\begin{remark}
\label{remark:svSYT}
A Dyck path  corresponds to a two-row Young tableau.
A $(1,k)$-Young tableau can be regarded as a set-valued two-row Young tableau,
or a Young tableau with $k+1$ rows.
For example, the left $(1,2)$-Young tableau in Eq. (\ref{eq:12Yt135}) corresponds to
the following tableaux:
\begin{align*}
\tikzpic{-0.5}{[scale=0.5]
\draw(0,0)--(6,0)--(6,-2)--(0,-2)--(0,0);
\draw(0,-1)--(6,-1)(2,0)--(2,-2)(4,0)--(4,-2);
\draw(1,-0.5)node{$1$}(3,-0.5)node{$3$}(5,-0.5)node{$5$};
\draw(1,-1.5)node{$2,4$}(3,-1.5)node{$6,7$}(5,-1.5)node{$8,9$};
}\qquad
\tikzpic{-0.5}{[scale=0.5]
\foreach \x in {0,1,2,3} \draw(\x,0)--(\x,3)(0,\x)--(3,\x);
\draw(0.5,2.5)node{$1$}(0.5,1.5)node{$2$}(0.5,0.5)node{$4$};
\draw(1.5,2.5)node{$3$}(1.5,1.5)node{$6$}(1.5,0.5)node{$7$};
\draw(2.5,2.5)node{$5$}(2.5,1.5)node{$8$}(2.5,0.5)node{$9$};
}\qquad
\tikzpic{-0.5}{[scale=0.5]
\foreach \x in {0,1,2,3} \draw(\x,0)--(\x,3)(0,\x)--(3,\x);
\draw(0.5,2.5)node{$1$}(0.5,1.5)node{$2$}(0.5,0.5)node{$7$};
\draw(1.5,2.5)node{$3$}(1.5,1.5)node{$4$}(1.5,0.5)node{$8$};
\draw(2.5,2.5)node{$5$}(2.5,1.5)node{$6$}(2.5,0.5)node{$9$};
}
\end{align*}
The left tableau is a set-valued two-row Young tableau obtained from 
the $(1,2)$-Young tableau, and the middle and right tableau is
Young tableaux with three rows.
The middle tableau is obtained from the set-valued two-row Young tableau
by putting the integers in a column on the cells in the same column of the 
three-row Young diagram.
The right tableau is obtained from the perfect matching of a $(1,k)$-Dyck 
path discussed later.

Note that not all $(k+1)\times n$ Young tableaux correspond to 
$(1,2)$-Dyck paths.
\end{remark}

For simplicity, we first introduce perfect matchings for Dyck paths.
Let $D$ be a Dyck path of size $n$. 
A {\it non-crossing perfect matching} of the path $p$ is defined as follows.
Suppose that $\mathbf{u}(p)$ is the step sequence of $p$ and 
$Y(p)$ is the corresponding Young tableau. 
We denote by $\mathbf{u}^{\vee}(p)$ the set of integers in the 
second row of $Y(p)$.
We construct a collection $S(p)$ of the integer sets as follows.
\begin{enumerate}
\item Set $\mathbf{v}:=\mathbf{u}(p)$, $\mathbf{v}^{\vee}:=\mathbf{u}^{\vee}(p)$, and $S(p)=\emptyset$.
\item Take the maximum element $l$ in $\mathbf{v}$,  {\it i.e.}, $l:=\max(\mathbf{v})$.
\item Take the element $r$ such that $r\in\mathbf{v}^{\vee}$ and $r$ is minimum among $l<r$.
\item Replace $S(p)$ by $S(p)\cup\{\{l,r\}\}$, $\mathbf{v}$ by $\mathbf{v}\setminus\{l\}$, 
and $\mathbf{v}^{\vee}$ by $\mathbf{v}^{\vee}\setminus\{r\}$.
\item Go to (2) and the algorithm stops when $\mathbf{v}, \mathbf{v}^{\vee}=\emptyset$.
\end{enumerate}
A collection $S(p)$ consists of the set of $n$ pairs of integers.
We call a collection $S(p)$ a perfect matching of $p$.

By construction, $S(p)$ is a collection of sets of integers.
Consider a circle with $2n$ points which are numbered clockwise.
In the case of Dyck paths, the collection $S(p)$ is a collection of 
$n$ pairs of two integers. We denote them $S_{1},S_{2},\ldots,S_{n}$. 
To depict the collection $S(p)$ visually, we connect 
integers in $S_{j}$, $1\le j\le n$, by arcs.

For example, we consider the middle Dyck path in Example \ref{ex:Dyck}.
We have the following correspondence:
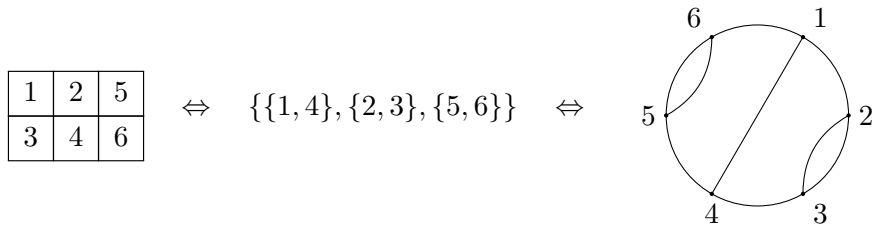
\begin{align*}
\begin{ytableau}
1 & 2 & 5  \\
3 & 4 & 6 
\end{ytableau}
\quad \Leftrightarrow \quad	
\{\{1,4\},\{2,3\},\{5,6\}\}\quad \Leftrightarrow \quad
\tikzpic{-0.5}{[scale=0.4]
\draw circle(3cm);
\foreach \a in {60,0,-60,-120,-180,-240}
\filldraw [black](\a:3cm)circle(1.5pt);
\draw(60:3cm)node[anchor=south west]{$1$};
\draw(0:3cm)node[anchor=west]{$2$};
\draw(-60:3cm)node[anchor=north west]{$3$};
\draw(-120:3cm)node[anchor=north]{$4$};
\draw(-180:3cm)node[anchor=east]{$5$};
\draw(-240:3cm)node[anchor=south east]{$6$};
\draw(60:3cm)to(240:3cm);
\draw(0:3cm)to[bend right=30](-60:3cm);
\draw(120:3cm)to[bend left=30](180:3cm);
}
\end{align*}

By construction of $S(p)$, the arcs in the circle representation are 
non-crossing.

We generalize the notion of perfect matchings for Dyck paths to the case of 
rational Dyck paths.
Let $p\in\mathtt{Dyck}_{(a,b)}(n)$ and $\mathbf{u}(p)=(u_1,\ldots,u_{an})$ 
be the step sequence of $p$.
By definition of a step sequence, the $u_{i}$-th step in $p$ is an up step.
This step is an edge from $(x(i),i-1)$ to $(x(i),i)$ with some $x(i)\in\{0,1,\ldots,bn-1\}$.
We construct a collection $S(p)$ of $an$ sets of integers as follows.
\begin{enumerate}
\item Set $\mathbf{v}^{\vee}:=\{1,\ldots,(a+b)n\}\setminus\mathbf{u}(p)$, $\mathbf{v}:=\mathbf{u}(p)$, 
$S(p):=\emptyset$, and $m=an$.
\item Take the maximal element $u_{m}$ from $\mathbf{v}$.
\item Consider a visualization of $p$ by up and down steps. 
We depict a line $l$ whose tangent is $a/b$ from the vertex $(x(m),m-1)$ to right where 
$x(m)$ is the $x$-coordinate of the $m$-th up step.
Let $(s,t)$ be the first intersection of $l$ and the $(a,b)$-Dyck path $p$.
We define an integer $s':=\lfloor s \rfloor$.  
\item We define a set $S_{m}$ of integers by 
\begin{align*}
S_{m}:=\{u_{m}\}\cup\left(\mathbf{v}^{\vee}\cap[x(m)+m+1,s'+t]\right).
\end{align*} 
Replace $S(p)$ by $S(p)\cup\{S_{m}\}$, 
$\mathbf{v}^{\vee}$ by $\mathbf{v}^{\vee}\setminus S_{m}$, $\mathbf{v}$ by $\mathbf{v}\setminus\{u_{m}\}$ 
and $m$ by $m-1$.
\item Go to (2) and the algorithm stops when $m=0$.
\end{enumerate}
We call a collection $S(p)$ of sets of integers a perfect matching of $p$.
We visualize a perfect matching in the same manner as the case of 
Dyck paths.
We call each $S_{j}$, $1\le j\le an$, a matching block of $S(p)$. 
A block $S_{j}$ consists of at least one integer.
When $a>b$, a perfect matching of a rational Dyck path has a block of size one.
As in the case of Dyck paths, blocks are non-crossing in the circle representation.

\begin{defn}
\label{defn:PM}
We define $\mathtt{PM}: p\mapsto S(p)$ as a map from a path $p\in\mathtt{Dyck}_{(a,b)}(n)$ 
to its perfect matching $S(p)$.
\end{defn}

Figure \ref{fig:PM} shows an example of a perfect matching of a rational 
Dyck path. 
In the case of a Dyck path, a block $S_{j}$ consists of two integers.
Similarly, in the case of $k$-Dyck paths, a block $S_{j}$ consists of 
$k+1$ integers.
However, in general, the size $|S_{j}|$ of a block $S_{j}$ in a perfect matching 
of a rational Dyck path for $(a,b)\neq(1,k)$, $k\ge1$, may depend on a path.
In the case of the Dyck path in Figure \ref{fig:PM}, the sizes of blocks 
are two or three.

\begin{figure}[ht]
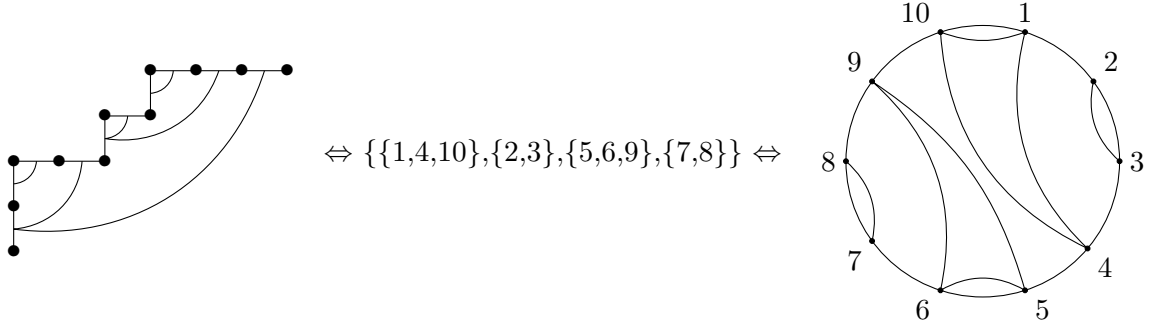

\tikzpic{-0.5}{[scale=0.6]
\draw(0,0)node{$\bullet$}--(0,2)node{$\bullet$}--(2,2)node{$\bullet$}
--(2,3)node{$\bullet$}--(3,3)node{$\bullet$}--(3,4)node{$\bullet$}--(6,4)node{$\bullet$};
\draw(0,1)node{$\bullet$}(1,2)node{$\bullet$};
\draw(4,4)node{$\bullet$}(5,4)node{$\bullet$};
\draw(0,1.5)to[bend right=40](0.5,2);
\draw(0,0.5)to[bend right=40](1.5,2)(0,0.5)to[bend right=40](5.5,4);
\draw(2,2.5)to[bend right=40](2.5,3)(2,2.5)to[bend right=40](4.5,4);
\draw(3,3.5)to[bend right=40](3.5,4);
}
$\Leftrightarrow$
\{\{1,4,10\},\{2,3\},\{5,6,9\},\{7,8\}\}
$\Leftrightarrow$
\tikzpic{-0.5}{[scale=0.6]
\draw circle(3cm);
\foreach \a in {0,36,72,108,144,180,216,252,288,320}
\filldraw [black](\a:3cm)circle(1.5pt);
\draw(72:3cm)node[anchor=south]{$1$};
\draw(36:3cm)node[anchor=south west]{$2$};
\draw(0:3cm)node[anchor=west]{$3$};
\draw(320:3cm)node[anchor=north west]{$4$};
\draw(288:3cm)node[anchor=north west]{$5$};
\draw(252:3cm)node[anchor=north east]{$6$};
\draw(216:3cm)node[anchor=north east]{$7$};
\draw(180:3cm)node[anchor=east]{$8$};
\draw(144:3cm)node[anchor=south east]{$9$};
\draw(108:3cm)node[anchor=south east]{$10$};
\draw(72:3cm)to[bend right=30](320:3cm)to[bend left=30](108:3cm)to[bend right=20](72:3cm);
\draw(36:3cm)to[bend right=30](0:3cm);
\draw(288:3cm)to[bend right=30](252:3cm)to[bend right=30](144:3cm)to[bend left=20](288:3cm);
\draw(216:3cm)to[bend right=30](180:3cm);
}
\caption{A perfect matching of a $(2,3)$-Dyck path.
The rational Dyck path on the left corresponds to the 
perfect matching $\{\{1,4,10\},\{2,3\},\{5,6,9\},\{7,8\}\}$.}
\label{fig:PM}
\end{figure}

\begin{remark}
Let $(a,b)=(1,k)$ and $S(p)$ be a perfect matching of a $(1,k)$-Dyck path.
Recall that $S(p)$ is a collection of the sets of $k+1$ integers, i.e., 
$S(p):=\{S_1,\ldots,S_{n}\}$ where $|S_{i}|=k+1$ for $1\le i\le n$.
We associate a $(k+1)\times n$ standard Young tableau to $S(p)$ in such a way that
we put the integer $p\in S_{i}$ on the $q$-th row if $p$ is the $q$-th smallest in $S_{i}$.
For example, a $(1,2)$-Dyck path whose step sequence is $(1,3,5)$ gives 
the perfect matching $\{\{1,2,9\},\{3,4,8\},\{5,6,7\}\}$.
It is easy to verify that we have the right Young tableau in Remark \ref{remark:svSYT}.
\end{remark}

One can obtain a $(b,a)$-Dyck path from an $(a,b)$-Dyck path 
by taking the mirror image along the diagonal.
To state this relation precisely, we introduce two operations 
$\ast$, $\overline{\bullet}$ on generalized Dyck paths.

First, we define the operator $\ast: \mathtt{SYT}_{(a,b)}(N)\rightarrow\mathtt{SYT}_{(b,a)}(N)$
as follows.
Let $R_1$ and $R_2$ be the first and second rows in $\mathtt{SYT}_{(a,b)}(N)$. 
The operation $\ast$ exchanges the role of $R_1$ and $R_2$ by replacing $i$ by 
$(a+b)N+1-i$.
For example, we have 
\begin{align*}
\ast: \tikzpic{-0.5}{[scale=0.5]
\draw(0,0)--(12,0)--(12,-2)--(0,-2)--(0,0);
\draw(0,-1)--(12,-1);
\draw(3,0)--(3,-1)(6,0)--(6,-1)(9,0)--(9,-1)(12,0)--(12,-1);
\draw(2,-1)--(2,-2)(4,-1)--(4,-2)(6,-1)--(6,-2)(8,-1)--(8,-2)
(10,-1)--(10,-2);
\draw(1.5,-0.5)node{$1$}(4.5,-0.5)node{$2$}(7.5,-0.5)node{$5$}(10.5,-0.5)node{$7$};
\draw(1,-1.5)node{$3$}(3,-1.5)node{$4$}(5,-1.5)node{$6$}(7,-1.5)node{$8$}(9,-1.5)node{$9$}(11,-1.5)node{$10$};
}
\rightarrow
\tikzpic{-0.5}{[scale=0.5]
\draw(0,0)--(12,0)--(12,-2)--(0,-2)--(0,0);
\draw(0,-1)--(12,-1);
\draw(3,-1)--(3,-2)(6,-1)--(6,-2)(9,-1)--(9,-2)(12,-1)--(12,-2);
\draw(2,0)--(2,-1)(4,0)--(4,-1)(6,0)--(6,-1)(8,0)--(8,-1)
(10,0)--(10,-1);
\draw(1.5,-1.5)node{$4$}(4.5,-1.5)node{$6$}(7.5,-1.5)node{$9$}(10.5,-1.5)node{$10$};
\draw(1,-0.5)node{$1$}(3,-0.5)node{$2$}(5,-0.5)node{$3$}(7,-0.5)node{$5$}(9,-0.5)node{$7$}(11,-0.5)node{$8$};
}
\end{align*}

A bar operation on a perfect matching $C(p)$ is defined by replacing $i$ by $(a+b)n+1-i$ for 
$1\le i\le (a+b)n$. We denote it by $C(p)\mapsto\overline{C(p)}$.
We introduce another map to construct a perfect matching:
\begin{defn}
Let $\mathtt{PM}$, $\ast$ and $\overline{\bullet}$ be as above.
We define 
\begin{align*}
\mathtt{dPM}:=\overline{\mathtt{PM}\circ\ast}.
\end{align*}
We say that $\mathtt{dPM}$ is dual to the map $\mathtt{PM}$.
\end{defn}

Figure \ref{fig:dPM} is an example of the map $\mathtt{dPM}$.
When we obtain a perfect matching of an $(a,b)$-Dyck path, we make use of a line with tangent $a/b$.
On the contrary, we make use of a line with tangent $b/a$ for $\mathtt{dPM}$. 
\begin{figure}[ht]
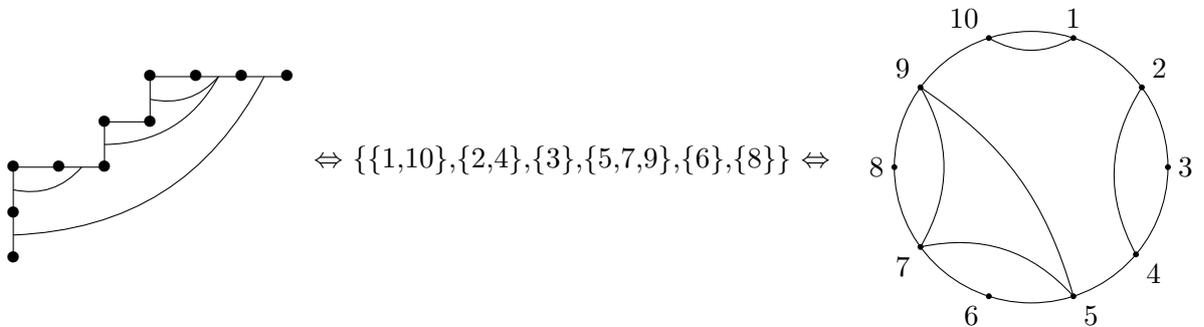

\tikzpic{-0.5}{[scale=0.6]
\draw(0,0)node{$\bullet$}--(0,2)node{$\bullet$}--(2,2)node{$\bullet$}
--(2,3)node{$\bullet$}--(3,3)node{$\bullet$}--(3,4)node{$\bullet$}--(6,4)node{$\bullet$};
\draw(0,1)node{$\bullet$}(1,2)node{$\bullet$};
\draw(4,4)node{$\bullet$}(5,4)node{$\bullet$};
\draw(0,0.5)to[bend right=30](5.5,4);
\draw(0,1.5)to[bend right=30](1.5,2);
\draw(2,2.5)to[bend right=30](4.5,4);
\draw(3,3.5)to[bend right=30](4.5,4);
}$\Leftrightarrow$
\{\{1,10\},\{2,4\},\{3\},\{5,7,9\},\{6\},\{8\}\}
$\Leftrightarrow$
\tikzpic{-0.5}{[scale=0.6]
\draw circle(3cm);
\foreach \a in {0,36,72,108,144,180,216,252,288,320}
\filldraw [black](\a:3cm)circle(1.5pt);
\draw(72:3cm)node[anchor=south]{$1$};
\draw(36:3cm)node[anchor=south west]{$2$};
\draw(0:3cm)node[anchor=west]{$3$};
\draw(320:3cm)node[anchor=north west]{$4$};
\draw(288:3cm)node[anchor=north west]{$5$};
\draw(252:3cm)node[anchor=north east]{$6$};
\draw(216:3cm)node[anchor=north east]{$7$};
\draw(180:3cm)node[anchor=east]{$8$};
\draw(144:3cm)node[anchor=south east]{$9$};
\draw(108:3cm)node[anchor=south east]{$10$};
\draw(108:3cm)to[bend right=30](72:3cm);
\draw(36:3cm)to[bend right=30](320:3cm);
\draw(288:3cm)to[bend right=30](216:3cm)to[bend right=30](144:3cm)to[bend left=20](288:3cm);
}
\caption{The map $\mathtt{dPM}$ on a $(2,3)$-Dyck path.
The rational Dyck path on the left corresponds to the 
perfect matching $\{\{1,10\},\{2,4\},\{3\},\{5,7,9\},\{6\},\{8\}\}$}
\label{fig:dPM}
\end{figure}

Note that the map $\mathtt{dPM}$ coincides with the map $\mathtt{PM}$ in the case 
of Dyck paths, {\it i.e.}, $(a,b)=(1,1)$. This is because we have $a/b=b/a=1$, and 
we apply $\ast$ and $\overline{\bullet}$ successively to a Dyck path.

\subsection{Maps on a Dyck path}
Let $\mathcal{S}_{n}$ be the symmetric group of $[n]:=\{1,2,\ldots,n\}$, and 
$\mathcal{S}_{n}(321)$ be the set of $321$-avoiding permutations in $\mathcal{S}_{n}$.
A permutation $\mathcal{S}_{n}\ni\omega=\omega_1\ldots \omega_{n}$ is said to be $321$-avoiding 
if there is no triplet $i<j<k$ such that $\omega_{i}>\omega_j>\omega_{k}$.
A {\it Rothe diagram} of a permutation $\omega$ is a visualization of $\omega$.
In an $n\times n$ grid, we put a mark on the $\omega_{i}$-th row from bottom and the $i$-th 
column from left.
The main diagonal is the grids from the southwest corner to the northeast corner 
in the Rothe diagram.
Some marks are above or below the main diagonal, and the rest marks are on the main diagonal.
Let $M^{\uparrow}(\omega)$ (resp. $M^{\downarrow}(\omega)$) be 
the set of marks above (resp. below) the main diagonal.
Similarly, $M^{d}(\omega)$ is the set of marks on the main diagonal.

Suppose $\omega\in\mathcal{S}_{n}(321)$. 
We consider three Dyck paths from $(0,0)$ to $(n,n)$ in 
the Rothe diagram of $\omega$ as follows.

A Dyck path $p(\omega)$ is defined to be a path above the main diagonal, whose peaks are 
characterized by the set $M^{\uparrow}(\omega)\cup M^{d}(\omega)$.
Here, when $(x,y)$-th grid $g$ has a mark, the path $p(\omega)$ has an up step followed by 
a down step at the northwest corner of $g$.

We will introduce three Dyck paths $v(\omega), q(\omega)$, and $w(\omega)$ constructed 
from the Dyck path $p(\omega)$.
\begin{enumerate}
\item 
Let $p(\omega)$ be a Dyck path obtained as above.
A Dyck path $v(\omega)$ is obtained from $p(\omega)$ such that 
the peaks of $v(\omega)$ are characterized by the valleys of $p(\omega)$.
Suppose that the south and east edges of the $(x,y)$-th grid $g$ corresponds 
to a valley of $p(\omega)$. 
A peak of $v(\omega)$ is characterized by the grid $g$ as in the case of $p(\omega)$.
If necessary, we make use of the cells on the main diagonal to characterize 
the peaks of $v(\omega)$.

For example, if $p(\omega)=URU^3R^3UR$, then we have two valleys in $p(\omega)$.
These two valleys define the peaks of $v(\omega)$, and we need to include 
a cell on the main diagonal to characterize a peak of $v(\omega)$.
As a result, we obtain $v(\omega)=U^2R^2URU^2R^2$.

If there is no valleys in $p(\omega)$, i.e., $p(\omega)=U^{n}R^{n}$ as a Dyck path,
we define $v(\omega)=(UR)^{n}$ as a Dyck path.
\item
We construct a Dyck path $q'(\omega)$ below the main diagonal consisting of right steps and 
up steps. Note that $q'$ starts from a right step and ends with a up step.
The valley of $q'(\omega)$ is characterized by the set $M^{\downarrow}(\omega)$.
This means that if the $(x,y)$-th grid $g$ has a mark, we have an up step followed by a right step 
at the northwest corner of $g$.
We reflect $q'(\omega)$ along the main diagonal to obtain a Dyck path $q(\omega)$ in our convention.
\item
A Dyck path $w'(\omega)$ below the main diagonal is a path whose peaks are characterized by 
the set $M^{\downarrow}(\omega)\cup M^{d}(\omega)$. The path $w'(\omega)$ is constructed 
as in the case of $p(\omega)$. 
We reflect $w'(\omega)$ along the main diagonal to obtain a Dyck path $w(\omega)$.
\end{enumerate}

\begin{defn}
Suppose $\omega\in\mathcal{S}_{n}(321)$.
Let $p(\omega)$, $v(\omega)$, $q(\omega)$ and $w(\omega)$ be Dyck paths as above. 
We define three maps $E_{p}:\omega\mapsto p(\omega)$, $E_{v}:\omega\mapsto v(\omega)$, 
$E_{q}:\omega\mapsto q(\omega)$ and $E_{w}:\omega\mapsto w(\omega)$.
We define three maps 
\begin{align*}
\mathtt{Dyck}_{1}:=E_{v}\circ E_{p}^{-1}, \quad 
\mathtt{Dyck}_{2}:=E_{q}\circ E_{p}^{-1}, \quad
\mathtt{Dyck}_{3}:=E_{w}\circ E_{p}^{-1}.
\end{align*} 
\end{defn}

We give an example of the paths $p(\omega), v(\omega)$ and $q(\omega)$ for 
$\omega=13425\in\mathcal{S}_{5}(321)$ in Figure \ref{fig:DycktoDyck}.
The stars are the marks corresponding to the permutation $13425$.
The Dyck paths are $v(\omega)$, $p(\omega)$, $q'(\omega)$, and $w'(\omega)$
from left to right.

\begin{figure}[ht]
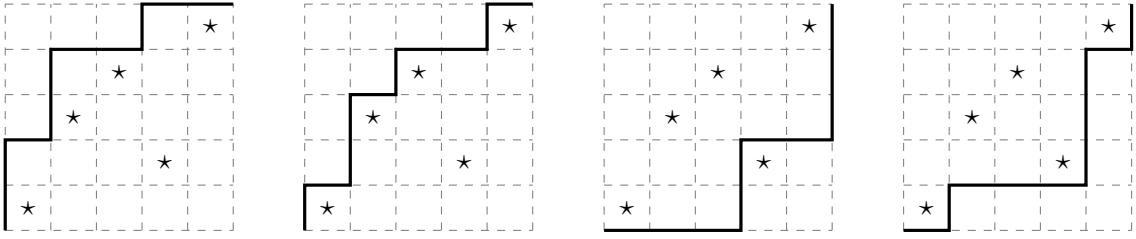

\tikzpic{-0.5}{[scale=0.6]
\foreach \a in {0,1,2,3,4,5}
\draw[gray,dashed](\a,0)--(\a,5)(0,\a)--(5,\a);
\draw(0.5,0.5)node{$\star$}(1.5,2.5)node{$\star$}(2.5,3.5)node{$\star$}
(3.5,1.5)node{$\star$}(4.5,4.5)node{$\star$};
\draw[very thick](0,0)--(0,2)--(1,2)--(1,4)--(3,4)--(3,5)--(5,5);
}
\quad
\tikzpic{-0.5}{[scale=0.6]
\foreach \a in {0,1,2,3,4,5}
\draw[gray,dashed](\a,0)--(\a,5)(0,\a)--(5,\a);
\draw(0.5,0.5)node{$\star$}(1.5,2.5)node{$\star$}(2.5,3.5)node{$\star$}
(3.5,1.5)node{$\star$}(4.5,4.5)node{$\star$};
\draw[very thick](0,0)--(0,1)--(1,1)--(1,3)--(2,3)--(2,4)--(4,4)--(4,5)--(5,5);
}
\quad
\tikzpic{-0.5}{[scale=0.6]
\foreach \a in {0,1,2,3,4,5}
\draw[gray,dashed](\a,0)--(\a,5)(0,\a)--(5,\a);
\draw(0.5,0.5)node{$\star$}(1.5,2.5)node{$\star$}(2.5,3.5)node{$\star$}
(3.5,1.5)node{$\star$}(4.5,4.5)node{$\star$};
\draw[very thick](0,0)--(3,0)--(3,2)--(5,2)--(5,5);
}
\quad
\tikzpic{-0.5}{[scale=0.6]
\foreach \a in {0,1,2,3,4,5}
\draw[gray,dashed](\a,0)--(\a,5)(0,\a)--(5,\a);
\draw(0.5,0.5)node{$\star$}(1.5,2.5)node{$\star$}(2.5,3.5)node{$\star$}
(3.5,1.5)node{$\star$}(4.5,4.5)node{$\star$};
\draw[very thick](0,0)--(1,0)--(1,1)--(4,1)--(4,4)--(5,4)--(5,5);
}
\caption{Three Dyck paths for the $321$-avoiding permutation $\omega=13425$.
The paths are $v(\omega)$, $p(\omega)$, $q'(\omega)$ and $w'(\omega)$ from left to right.}
\label{fig:DycktoDyck}
\end{figure} 

\begin{lemma}
The map $\mathtt{Dyck}_{3}$ is an involution.
\end{lemma}
\begin{proof}
Since $p(\omega)$ and $w'(\omega)$ has the same positions of marks, it is clear that 
$\mathtt{Dyck}_{3}$ is an involution.
\end{proof}

\begin{lemma}
\label{lemma:D2inv}
The map $\mathtt{Dyck}_2$ is an involution.
\end{lemma}
\begin{proof}
Suppose that $\omega=(n,1,2,\ldots,n-1)$. Then, by a simple calculation, we have 
$\mathtt{Dyck}_{2}(\omega)=(1,2,\ldots,n)$ and $\mathtt{Dyck}_2(1,2,\ldots,n)=\omega$.
The map $\mathtt{Dyck}_{2}$ is an involution on $\omega$.

Below, we assume that $\omega\neq(n,1,2,\ldots,n-1)$ is $321$-avoiding permutation.
We denote by $\mathrm{st}(\omega')$ 
the standardized sequence of $\omega'$.
We write $v\xleftarrow{p}q$ as an insertion of $q$ into $v$ at the position $p$ 
and its standardization. For example, $1234\xleftarrow{3}2=13245$.

Let $\omega=\alpha \omega'$ with $\alpha\in[1,n-1]$. 
By construction of $w'(\omega)$, we have 
\begin{align}
\label{eq:Dyck2w1}
\mathtt{Dyck}_{2}(\omega)=\mathtt{Dyck}_{2}(\mathrm{st}(\omega'))\xleftarrow{\alpha+1}1.
\end{align}
Similarly, if $\omega=\omega'1\omega''$ where the length of $\omega'$ is $\alpha-1$, then we have 
\begin{align}
\label{eq:Dyck2w2}
\mathtt{Dyck}_{2}(\omega)=\mathtt{Dyck}_{2}(\mathrm{st}(\omega'\omega''))\xleftarrow{1}\alpha-1.
\end{align}

We prove the statement by induction on the length of $\omega$.
From these, if $\omega=\alpha \omega'$ with $\alpha\in[1,n-1]$, we have 
\begin{align*}
\mathtt{Dyck}_{2}(\mathtt{Dyck}_{2}(\omega))
&=\mathtt{Dyck}_{2}(\mathtt{Dyck}_{2}(\mathrm{st}(\omega'))\xleftarrow{\alpha+1}1), \\
&=\mathtt{Dyck}_{2}(\mathtt{Dyck}_{2}(\mathrm{st}(\omega')))\xleftarrow{1}\alpha, \\
&=\mathrm{st}(\omega')\xleftarrow{1}\alpha, \\
&=\omega,
\end{align*}
where we have used Eqs. (\ref{eq:Dyck2w1}) and (\ref{eq:Dyck2w2}), and the induction hypothesis on $\omega'$.
This completes the proof.
\end{proof}

\begin{remark}
The $\mathtt{Dyck}_2$ is called Lalanne--Kreweras (LK) involution \cite{Eli12,Kre70,Lal92} on Dyck paths.
The LK involution is further studied in \cite{Eli25,HopJos22}.
\end{remark}

\subsection{RSK and \texorpdfstring{$321$}{321}-avoiding permutations}
We briefly review the correspondence between $321$-avoiding permutations and Dyck paths
via Robinson--Schensted--Knuth (RSK) correspondence following \cite{AdeEli23}.

The RSK correspondence is a bijection between a permutation and a pair of Young tableaux 
of the same shape. 
Let $\pi$ be a permutation. 
The action $\mathrm{RSK}(\pi)$ of RSK correspondence on $\pi$ is given by 
$\mathrm{RSK}(\pi):=(P,Q)$ where $P$ and $Q$ is Young tableaux of the same shape.
The tableau $P$ (resp. $Q$) is called the insertion (resp. recording) tableaux.
If and only if the tableaux $P$ and $Q$ have at most two rows, 
the permutation $\pi$ is $321$-avoiding.
Then, we introduce a bijection from a $321$-avoiding permutation to 
a Dyck path following \cite{AdeEli23,EliPak04}.

\begin{defn}[Definition 2 in \cite{AdeEli23}]
Let $\pi$ be a $321$-avoiding permutation and $\mathrm{RSK}(\pi)=(P,Q)$.
We define a Dyck path $\widehat{RSK}(\pi)$ as follows.
For $1\le i\le n$, let the $i$-th step be an up step if $i$ is in the first row of $P$, and 
a down step if $i$ is the second row of $P$.
Similarly, let the $(2n+1-i)$-th step be a down step if $i$ is in the first row of $Q$, 
and a up step if $i$ is in the second row of $Q$.
\end{defn}

\begin{example}
Consider the $321$-avoiding permutation $\pi=13524$.
The action of $\widehat{RSK}(\pi)$ gives the Dyck path whose step sequence 
is $(1,2,4,6,7)$:
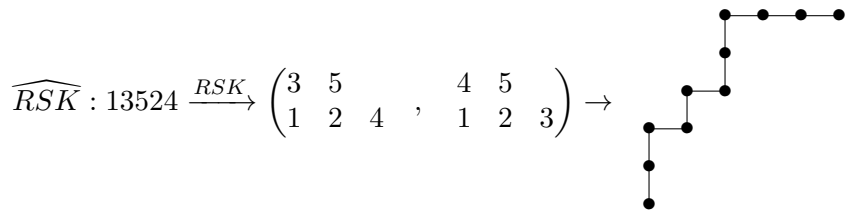
\begin{align*}
\widehat{RSK}: 
13524\xrightarrow{RSK}
\left(
\begin{matrix}
3 & 5 \\
1 & 2 & 4
\end{matrix}\quad,\quad
\begin{matrix}
4 & 5 \\
1 & 2 & 3
\end{matrix}
\right)
\rightarrow
\tikzpic{-0.5}{[scale=0.5]
\draw(0,0)node{$\bullet$}--(0,1)node{$\bullet$}--(0,2)node{$\bullet$}--(1,2)node{$\bullet$}
--(1,3)node{$\bullet$}--(2,3)node{$\bullet$}--(2,4)node{$\bullet$}--(2,5)node{$\bullet$}
--(3,5)node{$\bullet$}--(4,5)node{$\bullet$}--(5,5)node{$\bullet$};
}
\end{align*}
\end{example}

By construction of the map $\widehat{RSK}$, the inverse permutation $\pi^{-1}$
is also $321$-avoiding if $\pi$ is $321$-avoiding.
In general, from \cite[Theorem 2]{Sch61}, we have 
$\mathrm{RSK}(\pi^{-1})=(Q,P)$ if $\mathrm{RSK}(\pi)=(P,Q)$.

\section{Promotion and evacuation}
\label{sec:proeva}
\subsection{General theory}

We summarize the properties of promotion and evacuations on a poset
following \cite{Sta09}.

Let $P$ be a poset with $p$ elements.
We write $t\lessdot u$ if $u$ covers $t$.
We denote by $\mathcal{L}(P)$ the set of all linear 
extensions of $P$.
We consider a bijection $f:P\rightarrow[p]:=\{1,2,\ldots,p\}$ 
such that if $t<u$ in $P$, then $f(t)<f(u)$.
We define a bijection $\partial:\mathcal{L}(P)\rightarrow\mathcal{L}(P)$, 
called a {\it promotion} as follows.

Let $t_1\in P$ satisfy $f(t_1)=1$. 
Take an element $t_2\in P$ such that the label $f(t_2)$ is the least 
among the elements which cover $t_1$.
Delete this label $f(t_2)$ and place it at $t_1$.
Among the elements in $P$ which cover $t_2$, let $t_3\in P$ be the element
which has the least label $f(t_3)$.
Now we delete this label $f(t_3)$ from $t_3$ and place it at $t_2$.
We continue this process until we slide the label of a maximal element $t_{m}$ in $P$.
After we place $f(t_m)$ at $t_{m-1}$, label the element $t_m$ with $p+1$.
By construction, all elements in $P$ are newly labeled by a unique integer in $[2,p+1]$.
We decrease the labels of elements by one.
Then, we obtain a new linear extension $\partial f$.
Note that we let $\partial$ operate on the right.
 
We define the {\it dual promotion} $\partial^{\ast}$ as in the case of the promotion $\partial$.
Instead of taking an element with the minimal label, we take an element with the maximal label.
Let $u_1\in P$ satisfy $f(u_1)=p$. We fist delete the label of $u_1$.
Take an element $u_2$ with the maximal label such that $u_2$ is covered by $u_1$.
We slide the label $f(u_2)$ to $u_1$.
We slide up the labels one-by-one, and arrive at a minimal element $u_m$ in $P$. 
After we slide $f(u_m)$ to $u_{m-1}$, we label $u_m$ by $0$.
By construction, all elements in $P$ are newly labeled by a unique integer in $[0,p-1]$.
We increase the labels of the elements by one, and 
obtain a new linear extension $\partial^{\ast}f$.  

We next define two involutions on $\mathcal{L}(P)$, called {\it evacuation} and 
{\it dual evacuation}.
The evacuation is a variant of promotion.
We follow the same process as the promotion until we place the label $f(t_m)$ at 
$t_{m-1}$.
Now, we have no label on $t_{m}$. We put the integer $p$ on $t_{m}$ and we say that 
the element $t_m$ is frozen.
We continue to the same process as the promotion on the $p-1$ elements which are not 
frozen. We put the integer $p-1$ on the maximal element and freeze this element.
We continue this process until all the labels of the elements become frozen.
The new linear extension $\mathtt{ev}(f)$ is defined by the frozen labels.

The dual evacuation $\mathtt{ev}^{\ast}$ is defined similarly to the case of 
evacuation $\mathtt{ev}$.
Instead of taking the minimal element, we take a maximal element and apply 
$\partial^{\ast}$. We put the integer $1$ on a minimal element $u_{m}$, and 
freeze it.
As in the case of evacuation, we put an integer $1,2,\ldots,p$ one-by-one 
on minimal elements and freeze it.
Then, we define the new linear extension $\mathtt{ev}^{\ast}(f)$ as the 
frozen labels.

It is known \cite{Sta09} that the promotion and the evacuations can be 
defined as a composition of toggles.
{\it Toggles} form the group with generators $t_1,\ldots,t_{N-1}$ and 
have relations 
\begin{align*}
t_i^2=1, \quad 1\le i\le N-1, \\
t_it_j=t_jt_i, \text{ if } |i-j|>1.
\end{align*}
We regard the linear extension $f\in\mathcal{L}(P)$ as the word 
$f^{-1}(1),\ldots,f^{-1}(p)$. Then, the action of a toggle $t_{i}$, 
$1\le i\le n-1$, $t_{i}:\mathcal{L}(P)\rightarrow\mathcal{L}(P)$
is given by \cite{Sta09}: 
\begin{align}
t_{i}(u_1u_2\ldots u_{p})=
\begin{cases}
u_1u_2\ldots u_{p}, & \text{ if $u_i$ and $u_i+1$ are comparable in $P$}, \\
u_1u_2\ldots u_{i+1}u_i\ldots u_{p}, & \text{ otherwise}. 
\end{cases}
\end{align}

When a partially ordered set has $N$ elements, the promotion and the evacuations
can be expressed in terms of toggles.
\begin{prop}[\cite{Sta09}]
\label{prop:Pevtoggle}
The promotion and evacuations are expressed in terms of toggles as follows:
\begin{align*}
\partial&=t_{N-1}t_{N-2}\ldots t_{1}, \\
\partial^{\ast}&=t_{1}t_{2}\ldots t_{N-1}, \\
\mathtt{ev}&=t_1(t_2t_1)\ldots(t_{N-2}\ldots t_{2}t_{1})(t_{N-1}\ldots t_{2}t_{1}) \\
\mathtt{ev}^{\ast}&=t_{N-1}(t_{N-2}t_{N-1})\ldots(t_2\ldots t_{N-2}t_{N-1})(t_1\ldots t_{N-2}t_{N-1}).
\end{align*}
\end{prop}

In this paper, since we consider only rational Dyck paths, it is enough to give 
an action of a toggle on a rational Dyck path, equivalently on a standard 
$(a,b)$-Young tableau.
Before giving an explicit action of toggles on $(a,b)$-Young tableaux, 
we give examples of the actions of promotions and evacuations on $(a,b)$-Young tableaux.

\begin{example}
\label{ex:proeva}
An example of the promotions on the rational $(2,3)$-Dyck path of size $2$ 
are given by
\begin{align*}
\partial: \tikzpic{-0.5}{[scale=0.5]
\draw(0,0)--(12,0)--(12,-2)--(0,-2)--(0,0);
\draw(0,-1)--(12,-1);
\draw(3,0)--(3,-1)(6,0)--(6,-1)(9,0)--(9,-1)(12,0)--(12,-1);
\draw(2,-1)--(2,-2)(4,-1)--(4,-2)(6,-1)--(6,-2)(8,-1)--(8,-2)
(10,-1)--(10,-2);
\draw(1.5,-0.5)node{$1$}(4.5,-0.5)node{$2$}(7.5,-0.5)node{$5$}(10.5,-0.5)node{$7$};
\draw(1,-1.5)node{$3$}(3,-1.5)node{$4$}(5,-1.5)node{$6$}(7,-1.5)node{$8$}(9,-1.5)node{$9$}(11,-1.5)node{$10$};
}
\rightarrow
\tikzpic{-0.5}{[scale=0.5]
\draw(0,0)--(12,0)--(12,-2)--(0,-2)--(0,0);
\draw(0,-1)--(12,-1);
\draw(3,0)--(3,-1)(6,0)--(6,-1)(9,0)--(9,-1)(12,0)--(12,-1);
\draw(2,-1)--(2,-2)(4,-1)--(4,-2)(6,-1)--(6,-2)(8,-1)--(8,-2)
(10,-1)--(10,-2);
\draw(1.5,-0.5)node{$1$}(4.5,-0.5)node{$3$}(7.5,-0.5)node{$4$}(10.5,-0.5)node{$6$};
\draw(1,-1.5)node{$2$}(3,-1.5)node{$5$}(5,-1.5)node{$7$}(7,-1.5)node{$8$}(9,-1.5)node{$9$}(11,-1.5)node{$10$};
}
\end{align*}
and 
\begin{align*}
\partial^{\ast}: \tikzpic{-0.5}{[scale=0.5]
\draw(0,0)--(12,0)--(12,-2)--(0,-2)--(0,0);
\draw(0,-1)--(12,-1);
\draw(3,0)--(3,-1)(6,0)--(6,-1)(9,0)--(9,-1)(12,0)--(12,-1);
\draw(2,-1)--(2,-2)(4,-1)--(4,-2)(6,-1)--(6,-2)(8,-1)--(8,-2)
(10,-1)--(10,-2);
\draw(1.5,-0.5)node{$1$}(4.5,-0.5)node{$2$}(7.5,-0.5)node{$5$}(10.5,-0.5)node{$7$};
\draw(1,-1.5)node{$3$}(3,-1.5)node{$4$}(5,-1.5)node{$6$}(7,-1.5)node{$8$}(9,-1.5)node{$9$}(11,-1.5)node{$10$};
}
\rightarrow
\tikzpic{-0.5}{[scale=0.5]
\draw(0,0)--(12,0)--(12,-2)--(0,-2)--(0,0);
\draw(0,-1)--(12,-1);
\draw(3,0)--(3,-1)(6,0)--(6,-1)(9,0)--(9,-1)(12,0)--(12,-1);
\draw(2,-1)--(2,-2)(4,-1)--(4,-2)(6,-1)--(6,-2)(8,-1)--(8,-2)
(10,-1)--(10,-2);
\draw(1.5,-0.5)node{$1$}(4.5,-0.5)node{$3$}(7.5,-0.5)node{$6$}(10.5,-0.5)node{$8$};
\draw(1,-1.5)node{$2$}(3,-1.5)node{$4$}(5,-1.5)node{$5$}(7,-1.5)node{$7$}(9,-1.5)node{$9$}(11,-1.5)node{$10$};
}
\end{align*}
Similarly, the evacuations on this rational Dyck path are given by
\begin{align*}
\mathtt{ev}: \tikzpic{-0.5}{[scale=0.5]
\draw(0,0)--(12,0)--(12,-2)--(0,-2)--(0,0);
\draw(0,-1)--(12,-1);
\draw(3,0)--(3,-1)(6,0)--(6,-1)(9,0)--(9,-1)(12,0)--(12,-1);
\draw(2,-1)--(2,-2)(4,-1)--(4,-2)(6,-1)--(6,-2)(8,-1)--(8,-2)
(10,-1)--(10,-2);
\draw(1.5,-0.5)node{$1$}(4.5,-0.5)node{$2$}(7.5,-0.5)node{$5$}(10.5,-0.5)node{$7$};
\draw(1,-1.5)node{$3$}(3,-1.5)node{$4$}(5,-1.5)node{$6$}(7,-1.5)node{$8$}(9,-1.5)node{$9$}(11,-1.5)node{$10$};
}
\rightarrow
\tikzpic{-0.5}{[scale=0.5]
\draw(0,0)--(12,0)--(12,-2)--(0,-2)--(0,0);
\draw(0,-1)--(12,-1);
\draw(3,0)--(3,-1)(6,0)--(6,-1)(9,0)--(9,-1)(12,0)--(12,-1);
\draw(2,-1)--(2,-2)(4,-1)--(4,-2)(6,-1)--(6,-2)(8,-1)--(8,-2)
(10,-1)--(10,-2);
\draw(1.5,-0.5)node{$1$}(4.5,-0.5)node{$2$}(7.5,-0.5)node{$3$}(10.5,-0.5)node{$8$};
\draw(1,-1.5)node{$4$}(3,-1.5)node{$5$}(5,-1.5)node{$6$}(7,-1.5)node{$8$}(9,-1.5)node{$9$}(11,-1.5)node{$10$};
}
\end{align*}
and 
\begin{align*}
\mathtt{ev}^{\ast}: \tikzpic{-0.5}{[scale=0.5]
\draw(0,0)--(12,0)--(12,-2)--(0,-2)--(0,0);
\draw(0,-1)--(12,-1);
\draw(3,0)--(3,-1)(6,0)--(6,-1)(9,0)--(9,-1)(12,0)--(12,-1);
\draw(2,-1)--(2,-2)(4,-1)--(4,-2)(6,-1)--(6,-2)(8,-1)--(8,-2)
(10,-1)--(10,-2);
\draw(1.5,-0.5)node{$1$}(4.5,-0.5)node{$2$}(7.5,-0.5)node{$5$}(10.5,-0.5)node{$7$};
\draw(1,-1.5)node{$3$}(3,-1.5)node{$4$}(5,-1.5)node{$6$}(7,-1.5)node{$8$}(9,-1.5)node{$9$}(11,-1.5)node{$10$};
}
\rightarrow
\tikzpic{-0.5}{[scale=0.5]
\draw(0,0)--(12,0)--(12,-2)--(0,-2)--(0,0);
\draw(0,-1)--(12,-1);
\draw(3,0)--(3,-1)(6,0)--(6,-1)(9,0)--(9,-1)(12,0)--(12,-1);
\draw(2,-1)--(2,-2)(4,-1)--(4,-2)(6,-1)--(6,-2)(8,-1)--(8,-2)
(10,-1)--(10,-2);
\draw(1.5,-0.5)node{$1$}(4.5,-0.5)node{$2$}(7.5,-0.5)node{$4$}(10.5,-0.5)node{$7$};
\draw(1,-1.5)node{$3$}(3,-1.5)node{$5$}(5,-1.5)node{$6$}(7,-1.5)node{$8$}(9,-1.5)node{$9$}(11,-1.5)node{$10$};
}
\end{align*}
\end{example}

By use of the operation $\ast$, the dual promotion and dual evacuation are expressed 
in terms of promotion and evacuation.
The next lemma is a direct consequence of the operation $\ast$ and the definitions of 
promotion and evacuations.
\begin{lemma}
\label{lemma:dprodev}
We have 
\begin{align*}
&\partial^{\ast}=\ast\circ\partial\circ \ast, \\
&\mathtt{ev}^{\ast}=\ast\circ\mathtt{ev}\circ\ast.
\end{align*}
\end{lemma}

We denote by $Y(i,i+1)$ a tableau obtained from a tableau $Y$ by exchanging 
$i$ and $i+1$. Note that $Y(i,i+1)$ may not be in $\mathtt{SYT}_{(a,b)}(n)$ even 
if $Y$ is in $\mathtt{SYT}_{(a,b)}(n)$.
In our context, the action of a toggle $t_{i}$ on a standard $(a,b)$-Young tableau $Y$ is 
given by 
\begin{align*}
t_{i} Y=\begin{cases}
Y(i,i+1), & \text{ if } Y(i,i+1)\in\mathtt{SYT}_{(a,b)}(n), \\
Y, & \text{if } Y(i,i+1)\notin\mathtt{SYT}_{(a,b)}(n),
\end{cases}
\end{align*}
For example, consider the $(2,3)$-Dyck path of size $2$ whose 
step sequence is $1257$.
This $(2,3)$-Dyck path is invariant under the actions of $t_1,t_3,t_8$, and $t_9$.
The actions of $t_2,t_4,t_5,t_6$, and $t_7$ are $1357$, $1247$, $1267$, 
$1256$ and $1258$ respectively.

The next theorem directly follows from 
Proposition \ref{prop:Pevtoggle} and the defining relations of toggles. 
\begin{theorem}[\cite{Sta09}]
\label{thrm:relproev}
The bijections $\partial, \partial^{\ast}, \mathtt{ev}$ and $\mathtt{ev}^{\ast}$
on $\mathcal{L}(P)$ satisfy 
\begin{enumerate}
\item  $\partial^{\ast}=\partial^{-1}$.
\item $\mathtt{ev}^{2}=(\mathtt{ev}^{\ast})^2=\mathrm{id}$.
\item $\partial^{N}=\mathtt{ev}^{\ast}\circ\mathtt{ev}$.
\item $\mathtt{ev}\circ\partial=\partial^{-1}\circ\mathtt{ev}$.
\end{enumerate}
where $N$ is the number of elements in $\mathcal{L}(P)$.
\end{theorem}

\subsection{Promotion and evacuation on Dyck paths}
The actions of promotion and evacuation on $(1,1)$-Dyck paths of size $n$ can be rephrased 
in a simple way.
Let $\mathtt{Rot}$ be a $180/n$ degree counterclockwise rotation of a perfect matching in 
the circular presentation.
Recall that $\ast$ is an involution on Young tableaux.
\begin{prop}
\label{prop:DyckPMev}
Set $(a,b)=(1,1)$. Then, we have 
\begin{align}
\label{eq:PMRotDyck}
\mathtt{PM}\circ\partial&=\mathtt{Rot}\circ\mathtt{PM}, \\
\label{eq:evDyck}
\mathtt{ev}&=\mathtt{ev}^{\ast}=\ast.
\end{align} 
\end{prop}

Before proceeding to the proof of Proposition \ref{prop:DyckPMev}, 
we study the action of a toggle on a Dyck path and on a perfect matching.
Given a Dyck path $p$, we replace an up (resp. right) step in $p$ with 
a symbol $U$ (resp. $R$). Then, we have a word $w(p)$ of length $2n$ consisting
of two alphabets $U$ and $R$.  We call $w(p)$ a Dyck word. 
We have an obvious bijection between a Dyck path and a Dyck word.
\begin{lemma}
\label{lemma:togDyckp}
Let $p$ be a Dyck path and $w(p):=w_1\ldots w_{2n}$ be a Dyck word in $\{U,R\}^{2n}$.
Let $w_{i}(p)$ a word obtained from $w(p)$ by exchanging $w_i$ and $w_{i+1}$.
Then, we have 
\begin{align}
t_{i}\cdot w(p)=
\begin{cases}
w_{i}(p), & \text{ if $(w_{i},w_{i+1})=(U,R)$ or $(R,U)$ and $w_{i}(p)$ is a Dyck word}, \\
w(p), & \text{otherwise}.
\end{cases}
\end{align}
\end{lemma}
\begin{proof}
Suppose that a Dyck word $w(p)$ satisfies $w_{i}=w_{i+1}=U$ or $R$. Then, in the $(1,1)$-Young tableau
corresponding to $p$, the integers $i$ and $i+1$ are in the same row. 
By definition of the action of a toggle $t_{i}$, we have $t_{i}\cdot w(p)=w(p)$.
If $(w_{i},w_{i+1})=(U,R)$ or $(R,U)$, the integer $i$ is in the first row in the $(1,1)$-Young tableau
and the integer $i+1$ is in the second row.
If $i$ and $i+1$ is not in the same column, we have $t_{i}\cdot w(p)=w_{i}(p)$.
If $i$ and $i+1$ is in the same column, then $w_{i}=U$ and $w_{i+1}=R$.
In this case, the Dyck path $p$ touches the line $y=x$ at $(i-1,i-1)$ and $(i,i)$.
Then, $w_{i}(p)$ cannot be a Dyck path. Therefore, we have $t_{i}\cdot w(p)=w(p)$.
This completes the proof.
\end{proof}

Let $C$ be a circular representation of a perfect matching. If two integers $i<j$ in $C$ are connected by an arc,
we say that $i$ is an $L$-vertex, and $j$ is an $R$-vertex.
An arc $(i,j)$ is said to be an innermost arc if $j=i+1$.

\begin{lemma}
The action of a toggle $t_{i}$ on $C$ is given by
\begin{align}
t_i\cdot C=
\begin{cases}
C_{1}, & \text{ if $i$ is an $R$-vertex and $i+1$ is an $L$-vertex}, \\
C_{2}, & \text{ if $(i,i+1)$ is an inner-most arc and $i-1$ is an $L$-vertex}, \\
C, & \text{ otherwise}.
\end{cases}
\end{align}
The perfect matching $C_{1}$ is obtained from $C$ by reconnecting $(i,i+1)$  
and $(j',j'')$ by arcs if $C$ has arcs $(j',i)$ and $(i+1,j'')$.
The perfect matching $C_{2}$ is obtained from $C$ by reconnecting $(i-1,i)$ and 
$(i+1,j)$ by arcs if $C$ has arcs $(i-1,j)$ and $(i,i+1)$.
\end{lemma}
\begin{proof}
If $i$ is an $R$-vertex and $i+1$ is an $L$-vertex, the Dyck word $w$ satisfies 
$(w_{i},w_{i+1})=(R,U)$. By Lemma \ref{lemma:togDyckp}, it is easy to see that 
the action of $t_i$ is given by $t_i\cdot C=C_{1}$.

If $(i,i+1)$ is an innermost arc and $i-1$ is an $L$-vertex, the Dyck word $w$
satisfies $(w_{i},w_{i+1})=(U,R)$, and $w_{i-1}=U$.
By Lemma \ref{lemma:togDyckp}, a new Dyck word is obtained by exchanging 
$w_{i}$ and $w_{i+1}$. If we translate these in terms of a perfect matching,
the action of $t_i$ on $C$ is given by $C_2$.

In other cases, the action of $t_i$ is trivial by Lemma \ref{lemma:togDyckp}.
This completes the proof.
\end{proof}

\begin{proof}[Proof of Proposition \ref{prop:DyckPMev}]
We prove the proposition by induction on the size $n$.
When $n=1$, we have a unique perfect matching and the claim holds.
When $n=2$, we have two perfect matchings and it is easy to verify the claim.
We assume $n\ge3$.

We fist show Eq. (\ref{eq:PMRotDyck}).
Let $p$ be a Dyck path of size $n$.
Recall that the promotion $\partial$ can be expressed in terms of toggles as in Proposition \ref{prop:Pevtoggle}.
We consider the action of toggles on $p$.
First, we consider the case where $p$ is a prime Dyck path.
Let $\mathbf{h}(p):=(h_1,\ldots,h_{n})$ be a sequence of integers such that 
$h_{i}$ is the height of $i$-th right step in $p$.
Since $p$ is prime, $h_{i}>i$ for $1\le i\le n-1$.
Since $p$ is prime, Lemma \ref{lemma:togDyckp} implies that 
$\partial\cdot p$ is a path $p'$ such that 
$\mathbf{h}(p')=(h_1-1,h_2-1,\ldots,h_{n-1}-1,h_{n})$.
In terms of a perfect matching, the arc $(i,j)$, $i\ge2$ in $p$ corresponds 
to $(i-1,j-1)$ in $p'$, and the arc $(1,n)$ corresponds to $(n-1,n)$. 
Therefore, we have Eq. (\ref{eq:PMRotDyck}).

Suppose that $p$ is not prime. Then, we have an arc $(1,a)$ with $a\ge2$ 
in the perfect matching of $p$. 
The path $p_1\ldots p_{a}$ is prime.
The action of toggles $t_{a-1}t_{a-2}\ldots t_{1}$ on $p$ gives a new path $p'$
such that an arc $(i,j)$ with $i,j\in[2,a-1]$ in $p$ is replaced with an arc 
$(i-1,j-1)$ in $p'$, and an arc $(1,a)$ in $p'$ is replaced with an arc $(a-1,a)$ in $p'$.
By applying the toggles $t_{2n-1}t_{2n-2}\ldots t_{a}$ to $p'$ and repeating the above argument, 
we have a new path $p''$ such that an arc $(i,j)$ with $i,j\in[1,2n]\setminus\{1,a\}$ in $p$
corresponds to an arc $(i-1,j-1)$ in $p''$, and an arc $(1,a)$ in $p$ corresponds 
to $(a-1,2n)$ in $p''$. This implies that we have Eq. (\ref{eq:PMRotDyck}).

Secondly, we show Eq. (\ref{eq:evDyck}).
Recall that evacuation is a variant of promotion. 
We delete the integer $1$ from the tableau, move integers and add an integer to fill
a blank cell. We apply this algorithm to a two-row Young tableau.
We first delete the minimal label from the two-row Young tableau, move integers and add $n$ at the 
bottom right cell. We freeze the integer $n$.
We delete integers one-by-one until we obtain a tableau such that all the labels of cells
are frozen.
Suppose $i$ is an $L$-vertex. If we delete $i$ from the tableau, we have an integer $j$
which is an $R$-vertex of the pair $(i,j)$.
The integer $j$ is in the second row of the tableau. 
We move the integers larger than $i$ to fill the blank cell. Since $(i,j)$ is a pair 
of an $L$-vertex and $R$-vertex, we move, at some point, the integer $j$ to the 
first row of the tableau. This means that the newly frozen integer is in the second 
row.
Suppose that $i$ is an $R$-vertex. If we delete $i$ from the the diagram, 
the newly frozen integer is in the first row of the tableau. This is because we have 
already removed the $L$-vertex for $i$, and other labels larger than $i$ form arcs.
From these observations, if $i$ is in the first (resp. second) row in the tableau, 
then evacuation is to put the label $2n+1-i$ in the second (resp. first) row.
This simply means that $\mathtt{ev}=\ast$.
We have $\mathtt{ev}^{\ast}=\ast\circ\mathtt{ev}\circ\ast=\ast$.
This completes the proof.
\end{proof}

\subsection{Promotion and evacuation on rational Dyck paths}
We first consider the promotion in the case of $(a,b)=(1,k)$ with $k\ge2$.
As already mentioned, the size of a matching block of a perfect matching of a $k$-Dyck path 
is always $k+1$. 
As in the case of Dyck paths, we have 
\begin{align}
\label{eq:delRot}
\mathtt{PM}\circ\partial&=\mathtt{Rot}\circ\mathtt{PM}, 
\end{align}
where $\mathtt{Rot}$ is a $360/((k+1)n)$ degree counterclockwise rotation of a perfect matching. 

\begin{remark}
For a general $(a,b)$ with $a\ge2$, the relation Eq. (\ref{eq:delRot}) does not hold.
This is because the sizes of a block of a perfect matching of an $(a,b)$-Dyck path depend 
on the path.
For example, consider two $(2,3)$-Dyck paths of size $1$. The perfect matchings are 
\begin{align*}
\mathcal{S}_1:=\{\{1,4,5\},\{2,3\}\}, \quad \mathcal{S}_2:=\{\{1,2,5\},\{3,4\}\}.
\end{align*} 
The perfect matching $\mathcal{S}_1$ is obtained from $\mathcal{S}_2$ by $\mathtt{Rot}$.
Similarly, we have $\mathcal{S}_2=\partial(\mathcal{S}_{1})$, but 
$\mathcal{S}_{2}\neq\mathtt{Rot}(\mathcal{S}_1)$.
This phenomenon occurs since the size of arcs in a perfect matching depends on arcs.
\end{remark}

We consider evacuation on an $(a,b)$-Dyck path.
We first rephrase the relation $\mathtt{ev}=\mathtt{ev}^{\ast}=\ast$ in Proposition \ref{prop:DyckPMev}. 
Let $C(p)$ be a circular presentation of a perfect matching of a Dyck path $p$.
We denote by $\overline{C(p)}$ a perfect matching obtained from $C(p)$ by replacing 
$i$ by $2n+1-i$. We call this operation a bar operation.
Then, Proposition \ref{prop:DyckPMev} implies that the following diagram is commutative:
\begin{align}
\label{eq:evPM}
\tikzpic{-0.5}{
\node (0) at (0,0){$p$};
\node (1) at (3,0){$\mathtt{ev}(p)$};
\node (2) at (0,-2){$C(p)$};
\node (3) at (3,-2){$\overline{C(p)}$};
\draw[->,anchor=south] (0) to node {$\mathtt{ev}$} (1);
\draw[->,anchor=south] (2) to node {$bar$} (3);
\draw[->,anchor=east] (0) to node {$\mathtt{PM}$}(2);
\draw[->,anchor=west] (1) to node {$\mathtt{PM}$}(3);
}
\end{align}
We will show that the diagram (\ref{eq:evPM}) also holds for 
a general $(a,b)$-Dyck path.
More precisely, let $C(p)$ be a perfect matching of an $(a,b)$-Dyck 
path and define a bar operation $C(p)\mapsto\overline{C(p)}$ by 
replacing $i$ by $(a+b)n+1-i$ in $C(p)$.
\begin{prop}
\label{prop:abDyckev}
For an $(a,b)$-Dyck path $p$, the diagram (\ref{eq:evPM}) is commutative.
\end{prop}
\begin{proof}
Let $P$ be a perfect matching.
We define an action of a toggle $\tilde{t}_{i}$ on $P$ by 
\begin{align}
\tilde{t}_{i}\cdot P=
\begin{cases}
P', & \text{ if $i$ and $i+1$ are not in the same block}, \\
P, & \text{otherwise},
\end{cases}
\end{align}
where $P'$ is a perfect matching obtained from $P$ by exchanging $i$ and $i+1$.
We define $\widetilde{\mathtt{ev}}$ by 
\begin{align}
\widetilde{\mathtt{ev}}:=\tilde{t}_1(\tilde{t}_2\tilde{t}_{1})\ldots (\tilde{t}_{N-1}\ldots \tilde{t}_{2}\tilde{t}_1).
\end{align}
To prove the proposition, it is enough to show that 
$\widetilde{\mathtt{ev}}\cdot P=\overline{P}$.

The action of $\tilde{t}_{r-1}\tilde{t}_{r-2}\ldots \tilde{t}_1$ on $P$ gives a new 
matching $P'$ by replacing $1$ with $r$, $i$ with $i-1$ for $2\le i\le r$.
From this, $\widetilde{\mathtt{ev}}\cdot P$ is a perfect matching obtained from $P$
by replacing $i$ with $N+1-i$. This is nothing but $\overline{P}$.
Therefore, the diagram (\ref{eq:evPM}) is commutative.
\end{proof}

\begin{example}
Let $p$ be a $(2,3)$-Dyck path of size $2$ whose step sequence is $1257$.
Then, the perfect matching of $p$ and the action of the bar operation on it 
are given by 
\begin{align*}
\{\{1,4,10\},\{2,3\},\{5,6,9\},\{7,8\}\}
\xrightarrow{bar}
\{\{1,7,10\},\{2,5,6\},\{3,4\},\{8,9\}\}.
\end{align*}
From proposition \ref{prop:abDyckev},
the evacuation of the path $p$ is given by the path whose step sequence 
is $1238$. Note that this computation is compatible with the example just 
below Proposition \ref{prop:Pevtoggle}.
\end{example}

We study the evacuation on an $(a,b)$-Dyck path $P$ of size $n$.
Let $S(P)$ be the perfect matching of $P$ given by Definition \ref{defn:PM}.
The set $S(P)$ consists of $an$ sets of positive integers. We denote 
by $S_{i}(P)$ such sets for $1\le i\le an$.
We define the set of $an$ positive integers by
\begin{align}
S^{\mathtt{ev}}(P):=\{(a+b)n+1-\max(S_{i}(P)) |  1\le i\le an \}.
\end{align}
We rearrange the elements of $S^{\mathtt{ev}}(P)$ in ascending order, and 
denote it $S^{\mathtt{ev}}_{<}(P)$.
Then, the following proposition is a direct consequence of the definition 
of the evacuation and that of a perfect matching.
\begin{prop}
\label{prop:evP}
Let $S^{\mathtt{ev}}_{<}(P)$ be as above.
The step sequence of the $(a,b)$-Dyck path $\mathtt{ev}(P)$ is given 
by $S^{\mathtt{ev}}_{<}(P)$.
\end{prop}
\begin{example}
We consider the $(2,3)$-Dyck path $P$ in Figure \ref{fig:PM}. 
Since the maximal elements of the perfect matching is $\{10,3,9,8\}$, we have 
$S^{\mathtt{ev}}_{<}(P)=(1,2,3,8)$.
Compare the evacuation of $P$ in Example \ref{ex:proeva}. 
\end{example}

To rephrase the dual evacuation on an $(a,b)$-Dyck path in terms of a perfect 
matching, we consider the map $\mathtt{dPM}$.
Given a path $p\in\mathtt{Dyck}_{(a,b)}(n)$, we denote 
$\mathtt{dPM}(p)=\{S_{1},\ldots, S_{bn}\}$ where 
$S_{i}$, $1\le i\le bn$, is a matching block of the perfect matching.
Let $s^{\min}_{i}:=\overline{\min(S_{i})}$ for $1\le i\le bn$ where 
$\overline{h}$ is the bar operation of $h\in[1,(a+b)n]$.
We define $S^{\min}(p):=\{s^{\min}_{i}| 1\le i\le bn\}$.

\begin{prop}
\label{prop:dev}
Let $p\in\mathtt{Dyck}_{(a,b)}(n)$. 
The step sequence $\mathbf{u}(\mathtt{ev}^{\ast}(p))$ of the path $\mathtt{ev}^{\ast}(p)$ is given by 
\begin{align*}
\mathbf{u}(\mathtt{ev}^{\ast}(p))=[1,(a+b)n]\setminus S^{\min}(p).
\end{align*}
\end{prop}
\begin{proof}
Recall that the operation $\ast$ maps an $(a,b)$-Dyck path to a $(b,a)$-Dyck path.

Since $\mathtt{dPM}=\overline{\mathtt{PM}\circ\ast}$, $\mathtt{ev}^{\ast}=\ast\circ\mathtt{ev}\circ\ast$,
and $\mathtt{PM}\circ\mathtt{ev}=\overline{\mathtt{PM}}$ by Proposition \ref{prop:abDyckev},
we have 
\begin{align*}
\mathtt{ev}^{\ast}&=\ast\circ\mathtt{ev}\circ\ast, \\
&=\ast\circ\mathtt{PM}^{-1}\circ\mathtt{PM}\circ\mathtt{ev}\circ\ast, \\
&=\ast\circ\mathtt{PM}^{-1}\circ\mathtt{dPM}.
\end{align*} 
Since we have $\mathtt{dPM}(p)=\{S_1,\ldots,S_{bn}\}$, the integers in $S_{i}\setminus\{\min(S_{i})\}$
is a down path of the $(b,a)$-Dyck path. 
Then, it is clear that $[1,(a+b)n]\setminus S^{\min}(p)$ gives the step sequence of $\mathtt{ev}^{\ast}(p)$,
which completes the proof.
\end{proof}

Proposition \ref{prop:dev} is summarized as the following commutative diagram:
\begin{align}
\tikzpic{-0.5}{
\node (0) at (0,0){$p$};
\node (1) at (3,0){$\mathtt{dPM}(p)$};
\node (2) at (1.5,-2){$\mathtt{ev}^{\ast}(p)$};
\draw[->,anchor=south] (0) to node {$\mathtt{dPM}$} (1);
\draw[->,anchor=east] (0) to node {$\mathtt{ev}^{\ast}$}(2);
\draw[->,anchor=west] (1) to node {$\ast\circ\mathtt{PM}^{-1}$}(2);
}
\end{align}

\begin{example}
Let $\mathtt{Dyck}_{(2,3)}(2)\ni p$ be a path with the step sequence $1257$.
From Figure \ref{fig:dPM}, the perfect matching $\mathtt{dPM}(p)$ is given by
$\{\{1,10\},\{2,4\},\{3\},\{5,7,9\},\{6\},\{8\}\}$.
From Proposition \ref{prop:dev}, we have $S^{\min}(p)=\{3,5,6,8,9,10\}$ and 
the step sequence $\mathbf{u}(\mathtt{ev}^{\ast}(p))=\{1,2,4,7\}$.
Note that this coincides with an example of the dual evacuation just above 
Lemma \ref{lemma:dprodev}. 
\end{example}

\section{Rowmotion and rowvacuation}
\label{sec:row}
\subsection{General theory}
\label{sec:GTrow}
We briefly review a general theory of rowmotion and rowvacuation following \cite{CamFon95,StrWil12}.

Let $\mathcal{P}$ be a poset of rank $r$.
A poset $\mathcal{P}$ has a rank function 
$\mathtt{rk}:\mathcal{P}\rightarrow\mathbb{Z}_{\ge0}$ which satisfies
\begin{enumerate}
\item $\mathtt{rk}(x)=0$ for all minimum elements $x$; 
\item $\mathtt{rk}(y)=\mathtt{rk}(x)+1$ if $x\lessdot y$;
\item $\mathtt{rk}(x)=r$ for all maximal elements $x$.
\end{enumerate}
For $0\le i\le r$, we define the $i$-th rank of $\mathcal{P}$ 
as $\mathcal{P}_{i}:=\{x\in\mathcal{P}:\mathtt{rk}(x)=i\}$.

We study operations called rowmotion and rowvacuation on a poset 
which have similar properties to the promotion and the evacuation.

An {\it order filter} of $\mathcal{P}$ is a subset $F\subseteq\mathcal{P}$
such that if $x\in F$ and $y\ge x$ in $\mathcal{P}$, then $y\in F$.
We denote by $\mathcal{F}(\mathcal{P})$ the set of order filters of $\mathcal{P}$.
An {\it order ideal} of $\mathcal{P}$ is a subset $I\subseteq\mathcal{P}$ 
such that if $x\in I$ and $y\le x$ in $\mathcal{P}$, then $y\in I$.
We denote by $\mathcal{I}(\mathcal{P})$ the set of order ideals of $\mathcal{P}$.
An {\it antichain} of $\mathcal{P}$ is a subset $A\subseteq\mathcal{P}$ such that 
any two elements in $A$ are incomparable. We denote by $\mathcal{A}(\mathcal{P})$ 
the set of antichains of $\mathcal{P}$.

To define rowmotions, which is defined on $\mathcal{F}(\mathcal{P})$ or equivalently 
on $\mathcal{A}(\mathcal{P})$, 
we introduce three bijections as follows.

\begin{enumerate}
\item
{\it Complementation} $\Theta:2^{\mathcal{P}}\rightarrow2^{\mathcal{P}}$, 
$S\mapsto P\setminus S$.
The bijection $\Theta$ maps order ideals to filter ideals and vice versa.
\item
{\it Up-transfer} $\Delta:\mathcal{I}(\mathcal{P})\rightarrow\mathcal{A}(\mathcal{P})$,
where $\Delta(I)$ denotes the set of maximal elements of $I$.
\item
{\it Down-transfer} $\nabla:\mathcal{F}(\mathcal{P})\rightarrow\mathcal{A}(\mathcal{P})$,
where $\nabla(F)$ denotes the set of minimal elements of $F$. 
\end{enumerate}

The inverses of $\Theta,\Delta$ and $\nabla$ are given by 
$\Theta^{-1}=\Theta$, 
\begin{align*}
\Delta^{-1}(A)&=\{x\in\mathcal{P}: x\le y \text{ for some } y\in A\},\\
\nabla^{-1}(A)&=\{x\in\mathcal{P}: x\ge y \text{ for some } y\in A\},
\end{align*}
where $A\in\mathcal{A}(\mathcal{P})$ is an antichain.

\begin{defn}
We define two rowmotions:
\begin{enumerate}
\item {\it Order filter rowmotion} 
$\delta:=\mathtt{Row}_{\mathcal{F}}:\mathcal{F}(\mathcal{P})\rightarrow\mathcal{F}(\mathcal{P})$
is defined by $\mathtt{Row}_{\mathcal{F}}:=\Theta\circ\Delta^{-1}\circ\nabla$.
\item {\it Antichain rowmotion} 
$\mathtt{Row}_{\mathcal{A}}:\mathcal{A}(\mathcal{P})\rightarrow\mathcal{A}(\mathcal{P})$ 
is defined by $\mathtt{Row}_{\mathcal{A}}:=\nabla\circ\Theta\circ\Delta^{-1}$.
\end{enumerate}
\end{defn}
The two rowmotions are conjugated to one another by $\nabla$.
In this paper, we mainly use the rowmotion $\delta=\mathtt{Row}_{\mathcal{F}}$.

As in the case of the promotion and the evacuation studied in the previous section, 
the rowmotion and rowvacuations are expressed in terms of simple involutions called toggles.

\begin{defn}
Let $p\in\mathcal{P}$. The {\it order filter toggle} at $p$, denoted 
by $t_{p}:\mathcal{F}(\mathcal{P})\rightarrow\mathcal{F}(\mathcal{P})$, 
is defined by 
\begin{align*}
t_{p}(F):=
\begin{cases}
F\cup\{p\}, & \text{ if } p\notin F \text{ and } F\cup\{p\}\in\mathcal{F}(\mathcal{P}), \\
F\setminus\{p\}, & \text{ if } p\in F \text{ and } F\setminus\{p\}\in\mathcal{F}(\mathcal{P}), \\
F, & otherwise.
\end{cases}
\end{align*}
\end{defn}
For $0\le i\le r$, we define 
\begin{align*}
\mathbf{t}_{i}:=\prod_{p\in\mathcal{P}_{i}} t_{p}.
\end{align*}
A toggle $\mathbf{t}_i$ is called the order filter rank toggle.
By definition of toggles, $t_{p}t_q=t_qt_p$ if $p$ and $q$ have the same rank.
It is easy to see that 
\begin{align*}
&\mathbf{t}_{i}^2=\mathrm{id}, \\
&\mathbf{t}_{i}\mathbf{t}_{j}=\mathbf{t}_{j}\mathbf{t}_{i}, \quad |i-j|>1.
\end{align*}

\begin{lemma}
The rowmotion $\delta$ is expressed in terms of order filter toggles:
\begin{gather*}
\delta=\mathbf{t}_0\mathbf{t}_1\ldots \mathbf{t}_{r}.
\end{gather*}
\end{lemma}
\begin{proof}
Let $F\in\mathcal{F}(\mathcal{P})$ be an order filter. 
Suppose that an element $x\in \nabla(F)$ and $\mathtt{rk}(x)=m$. 
Note that if $y\lessdot x$, then $y\notin F$ and $\mathtt{rk}(y)=m-1$.
By the definition of the order filter rowmotion, we have 
$\mathbf{t}_{m+1}\ldots\mathbf{t}_{r}(F)=F$, and 
$\mathbf{t}_{m}(F)=F\setminus \{x\}$.
Further $\mathbf{t}_{0}\ldots\mathbf{t}_{m-1}(F\setminus\{x\})=F\setminus\{x\}$.
On the other hand, the action of $\delta=\Theta\circ\Delta^{-1}\circ\nabla$ on $F$
is given by $F\setminus\nabla(F)$.
From these observation, we have $\delta=\mathbf{t}_0\mathbf{t}_1\ldots \mathbf{t}_{r}$.
\end{proof}

\begin{defn}
We define rowvacuation $\mathtt{Rvac}$ and dual rovacuation $\mathtt{DRvac}$ by  
\begin{gather*}
\mathtt{Rvac}:=(\mathbf{t}_r)(\mathbf{t}_{r-1}\mathbf{t}_{r})\ldots(\mathbf{t}_1\ldots \mathbf{t}_{r})
(\mathbf{t}_0\mathbf{t}_1\ldots \mathbf{t}_{r}), \\
\mathtt{DRvac}:=(\mathbf{t}_0)(\mathbf{t}_{1}\mathbf{t}_{0})\ldots(\mathbf{t}_{r-1}\ldots \mathbf{t}_{0})
(\mathbf{t}_r\mathbf{t}_{r-1}\ldots \mathbf{t}_{0}).
\end{gather*}
\end{defn}

As in the case of promotion and evacuation, the next proposition shows the 
relations among rowmotion and rowvacuations.
One can prove the proposition by using the expressions of $\mathtt{Rvac}$ 
and $\mathtt{DRvac}$ in terms of the order filter rank toggles.
\begin{prop}
\label{prop:Rvac}
Let $\mathcal{P}$ be a graded post of rank $r$.
Then, we have 
\begin{enumerate}
\item $\mathtt{Rvac}^2=\mathtt{DRvac}^2=\mathrm{id}$,
\item $\mathtt{Rvac}\circ\delta=\delta^{-1}\circ\mathtt{Rvac}$,
\item $\mathtt{DRvac}\circ\delta=\delta^{-1}\circ\mathtt{DRvac}$,
\item $\delta^{r+2}=\mathtt{DRvac}\circ\mathtt{Rvac}$.
\end{enumerate}
\end{prop}

\subsection{Rowmotion and Rowvacuation on Dyck paths}
Recall that a Dyck path is a lattice path from $(0,0)$ to $(n,n)$ consisting 
of up and right steps. We write an up (resp. right) step as $U$ (resp. $R$).
There are several unit boxes above $(UR)^{n}$ and below the top path $U^{n}R^{n}$, and 
we denote by $\mathcal{B}(n)$ the set of such unit boxes.
The centers of these boxes in $\mathcal{B}(n)$ lie on the line $y=x+r+1$, $0\le r\le n-2$.
We define the rank function $\mathtt{rk}:\mathcal{B}(n)\rightarrow \mathbb{Z}_{\ge0}$,
$b\mapsto r(b)$ if $b$ is on the line $y=x+r(b)+1$.
Suppose that $b_1,b_2\in\mathcal{B}(n)$. 
The box $b_{2}$ covers $b_{1}$ if $\mathtt{rk}(b_{2})=\mathtt{rk}(b_1)+1$ and 
the two boxes $b_2$ and $b_1$ share an edge.

Given a Dyck path $p$ above $(UR)^{n}$ and below $U^nR^n$, 
the order ideal $I(p)$ is given by the set of the boxes in $\mathcal{B}(n)$ such that 
they are below $p$.
Similarly, the order filter $F(p)$ is given by the set of boxes in $\mathcal{B}(n)$ 
such that they are above $p$.
In this way, we can view the set of boxes $\mathcal{B}(n)$ as a graded poset of rank 
$n-2$.

By applying the general theory in Section \ref{sec:GTrow} to the poset $\mathcal{B}(n)$,
one can consider the action of rowmotion and rowvacuations on Dyck paths.

The next propositions connect the maps $\mathtt{Dyck}_{1}$ and $\mathtt{Dyck}_{2}$
with the rowmotion and rowvacuation on a Dyck path.
\begin{prop}
\label{prop:rmD1}
The rowmotion $\delta$ on a Dyck path is given by $\mathtt{Dyck}_{1}$ on a Dyck path,
{\it i.e.}, $\delta=\mathtt{Dyck}_1$.
\end{prop}
\begin{proof}
Let $p$ be a Dyck path. We define the order filter $F(p)$ as the unit boxes 
above $p$ in the Rothe diagram.
Then, the antichain $A(p)$ is given by the set of unit boxes which are just above 
the valleys of $p$. 
The rowmotion $\delta$ acts on $p$ as $\delta(p)=F(p)\setminus A(p)$.
It is clear from the definition of the map $\mathtt{Dyck}_1$ that $\delta=\mathtt{Dyck}_1$.
\end{proof}

The next lemma connects the rowmotion and the evacuation on the Dyck paths of size $n$.
\begin{lemma}
\label{lemma:rmev}
We have $\delta^{n}=\mathtt{ev}$.
\end{lemma}
The proof of Lemma \ref{lemma:rmev} will be given 
in Section \ref{sec:MatRSK} as Corollary \ref{cor:rmev}.

\begin{prop}
\label{prop:rv}
The rowvacuation $\mathtt{Rvac}$ on a Dyck path is given by $\mathtt{Dyck}_2$ on a Dyck path,
{\it i.e.,} $\mathtt{Rvac}=\mathtt{Dyck}_2$.
Similarly, the dual rowvacuation $\mathtt{DRvac}$ on a Dyck path is given by 
$\mathtt{DRvac}=\mathtt{ev}\circ\mathtt{Dyck}_2$.
\end{prop}
\begin{proof}
We first show $\mathtt{Rvac}=\mathtt{Dyck}_2$.
Let $P$ be a Dyck path and $R(P)$ be the Rothe diagram for $P$.
We denote by $\mathrm{Val}(P)$ the set of cells $c$ in $R(P)$ such that 
the south and east edges of $c$ are a valley of $P$.
We consider the action of $\delta$ on $P$.
The $321$-avoiding permutation $w$ for $\delta(P)$ can be obtained by the 
following procedures on the Rothe diagram $R(P)$. 
Recall that the diagram $R(P)$ has $n$ marks corresponding to the $321$-avoiding
permutation for $P$.
We enumerate 
the columns in $R(P)$ by $1,2,\ldots,n$ from left to right.
\begin{enumerate}
\item Set $k=n$, and $D=R(P)$.
\item 
\begin{enumerate}
\item
If there is no cell $c$ in $\mathrm{Val}(P)$ such that $c$ is in the $k$-th column in $D$.
Then, go to (3).
\item
Suppose that the $k$-th column in $D$ has a mark at $c_1$, and a cell $c$ in $\mathrm{Val}(P)$.
We delete the mark in $c_k$, and put a mark on $c$. We freeze the mark on $c$.
Let $S(c_k)$ be the set of unfrozen marks in $D$ which are on the cells above and right to $c_k$.
We change the positions of the unfrozen marks in $S(c_k)$ such that they give a $321$-avoiding permutation.
\end{enumerate}
\item We decrease $k$ by one, and replace $D$ by a new diagram with marks. 
If $k=1$, then go to (4). Otherwise, go to (2).
\item 
\begin{enumerate}
\item
If the first column of $D$ has a cell $c$ in $\mathrm{Val}(P)$.
We apply (2b) to the first column and we obtain a $321$-avoiding permutation $w$.
\item The first column of $D$ does not have a cell $c$ in $\mathrm{Val}(P)$, and  
has a mark at the cell $c_{1}$.
Then, we move the mark from $c_1$ to the bottom cell in the first column.
We change the positions of unfrozen marks in $D$ which are in the cells below and right to $c_1$ 
such that they give a $321$-avoiding permutation $w$.
\end{enumerate}
\end{enumerate}
By construction, it is easy to see that the procedures above give the $321$-avoiding 
permutation for $\delta(P)$.

We first show that $\mathtt{Dyck}_2\circ\delta=\delta^{-1}\circ\mathtt{Dyck}_2$.
We use the notation $w\xleftarrow{p}q$ as in the proof of Lemma \ref{lemma:D2inv}.
We prove the claim by induction on the size of $w$. If the size of $w$ is two, 
it is straightforward to show the claim. We assume that the size of $w$ is larger than two.
We write $w$ as $w=w'\xleftarrow{k}1$.
We have three cases: 1) $k\ge3$, 2) $k=2$, and 3) $k=1$.

Case 1). We have $k\ge3$ and $w=w'\xleftarrow{k}1$.
Let $w_1,\ldots,w_{k-1}$ be the first $k-1$ integers in $w$. These $k-1$ integers 
are larger than $1$ and  increasing 
since $w=w'\xleftarrow{k}1$ and $w$ is a $321$-avoiding permutation.
Recall that $\delta$ is equivalent to $\mathtt{Dyck}_1$ by Proposition \ref{prop:rmD1}.
From this, we have $\delta(w)=\delta(w')\xleftarrow{k-1}1$.
By the definition of $\mathtt{Dyck}_2$, we have 
$\mathtt{Dyck}_2(\delta(w))=\mathtt{Dyck}_2(\delta(w'))\xleftarrow{1}k-2$.
On the other hand, we have $\mathtt{Dyck}_2(w)=\mathtt{Dyck}_2(w')\xleftarrow{1}k-1$.
From $\delta=\mathtt{Dyck}_1$, we have 
$\delta^{-1}(\mathtt{Dyck}_2(w))=\delta^{-1}(\mathtt{Dyck}_2(w'))\xleftarrow{1}k-2$.
From these and induction hypothesis, we have $\mathtt{Dyck}_2\circ\delta=\delta^{-1}\circ\mathtt{Dyck}_2$ if 
$k\ge3$.

Case 2). We have $w=w'\xleftarrow{2}1$. Then, we have $\mathtt{Dyck}_2(w)=\mathtt{Dyck}_2(w')\xleftarrow{1}1$, 
and $\delta^{-1}(\mathtt{Dyck}_{2}(w))=\delta^{-1}(\mathtt{Dyck}_2(w'))\xleftarrow{2}1$.
Similarly, we have $\delta(w)=\delta(w')\xleftarrow{1}1$, and 
$\mathtt{Dyck}_2(\delta(w))=\mathtt{Dyck}_2(\delta(w'))\xleftarrow{2}1$. 
By induction assumption, the claim holds.

Case 3). We have $w=w'\xleftarrow{1}1$.
We have $\mathtt{Dyck}_2(w)=\mathtt{Dyck}_2(w')\xleftarrow{2}1$, and 
$\delta^{-1}(\mathtt{Dyck}_2(w))=\delta^{-1}(\mathtt{Dyck}_2(w'))\xleftarrow{3}1$.
Similarly, we have $\delta(w)=\delta(w')\xleftarrow{1}2$, and 
$\mathtt{Dyck}_2(\delta(w))=\mathtt{Dyck}_2(\delta(w'))\xleftarrow{3}1$.
By induction assumption, the claim holds.

Therefore, we have $\mathtt{Dyck}_2\circ\delta=\delta^{-1}\circ\mathtt{Dyck}_2$.
Since the map $\mathtt{Dyck}_2$ is an involution, $\mathtt{Dyck}_2$ can be written 
as $\mathtt{Dyck}_2=C\circ\mathtt{Rvac}$.
Here, $C$ satisfies 
\begin{align*}
C\circ\delta=\delta\circ C, \qquad C\circ\mathtt{Rvac}=\mathtt{Rvac}\circ C^{-1}.
\end{align*}
This implies that $C=\delta^{X}$ with some $X$.
	
We will show that $X=0$. To see this, consider the top Dyck path $U^{n}R^{n}$.
By a simple calculation, we have $\mathtt{Dyck}_{2}(U^{n}R^{n})=(UR)^{n}$, 
and $\mathtt{Rvac}(U^{n}R^{n})=(UR)^{n}$.
Therefore, we have $\mathtt{Dyck}_2=\mathtt{Rvac}$ on $U^{n}R^{n}$,
which implies $X=0$, and we have $\mathtt{Dyck}_2=\mathtt{Rvac}$.

We will show that $\mathtt{DRvac}=\mathtt{ev}\circ\mathtt{Rvac}$.
From Proposition \ref{prop:Rvac}, we have $\mathtt{DRvac}=\delta^{r+2}\circ\mathtt{Rvac}$.
Since $r=n-2$, we have $\mathtt{DRvac}=\delta^{n}\circ\mathtt{Rvac}=\mathtt{ev}\circ\mathtt{Rvac}$
by Lemma \ref{lemma:rmev}, which completes the proof.
\end{proof}

\section{Matching map and RSK correspondence}
\label{sec:MatRSK}
\subsection{Matching map}
\label{sec:defMatch}
We introduce a map on Dyck paths which we call {\it matching map} 
following \cite{ArmStuTho13}.
Let $P$ be a Dyck path of size $n$.  A valley of $P$ is a pattern $RU$ in $P$.
We introduce a bared integer $\bar{i}=2n+1-i$ for $i\in[1,2n]$.

We define an integer sequence $K(P):=(k_1,\ldots,k_{n})$ which consists of 
integers with or without a bar.
Suppose we have a valley of $P$ in the $i$-th row from top and $j$-th 
column from left. We define $v_{i}=j$ if we have a valley at the $i$-th row
and $v_{i}=\emptyset$ otherwise.
The entry $k_{i}$ is given by
\begin{align}
k_{i}:=
\begin{cases}
n-v_{i}+1, & \text{ if } v_{i}\neq\emptyset, \\
\bar{i}, & \text{ if } v_{i}=\emptyset.
\end{cases}
\end{align}

Since $K(P)$ consists of $n$ entries, we construct a chord diagram of size $n$
from $K(P)$.

\begin{enumerate}
\item Set $i=1$ and $I=[1,2n]$.
\item Recall that $k_i$ is an integer with or without a bar.
Pick a maximal (resp. minimal) integer $l_{i}$ in $I$ which is smaller (resp. larger) 
than $k_{i}$ if $k_i$ is an integer without (resp. with) a bar.
\item 
The algorithm stops when $i=n$. 
Otherwise, replace $I$ by $I\setminus\{k_i,l_{i}\}$ and increase $i$ by one.
Then, go to (2). 
\end{enumerate}
The pairs of two integers $(k_i,l_{i})$ $1\le i\le n$, define 
$n$ arches of a chord diagram.
We denote by $P'$ the Dyck path obtained from $K(P)$.

\begin{defn}
We define the matching map $\mathtt{Mat}:\mathtt{Dyck}(n)\rightarrow\mathtt{Dyck}(n)$, 
$P\mapsto P'$.
\end{defn}

\begin{example}
Let $P=URUURURRUR$ be a Dyck path of size $5$. We have valleys of $P$ in the first, second,
and fourth rows. From this, we have 
$K(P)=(2,4,\bar{3},5,\bar{5})$. Since $\bar{3}=8$ and $\bar{5}=6$, 
we have five arches: $(1,2),(3,4),(8,9),(5,10)$ and $(6,7)$ in the chord diagram 
corresponding to $\mathtt{Mat}(P)$.
From these, $\mathtt{Mat}(P)=URURUURURR$. 
\end{example}

The inverse of the map $\mathtt{Mat}$ is explicitly given in Section \ref{sec:Matk} 
for a $k$-Dyck path with $k\ge1$.

\subsection{Crossing perfect matchings and RSK correspondence}
\label{sec:cPMRSK}
Let $w=w_1\ldots w_n$ be a $321$-avoiding permutation.
We associate $w$ with a perfect matching with crosses and nests.
We define $n$ pairs $p_i$, $1\le i\le n$, of integers in $[1,2n]$ by 
$p_{i}:=(w_{n+1-i},n+i)$.
We put $2n$ labeled points $1,2,\ldots,2n$ in line from left to right.
If $k,l\in p_{i}$ for some $i$, we connect the two points $k$ and $l$
by an arch.
We denote by $\pi'(w)$ the perfect matching obtained as above.

Given a perfect matching $\pi'$, we obtain a non-crossing perfect matching
by resolving crosses by 
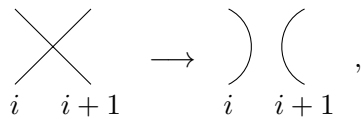
\begin{align*}
\tikzpic{-0.5}{
\draw(0,0)node[anchor=north]{$i$}--(1,1)(1,0)node[anchor=north]{$i+1$}--(0,1);
}
\longrightarrow
\tikzpic{-0.5}{
\draw(0,0)node[anchor=north]{$i$}..controls(0.4,0.2)and(0.4,0.8)..(0,1)
(1,0)node[anchor=north]{$i+1$}..controls(0.6,0.2)and(0.6,0.8)..(1,1);
},
\end{align*}
where $(i,j_1)$ and $(j_2,i+1)$ are arches satisfying $j_1>i+1$ and $j_2<i$.
We resolve all the crosses in $\pi'$, and obtain a non-crossing perfect matching
$\pi$.

\begin{defn}
We define a map $\mathtt{PM}^{\times}$ from a $321$-avoiding permutation to a non-crossing perfect matching 
by the composition of maps: $\mathtt{PM}^{\times}: w\mapsto \pi'(w)\mapsto \pi$.
\end{defn}

\begin{example}
\label{ex:31425}
Let $w=31425$. Then we have five pairs $\{(5,6),(2,7),(4,8),(1,9),(3,10)\}$.
The corresponding perfect matching and the non-crossing perfect matching for $w$ are 
given as follows:
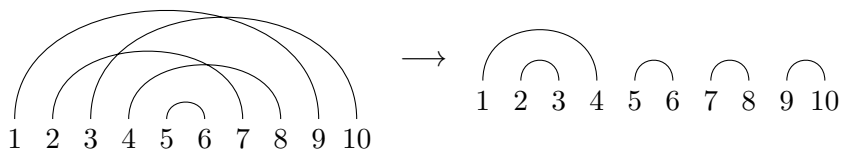
\begin{align*}
\tikzpic{-0.5}{[yscale=0.3,xscale=0.5]
\draw(5,0)..controls(5,1)and(6,1)..(6,0)
(2,0)..controls(2,4)and(7,4)..(7,0)(4,0)..controls(4,3.2)and(8,3.2)..(8,0)
(1,0)..controls(1,6.4)and(9,6.4)..(9,0)(3,0)..controls(3,6)and(10,6)..(10,0);
\foreach \x in {1,2,3,4,5,6,7,8,9,10} \draw(\x,0)node[anchor=north]{$\x$};
}
\longrightarrow
\tikzpic{-0.5}{[xscale=0.5,yscale=0.3]
\draw(1,0)..controls(1,3)and(4,3)..(4,0)
(2,0)..controls(2,1.2)and(3,1.2)..(3,0)
(5,0)..controls(5,1.2)and(6,1.2)..(6,0)
(7,0)..controls(7,1.2)and(8,1.2)..(8,0)(9,0)..controls(9,1.2)and(10,1.2)..(10,0);
\foreach \x in {1,2,...,10} \draw(\x,0)node[anchor=north]{$\x$};
}
\end{align*}
\end{example}

The map $\mathtt{PM}^{\times}$ is equivalent to the following operation on a Rothe diagram.
Let $w=w_1\ldots w_n$ be a permutation in one-line notation.
We consider an $n$-by-$n$ square and put a cross $\times$ at $w_{n+1-i}$-th row from bottom 
in the $i$-th column for $1\le i\le n$.
The set $D(w)$ of unit boxes in a Rothe diagram for $w$ is a set of unit boxes $c$ in the Rothe diagram 
such that there is no cross either above or to the left of $c$. 
We replace a unit box in the set $D(w)$ by the following tile:
\begin{align}
\label{eq:tile}
\tikzpic{-0.5}{
\draw(0,0)--(0,1)--(1,1)--(1,0)--(0,0);
\draw[red,thick](0.5,0)..controls(0.5,0.2)and(0.2,0.5)..(0,0.5);
\draw[red,thick](0.5,1)..controls(0.5,0.8)and(0.8,0.5)..(1,0.5);
}
\end{align}
For the remaining cells $c$, we put a horizontal (resp. vertical) line if 
$c$ has a cross right to (resp. below) it.
We number the left border of a Rothe diagram from the bottom by $1$ to $n$, and the top border 
from left by $n+1$ to $2n$.
In this way, we have a diagram with paths connecting $i$ and $j$ such that $1\le i<j\le 2n$.
Note that the number of crosses in a perfect matching is the same as the number of cells
in the set $D(w)$.

\begin{example}
Let $w=31425$. We have the following Rothe diagram:
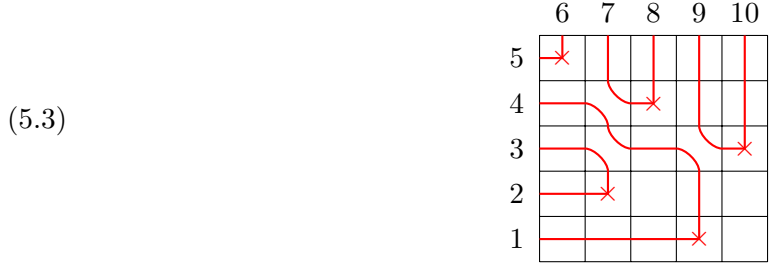
\begin{align}
\tikzpic{-0.5}{[scale=0.6]
\draw(0,0)grid(5,5);
\foreach \x/\y in {1/2,1/3,3/2}
{\draw[red,thick](\x+0.5,\y+0)..controls(\x+0.5,\y+0.2)and(\x+0.2,\y+0.5)..(\x+0,\y+0.5);
\draw[red,thick](\x+0.5,\y+1)..controls(\x+0.5,\y+0.8)and(\x+0.8,\y+0.5)..(\x+1,\y+0.5);}
\foreach \x/\y in {0.5/4.5,1.5/1.5,2.5/3.5,3.5/0.5,4.5/2.5}
{\draw[red,thick](\x,\y)node{$\times$};
\draw[red,thick](\x-0.5,\y)--(\x,\y)--(\x,\y+0.5);}
\foreach \x/\y in {0.5/0.5,0.5/1.5,0.5/2.5,0.5/3.5,1.5/0.5,2.5/0.5,2.5/2.5} {\draw[red,thick](\x-0.5,\y)--(\x+0.5,\y);}
\foreach \x/\y in {1.5/4.5,2.5/4.5,3.5/4.5,3.5/3.5,3.5/1.5,4.5/4.5,4.5/3.5} {\draw[red,thick](\x,\y-0.5)--(\x,\y+0.5);}
\foreach \x in {1,2,3,4,5}{\draw(-0.5,\x-0.5)node{$\x$};}
\foreach \x in {6,7,8,9,10}{\draw(\x-5.5,5.5)node{$\x$};}
}
\end{align}
This Rothe diagram gives the non-crossing perfect matching $\{\{1,4\},\{2,3\},\{5,6\},\{7,8\},\{9,10\}\}$.
\end{example}

\begin{prop}
\label{prop:RSKPM}
Let $w$ be a $321$-avoiding permutation.
We have $\widehat{RSK}(w)=\mathtt{PM}^{\times}(w)$ as a Dyck path.
\end{prop}
\begin{proof}
Let $w=w_1\ldots w_n$ be a $321$-avoiding permutation, and $w_{\downarrow}=w_1\cdots w_{n-1}$.
To show the equivalence $\widehat{RSK}(w)=\mathtt{PM}^{\times}(w)$, it is enough to show that 
$\widehat{RSK}(w)$ gives the Dyck path which corresponds to a non-crossing perfect matching 
in the Rothe diagram for $w$.
 
We first study the $\widehat{RSK}(w)$.
We insert the integer $w_{n}$ into the pair of two-row Young tableau 
$(\mathtt{ins}(w_{\downarrow}),\mathtt{rec}(w_{\downarrow}))$
where $\mathtt{ins}(w_{\downarrow})$ (resp. $\mathtt{rec}(w_{\downarrow})$) is the insertion 
(resp. recording) tableau by the RSK correspondence for $w_{\downarrow}$.
We consider the following three cases: 1) $w_{n}<n$ and $w_{n}+1$ is in the first row in $\mathtt{ins}(w_{\downarrow})$,
2) $w_{n}<n$ and $w_{n}+1$ is in the second row in $\mathtt{ins}(w_{\downarrow})$, 
and 3) $w_{n}=n$.

Case 1). By inserting $w_{n}$ into the first row of $\mathtt{ins}(w_{\downarrow})$, we 
remove $w_{n}+1$ from the first row, and add $w_{n}+1$ in the second row.
Since $w$ is $321$-avoiding, the integer $w_{n}+1$ stays in the second row.
The other integers in $\mathtt{ins}(w)$ are in the same row as $\mathtt{ins}(w_{\downarrow})$.
Similarly, we put $n$ in the second row of $\mathtt{rec}(w_{\downarrow})$ since $w_{n}+1$ 
is placed in the second row in $\mathtt{ins}(w)$.

Case 2). Let $\alpha$ be the smallest integer in $\mathtt{ins}(w_{\downarrow})$ which is larger than $w_{n}$.
If there is no such an integer, we append $w_{n}$ in the first row of $\mathtt{ins}(w_{\downarrow})$, and 
similarly append $n$ in the first row of $\mathtt{rec}(w_{\downarrow})$.
The insertion of $w_{n}$ removes $\alpha$ in the first row of $\mathtt{ins}(w_{\downarrow})$ and 
places it on the second row. Similarly, we append $n$ in the second row of $\mathtt{rec}(w_{\downarrow})$.

Case 3). Since $w_n=n$, we append $n$ in the first row of $\mathtt{ins}(w_{\downarrow})$ 
and $\mathtt{rec}(w_{\downarrow})$.

Note that in the first two cases, $w_{n}$ is in the first row and $w_{n}+1$ (if it exists) 
is in the second row of $\mathtt{ins}(w)$, and $w_n=n$ is in the first row.

Secondly, we consider a non-crossing perfect matching in the Rothe diagram of $w$.
Obviously, the Rothe diagram has a mark $m$ in the $w_{n}$-th row from bottom in the first column.
All the cells above $m$ are in the set $D(w)$.
We replace a cell in the set $D(w)$ by the tile expressed in Eq. (\ref{eq:tile}).
This means that we have a matching between $w_{n}$ and $w_{n}+1$.
Further, if $\{i,j\}$ with $i<j$ is a matching in $w_{\downarrow}$, then we 
increase $i$ (resp. $j$) by two if $i>w_{n}$ (resp. $j>w_{n}$) and make a matching 
by new $i$ and $j$.	 
We compare this observation with the Dyck path $\widehat{RSK}(w)$, and 
it is straightforward to see $\widehat{RSK}(w)=\mathtt{PM}^{\times}(w)$.
\end{proof}

To have the inverse $\widehat{RSK}^{-1}$, we introduce the notion 
of maximal cover-inclusive Dyck tilings following \cite{SZJ12}.
The advantage of the use of Dyck tilings is that we need only a Young diagram above 
a Dyck path, and we do not use the property of a $321$-avoiding permutations 
as $\mathtt{PM}^{\times}$. 
This simplicity plays a role when we define $\widehat{RSK}^{-1}$ for $k$-Dyck 
paths with $k\ge2$ in Section \ref{sec:prorow}.

A Dyck tile is a connected skew shape which consists of several unit cells such that 
the centers of the cells form a Dyck path.
A Dyck tile $d$ is said to be trivial (resp. non-trivial) if $d$ consists of 
a single cell (resp. several cells).
Let $\lambda, \mu$ be two Dyck paths such that $\lambda\le\mu$, i.e., $\mu$
is above $\lambda$.
We consider a tiling by Dyck tiles in the skew shape $\lambda/\mu$.
We say that a tiiling $T$ is a cover-inclusive Dyck tiling if 
we move a Dyck tile $d_1$ in $T$ in a southeastern direction by one unit, then
$d_1$ is below $\lambda$, or contained another Dyck tile $d_2$.
Here, the size of $d_2$ may be the same as that of $d_1$.
A cover-inclusive Dyck tiling $D$ is maximal if $D$ contains Dyck tiles of 
maximum sizes. In other words, we cannot enlarge a Dyck tile in $D$ any more. 

We consider the maximal cover-inclusive Dyck tiling $T$ above $\lambda$ and 
below the top path $U^nR^n$.
The Dyck tiling $T$ consists of several Dyck tiles.
We consider the order of Dyck tiles to remove them one-by-one from bottom to top.
Given two Dyck tiles $d_1$ and $d_2$, we write $d_1\lessdot d_2$ if $d_1$ 
is removed before $d_2$.
We fix a removal order on Dyck tiles and write $d_1\lessdot d_2\lessdot \ldots\lessdot d_r$. 
Since a Dyck tile is characterized by a Dyck path, a Dyck tile has a unique south-most edge 
and a unique rightmost edge. 
Since the Dyck tile $d_1$ is just above the path $\lambda$,  its south-most and rightmost edges 
are on $\lambda$. If the south-most (resp. rightmost) edge of $d_1$ is the $i_1$-th (resp. $j_1$-th)
step from left in $\lambda$, we associate $d_1$ with the transposition $t_{i_1,j_1}$.
We obtain a new path $\lambda_1$ which is the lowest path above $\lambda$ and a Dyck tile $d_1$.
Then, the Dyck tile $d_2$ is just above $\lambda_{1}$.
As in the case of $d_1$, we associate $d_2$ with the transposition $t_{i_2,j_2}$ if 
the south-most (resp. rightmost) edge of $d_2$ is the $i_2$-th (resp. $j_2$-th) step in $\lambda_{1}$.
By repeating the above procedure, we associate a Dyck tile $d_{p}$ with the transposition $t_{i_p,j_p}$
for $1\le p\le r$.

Recall that a chord diagram corresponding to $\lambda$ consists of $n$ pairs of integers.
We act $r$ transpositions $t_{i_p,j_p}$, $1\le p\le r$, on the pairs of integers by exchanging the 
integers $i_p$ and $j_p$ starting from $p=1$.
Then, we obtain a perfect matching which may have crosses.
This perfect matching consists of $n$ pairs, and the integers in $[n+1,2n]$ belong to 
distinct pairs. We define the permutation $w=w_1\ldots w_{n}$ where the integer $w_i$ 
forms a pair with $n+i$, for $1\le i\le n$. 
In this way, we obtain a permutation $w$ from a Dyck path $\lambda$.

\begin{defn}
We define the map from a Dyck path to $321$-avoiding permutation by $\mathtt{DT}$, i.e., 
$\mathtt{DT}:\lambda\mapsto w$.
\end{defn}

\begin{lemma}
The map $\mathtt{DT}$ is well-defined. In other words, the non-crossing perfect 
matching $\mathtt{DT}(P)$ contains arches connecting $i$ and $j$ such that
$i\in[1,n]$ and $j\in[n+1,2n]$.
\end{lemma}
\begin{proof}
For the map $\mathtt{DT}$, we consider a Dyck tiling above $P$ and the top path 
$U^{n}R^{n}$. Suppose that we have a up step in $P$ at position $i$ with $i\ge n+1$.
We make use of the maximal cover-inclusive Dyck tiling below $U^{n}R^{n}$, and 
assign a transposition $t_{i_p,j_p}$ to each Dyck tile. 
The product of transpositions which are assigned to Dyck tiles just above $P$ 
changes the pair $(i_1,j_1)$ with the maximal $j_{1}\in[n+1,2n]$ to the pair 
$(i'_{1},j_{1})$ with $i'\in[1,n]$.
We remove the Dyck tiles just above $P$, then we obtain a new Dyck path $P'$.
The Dyck tiles just above $P'$ changes the pair $(i_{2},j_{2})$ with the second 
maximal $j_{2}\in[n+1,2n]$ to the pair $(i'_{2},j_{2})$ with $i'_2\in[1,n]$.
We continue this process until we obtain the top path $U^{n}R^{n}$.
Then, the all the pairs $(i,j)$ with $i,j\in[n+1,2n]$ are changed to 
the pairs $(i',j)$ with $i'\in[1,n]$.
The non-crossing perfect matching $\mathtt{DT}(P)$ consists of only the 
pairs $(i,j)$ with $i\in[1,n]$ and $j\in[n+1,2n]$, which completes the 
proof.
\end{proof}

The following proposition relates a Dyck tiling to the map $\widehat{RSK}^{-1}$.
\begin{prop}
\label{prop:RSKDT}
Let $P$ be a Dyck path of length $n$.
We have $\widehat{RSK}^{-1}(P)=\mathtt{DT}(P)$ as a $321$-avoiding permutation.
\end{prop}
\begin{proof}
We denote $\overline{i}:=2n+1-i$.
We prove the claim by induction on $n$. For $n=1,2$, the claim follows by direct calculations.
We assume that the claim holds up to $n-1$.
We consider two cases: 1) $\overline{n}$ is a right step, and 2) $\overline{n}$ is 
an up step.

Case 1). Let $\alpha$ be the maximal integer such that $\alpha$-th step is an up step and $\alpha\in[1,n]$.
Let $P_{-}$ be the Dyck path obtained from $P$ by deleting $\alpha$-th and $\overline{n}$-th steps.
This operation is well-defined since we have $2n$ steps in $P$ in total.
Let $w_{-}$ be the $321$-avoiding permutation obtained from $P_{-}$ by $\widehat{RSK}^{-1}$.
By induction hypothesis, $w_-=\mathtt{DT}(P_-)$.
The position of an up step above $\alpha$-th up step in $P$ is expressed as $\overline{i}$ 
with $i\in[1,n]$ due to the maximal property of $\alpha$.
Recall that we consider the maximal cover-inclusive Dyck tiling above $P$ and below $U^{n}R^{n}$.
Since $\overline{n}$ is a right step by assumption, 
we have a non-trivial Dyck tile just above the $\alpha$-th up step by the maximal property of the Dyck tiling
if there is a right step at the position $j$ with $1\le j\le \alpha-1$.
By chasing the transpositions corresponding to Dyck tiles, we have the following observations:
\begin{enumerate}[(a)]
\item We have the pair $(\alpha,\overline{n})$, 
\item The pairs $(i,\overline{j})$ in $\mathtt{DT}(P_{-})$ are changed to 
$(i,\overline{j+2})$ if $i<\alpha$, and $(i+1,\overline{j+2})$ if $i\ge\alpha$.
\end{enumerate}
These imply that we have a permutation $w$ obtained from $w_{-}$ by increasing 
$i$ in $w_{-}$ if $i>\alpha$ by one, and by appending $\alpha$ from right.
On the other hand, the action of $\widehat{RSK}^{-1}$ on $P$ gives the same 
permutation $w$ since we have $\alpha$ in the first row of the insertion tableau
and $n$ in the first row of the recording tableau.
From these, we have $\widehat{RSK}^{-1}(P)=\mathtt{DT}(P)$.
	
Case 2). The $\overline{n}$-th step is an up step in $P$.
Let $\beta\in[1,n]$ be the maximal integer such that the $\beta$-th step
is a right step, and $\alpha<\beta$ be the maximal integer such that 
$\alpha$-th step is an up step in $P$.
Let $P_{-}$ be a Dyck path of size $n-1$ obtained from $P$ by changing 
the $\beta$-th right step to an up step, and by deleting the $\alpha$-th 
and $\overline{n}$-th steps. 
By induction assumption, we have $w_{-}=\widehat{RSK^{-1}}(P_{-})=\mathtt{DT}(P_-)$.
As in Case 1), by chasing the transpositions corresponding to Dyck tiles,
$\mathtt{DT}(P)$ is obtained from $\mathtt{DT}(P_{-})$ by adding the pair 
$(\alpha,\overline{n})$, and by changing the pairs as in (b) of Case 1).
On the other hand, $\widehat{RSK}^{-1}(P)$ gives the pair of two-row tableaux, and 
$\alpha$ is the second row of the insertion tableau, and $n$ is the second row of the 
recording tableau. In $w_{-}=\widehat{RSK}^{-1}(P_{-})$, $\alpha-1$ is in the first row of the 
insertion tableau. From these, we insert $\beta$ into the two-row tableaux to obtain 
$\widehat{RSK}^{-1}(P)$.
By combining these observations, $\widehat{RSK}^{-1}(P)$ and $\mathtt{DT}(P)$ give the 
same permutation $w$ obtained from $w_-$ by adding $\alpha$ from right and by the standardization.
This completes the proof.
\end{proof}

\begin{example}
Let $P=U^2R^2(UR)^3$ be a Dyck path.
The maximal Dyck tiling for $P$ is given by
\begin{align*}
\tikzpic{-0.5}{[scale=0.5]
\draw[thick](0,0)--(0,2)--(2,2)--(2,3)--(3,3)--(3,4)--(4,4)--(4,5)--(5,5);
\draw(0,2)--(0,5)--(4,5)(1,2)--(1,4)--(2,4)--(2,5)(0,3)--(1,3);
\draw[red](1.5,2.5)--(1.5,3.5)--(2.5,3.5)--(2.5,4.5)--(3.5,4.5)
(0.5,3.5)--(0.5,4.5)--(1.5,4.5)(0.5,2.5)node{$\bullet$};
}
\end{align*}
We have three Dyck tiles and these tiles give the following 
transpositions $(4,9),(3,4)$ and $(4,7)$.

The Dyck path $P$ gives five pairs of integers, and if we act the transpositions on it,
we have 
\begin{align}
\label{eq:matcpm}
\begin{matrix}
1 & 2 & 5 & 7 & 9\\
4 & 3 & 6 & 8 & 10
\end{matrix}
\xrightarrow{(4,9)}
\begin{matrix}
1 & 2 & 5 & 7 & 4\\
9 & 3 & 6 & 8 & 10
\end{matrix}
\xrightarrow{(3,4)}
\begin{matrix}
1 & 2 & 5 & 7 & 3\\
9 & 4 & 6 & 8 & 10
\end{matrix}
\xrightarrow{(4,7)}
\begin{matrix}
1 & 2 & 5 & 4 & 3\\
9 & 7 & 6 & 8 & 10
\end{matrix}
\end{align}
The rightmost matrix in Eq. (\ref{eq:matcpm}) gives 
the crossing perfect matching by rearranging the order of pairs: 
\begin{align*}
\begin{matrix}
1 & 2 & 5 & 4 & 3\\
9 & 7 & 6 & 8 & 10
\end{matrix}
\rightarrow
\begin{matrix}
5 & 2 & 4 & 1 & 3\\
6 & 7 & 8 & 9 & 10
\end{matrix}
\end{align*}
Therefore, we have $31425$ from $P$. 
Compare this result with Example \ref{ex:31425}.
\end{example}

\subsection{Matching map and RSK correspondence}
In this section, we study the relation between the matching map $\mathtt{Mat}$
and the RSK correspondence $\widehat{RSK}$.

Let $P$ be a Dyck path and $w:=w(P)$ be the $321$-avoiding permutation 
corresponding to $P$.
Then, we have the following theorem.
\begin{theorem}
\label{thrm:RSKMat}
Let $P$, $w$ be as above.
Then, $\widehat{RSK}(w)=\partial\circ\mathtt{Mat}\circ\mathtt{ev}(P)$. 
\end{theorem}
\begin{proof}
We prove the statement by induction on $n$. For $n=1,2$, it is obvious 
that we have $\widehat{RSK}(w)=\partial\circ\mathtt{Mat}\circ\mathtt{ev}(P)$ where 
$w$ is a $321$-avoiding permutation for $P$.
We assume the statement holds up to $n-1$.

Let $w=w_1\cdots w_{n}$ be the $321$-avoiding permutation for a Dyck path $P$.
We first study $\widehat{RSK}(w)$. 
We denote $w_{-}:=w_1\cdots w_{n-1}$.
Let $(\mathtt{ins}(w),\mathtt{rec}(w))$ be a pair of standard tableaux obtained from $w$
by the RSK correspondence. We call the first (resp. second) tableau an insertion (resp. recording) tableau.
We insert $w_{n}$ into the insertion tableau $\mathtt{ins}(w_{-})$.
Let $\alpha$ be the minimal element in the first row of $\mathtt{ins}(w_{-})$ which is larger than $w_{n}$.
If such $\alpha$ does not exist, we put $w_{n}$ in the first row of $\mathtt{ins}(w_{-})$ and obtain 
$\mathtt{ins}(w)$. In the recording tableau, we append $n$ in the first row of $\mathtt{rec}(w_{-})$ and 
obtain $\mathtt{rec}(w)$. 
If $\alpha\in[w_{n}+1,n]$, we remove $\alpha$ from $\mathtt{ins}(w_{-})$, put $w_{n}$ there, and 
put $\alpha$ in the second row of $\mathtt{ins}(w_{-})$. Since $w$ is $321$-avoiding, there is no
element $\beta$ which is larger than $\alpha$. 
Since we put $\alpha$ in the second row of $\mathtt{ins}(w_{-})$, we put $n$ in the second row 
of $\mathtt{rec}(w_{-})$ and obtain the recording tableau $\mathtt{rec}(w)$.
As a summary, $w_{n}$ (resp. $w_{n}+1$) is in the first (resp. second) row of $\mathtt{ins}(w)$, 
and $n$ is in the first (resp. second) row of $\mathtt{rec}(w)$ if $\alpha=\emptyset$ 
(resp. $\alpha\in[w_n+1,n]$).

Secondly, we study the composition $\partial\circ\mathtt{Mat}\circ\mathtt{ev}(P)$.
Recall that $\mathtt{Mat}(P)$ gives a sequence $K(P)$ of integers with or without a bar. 
We apply $\mathtt{Mat}$ to $P$ after the evacuation $\mathtt{ev}$, and 
as a result, the sequence $K(P)$ does not contain the integer $w_{n}+1$ but contains $w_{n}+2$.
The action of $\partial$ on $K(P)$ is given by replacing $i$ (resp. $\overline{i}$)
with $i-1$ (resp. $\overline{i+1}$). This is well-defined since the integer $i$
in $K(P)$ is in $[2,n]$ if $i$ is without a bar, and in $[1,n]$ if $i$ is with a bar.

Let $\overline{P}$ and $\overline{P_{-}}$ be the Dyck path corresponding to 
$\mathtt{ev}(P)$ and $\mathtt{ev}(P_{-})$ where $P_-$ corresponds to $w_-$. 
The integer sequence $K(\overline{P})$ is obtained from $K(\overline{P_{-}})$ 
by increasing the integers without a bar by one which are weakly larger than $w_{n}$, and 
by adding the bared integer $\overline{n}$. 
We obtain a non-crossing perfect matching from $K(P)$ as in Section \ref{sec:defMatch}. 
The previous paragraph implies that $K(\overline{P_-})$ always contains the integer $w_{n}$ without 
a bar. 
The increment of the integer $w_n$ by one is equivalent to insertion of an arch 
connecting $w_n$ and $w_{n}+1$ into the non-crossing perfect matching for $K(P_{-})$.
Further, we add $\overline{n}$ to $K(P_-)$. 
This corresponds to rename the arch connecting 
the integer $n-1$ and another integer into $\overline{n}$.
The action of $\partial$ simply moves the arches left by one.
By induction assumption, we have $\widehat{RSK}(w_{-})=\partial\circ\mathtt{Mat}\circ\mathtt{ev}(P_{-})$.

By combining these observations together, it is clear that $\widehat{RSK}(w)=\partial\circ\mathtt{Mat}\circ\mathtt{ev}(P)$,
which completes the proof.
\end{proof}

The following Corollary is a direct consequence of Proposition \ref{prop:RSKPM} and Theorem \ref{thrm:RSKMat}.
\begin{cor}
We have $\mathtt{Mat}=\partial^{-1}\circ\mathtt{PM}^{\times}\circ\mathtt{ev}$.
\end{cor}

\subsection{Lalanne--Kreweras involution and RSK correspondence}
To study the relation between the Lalanne--Kreweras involution and the RSK correspondence, 
we first study the relation between the rowmotion and RSK correspondence.

\begin{prop}
\label{prop:cd1}
We have the following commutative diagram on Dyck paths:
\begin{align}
\label{eq:cd1}
\tikzpic{-0.5}{[scale=0.8]
\node (0) at (0,0){$P$};
\node (1) at (3,0){$P'$};
\node (2) at (0,-2){$Q$};
\node (3) at (3,-2){$Q'$};
\draw[->,anchor=south] (0) to node {$\delta$} (1);
\draw[->,anchor=south] (2) to node {$\partial^{-1}$} (3);
\draw[->,anchor=east] (0) to node {$\widehat{RSK}$}(2);
\draw[->,anchor=west] (1) to node {$\widehat{RSK}$}(3);
}
\end{align}
\end{prop}
\begin{proof}
Let $Q$ be a Dyck path. 
We first show that we have the commutative diagram when $Q=U^nR^r$.
By a simple observation, we have $\widehat{RSK}(U^nR^n)=12\cdots n$.
The action of the rowmotion on $12\cdots n$ gives the permutation $23\cdots n1$.
By the RSK correspondence, this permutation gives the pair of two-row Young tableaux 
\begin{align}
\left(
\begin{ytableau}
2 \\
1 & 3 & \cdots & n
\end{ytableau}\quad, \quad
\begin{ytableau}
n \\
1 & 2 & \cdots & n'
\end{ytableau},
\right),
\end{align}
where $n'=n-1$.	This pair of two-row Young tableaux gives the permutation $(2,3\cdots,n-1,n,1)$.
The action of $\widehat{RSK}^{-1}$ on this pair of tableaux gives the Dyck path $URU^{n-1}R^{n-1}$.
On the other hand, the promotion of the Dyck path $U^nR^n$ is given by the permutation $(1,n,2,3\cdots,n-1)$,
and we have the Dyck path $URU^{n-1}R^{n-1}$.
Therefore, we have the commutative diagram (\ref{eq:cd1}) for $Q=U^nR^n$.

We prove the claim by induction on $n$. For $n=1,2$, it is obvious that we have the commutative diagram (\ref{eq:cd1}).
We assume that the diagram (\ref{eq:cd1}) holds up to $n-1$.
A Dyck path $Q\neq U^nR^n$ has several peaks between the $\overline{i}$-th and $\overline{i+1}$-th edges 
where $\overline{i}=n+1-i$ and $1\le i\le n$. We take the maximal $i$ among such $\overline{i}$. 
Then $Q$ is obtained from $Q_{-}$ by inserting the peak at the $\overline{i}$-th position where 
$Q_{-}$ is a Dyck path of size $n-1$. 
By induction hypothesis, it is enough to show that $Q'$ is obtained from $Q'_{-}=\partial^{-1}(Q_{-})$ by inserting 
the peak between the $\overline{i-1}$-th and $\overline{i}$-th edges.
Recall that the RSK gives two Young tableaux which are called an insertion tableau and a recording tableau.
The existence of the peak between the $\overline{i}$-th and $\overline{i+1}$ implies that 
the $\overline{i}$-th (resp. $\overline{i+1}$-th) step is $R$ (resp. $U$) in $Q$.
In terms of the recording tableau, $i$ (resp. $i+1$) is in the first (resp. second) row.
By the action of $\widehat{RSK}^{-1}$, the permutation $w:=\widehat{RSK}^{-1}(Q)$ satisfies 
$w_{i}>w_{i+1}$. 
It is easy to see that the action of the rowmotion $\delta$ on $w$ gives $w_{i-1}>w_{i}$. 
The action of $\widehat{RSK}$ gives the Dyck path which has a peak between the $\overline{i-1}$-th 
and $\overline{i}$-th edges.
These observations imply the commutative diagram (\ref{eq:cd1}), which completes the proof.
\end{proof}

Recall that the rowvacuation $\mathtt{Rvac}$ is expressed 
as $\mathtt{Rvac}=\delta^{(r)}\delta^{(r-1)}\cdots\delta^{(0)}$ where 
$\delta^{(i)}=\mathbf{t}_{i}\cdots\mathbf{t}_{n-1}$ is defined 
by a product of toggles.
Let $\widehat{\delta}:=\delta^{(r)}\cdots\delta^{(1)}$.

\begin{prop}
\label{prop:cd2}
We have the commutative diagram:
\begin{align}
\label{eq:cd2}
\tikzpic{-0.5}{[scale=0.8]
\node (0) at (0,0){$P$};
\node (1) at (3,0){$P'$};
\node (2) at (0,-2){$Q$};
\node (3) at (3,-2){$Q'$};
\draw[->,anchor=south] (0) to node {$\widehat{\delta}$}(1);
\draw[->,anchor=south] (2) to node {$\mathtt{ev}$}(3);
\draw[->,anchor=east] (0) to node {$\widehat{RSK}$}(2);
\draw[->,anchor=west] (1) to node {$\widehat{RSK}$}(3);
}
\end{align}
\end{prop}
\begin{proof}
To prove the statement, since $\mathtt{ev}$ reverses the order and the steps in a Dyck path,
it is enough to show that the action of $\widehat{\delta}$ on a Dyck path $R$ gives the Dyck path $R'$
such that the $321$-avoiding permutations of $R$ and $R'$ are inverse to each other.
Since we consider only the $321$-avoiding permutation, a permutation $w'$ which is inverse to $w$
is obtained by reflecting the Rothe diagram of $w$ along the diagonal line $y=x$.

We prove the statement by induction on $n$. 
For $n=1,2$, the diagram (\ref{eq:cd2}) is commutative.
We assume that (\ref{eq:cd2}) is commutative up to $n-1$.
We consider the two cases: 1) $R$ is not a prime Dyck path, and 2) $R$ is a prime Dyck path.

1) Since $R$ is not a prime Dyck path, $R$ can be decomposed into a concatenation of two 
Dyck paths $R_1$ and $R_2$ such that $R=R_1\circ R_2$.
We consider the action of $\widehat{\delta}$ on $R_1\circ R_2$. 
Since $R_1$ and $R_2$ are Dyck paths of size smaller than $n$, the action of $\widehat{\delta}$ on 
$R$ is locally equivalent to the actions of $\widehat{\delta}$ on $R_1$ and $R_2$.
By induction hypothesis, $\widehat{\delta}(R)$ is obtained by concatenating the two Dyck paths 
$R'_1$ and $R'_2$ where $R'_1$ (resp. $R'_2$) is $R'_1=\widehat{\delta}(R_1)$ (resp. $R'_2=\widehat{\delta}(R_2)$).
From these, we have $R'=\widehat{\delta}(R)$.

2) Since $R$ is a prime Dyck path, $R$ does not touch the line $y=x$ except the first and last end points.
Let $\widetilde{R}$ be a Dyck path obtained from $R$ by deleting the first and last edges.
The action of $\widehat{\delta}$ on $R$ is equivalent to the action of $\mathtt{Rvac}$ on $\widetilde{R}$.
From Proposition \ref{prop:rv}, the action of $\mathtt{Rvac}$ is equivalent to the action of 
$\mathtt{Dyck}_2$ on $\widetilde{R}$. The map $\mathtt{Dyck}_2$ defines a Dyck path in such a way that
the marked points give a valley on the new Dyck path $\widetilde{R}'$.
Recall the following two facts: 
a) a Dyck path in the Rothe diagram defines a Young diagram above it, and 
b) $R$ is obtained from $\widetilde{R}$ by adding an up step from left and a right step from right. 
From a) and b), it is easy to see that the Young diagrams for $\widetilde{R}'$ and $R'$ are the same.
This implies that the Young diagram for $\widetilde{R}'$ defines the positions of marked points 
for $R'$ in such a way that a peak gives a marked point.
Then, it is clear that the action of $\widehat{\delta}$ on $R$ gives the Dyck path $R'$ whose 
permutation is inverse of that of $R$. 
This completes the proof.
\end{proof}

Recall that the map $\mathtt{Dyck}_3$ on a Dyck path $P$ gives the inverse permutation for $P$.
The next corollary is a direct consequence of the proof of Proposition \ref{prop:cd2}.
\begin{cor}
\label{cor:cd3}
We have the commutative diagram:
\begin{align}
\label{eq:cd3}
\tikzpic{-0.5}{[scale=0.8]
\node (0) at (0,0){$P$};
\node (1) at (3,0){$P'$};
\node (2) at (0,-2){$Q$};
\node (3) at (3,-2){$Q'$};
\draw[->,anchor=south] (0) to node {$\mathtt{Dyck}_3$}(1);
\draw[->,anchor=south] (2) to node {$\mathtt{ev}$}(3);
\draw[->,anchor=east] (0) to node {$\widehat{RSK}$}(2);
\draw[->,anchor=west] (1) to node {$\widehat{RSK}$}(3);
}
\end{align}
\end{cor}

\begin{cor}
\label{cor:RSKDyck2}
We have the commutative diagram:
\begin{align}
\label{eq:cd4}
\tikzpic{-0.5}{[scale=0.8]
\node (0) at (0,0){$P$};
\node (1) at (3,0){$P'$};
\node (2) at (0,-2){$Q$};
\node (3) at (3,-2){$Q'$};
\draw[->,anchor=south] (0) to node {$\mathtt{Dyck}_2$}(1);
\draw[->,anchor=south] (2) to node {$\mathtt{ev}\circ\partial^{-1}$}(3);
\draw[->,anchor=east] (0) to node {$\widehat{RSK}$}(2);
\draw[->,anchor=west] (1) to node {$\widehat{RSK}$}(3);
}
\end{align}
\end{cor}
\begin{proof}
Note that $\mathtt{Dyck}_2=\widehat{\delta}\circ\delta$.
From Propositions \ref{prop:cd1} and \ref{prop:cd2}, we have the commutative diagram.
\end{proof}

\begin{cor}
We have the commutative diagram:
\begin{align}
\label{eq:cd5}
\tikzpic{-0.5}{[scale=0.8]
\node (0) at (0,0){$P$};
\node (1) at (3,0){$P'$};
\node (2) at (0,-2){$Q$};
\node (3) at (3,-2){$Q'$};
\draw[->,anchor=south] (0) to node {$\mathtt{Dyck}_2$}(1);
\draw[->,anchor=south] (2) to node {$\mathtt{ev}\circ\partial$}(3);
\draw[->,anchor=east] (0) to node {$\mathtt{Mat}$}(2);
\draw[->,anchor=west] (1) to node {$\mathtt{Mat}$}(3);
}
\end{align}
\end{cor}
\begin{proof}
From Theorem \ref{thrm:RSKMat}, we have
\begin{align*}
\mathtt{Mat}\circ\mathtt{Dyck}_2\circ\mathtt{Mat}^{-1}
&=\partial^{-1}\circ\widehat{RSK}\circ\mathtt{ev}\circ\mathtt{Dyck}_2\circ\mathtt{ev}\circ\widehat{RSK}^{-1}\circ\partial, \\
&=\partial^{-1}\circ\widehat{RSK}\circ\mathtt{Dyck}_2\circ\widehat{RSK}^{-1}\circ\partial, \\
&=\partial^{-1}\circ\mathtt{ev}\circ\partial^{-1}\circ\partial, \\
&=\mathtt{ev}\circ\partial,
\end{align*}
where we have used $\mathtt{ev}\circ\mathtt{Dyck}_2\circ\mathtt{ev}=\mathtt{Dyck}_2$ and 
Corollary \ref{cor:RSKDyck2}. 
\end{proof}

As an application of $\widehat{RSK}$, we are ready to prove Lemma \ref{lemma:rmev}.
The rowmotion and the evacuation on Dyck paths have the following relation. 
\begin{cor}
\label{cor:rmev}
We have $\delta^{n}=\mathtt{ev}$.
\end{cor}
\begin{proof}
By definition of $\delta$ and $\mathtt{ev}$, it is obvious that we have 
\begin{align}
\label{eq:evdel}
\mathtt{ev}\circ\delta=\delta\circ\mathtt{ev}, \qquad \mathtt{ev}^2=\mathtt{id}.
\end{align}
It is enough to construct a map which has the same property as $\mathtt{ev}$ in	 Eq. (\ref{eq:evdel}).
From (3) in Theorem \ref{thrm:relproev} and Eq. (\ref{eq:evDyck}) in Proposition \ref{prop:DyckPMev},
we have $\partial^{2n}=1$.
From Proposition \ref{prop:cd1}, we have $\delta=\widehat{RSK}^{-1}\circ\partial^{-1}\circ\widehat{RSK}$.
Therefore, we have 
\begin{align*}
\delta^{2n}&=\widehat{RSK}^{-1}\circ\partial^{-2n}\circ\widehat{RSK}, \\
&=1,
\end{align*}
where we have used $\partial^{-2n}=1$. 
Then, it is obvious that $\delta^{n}$ satisfies the same relation as Eq. (\ref{eq:evdel}).
The evacuation $\mathtt{ev}=C\cdot\delta^{n}$ with some $C$ such that $C^{2}=1$ and $\delta\circ C=C\circ\delta$.
To fix $C$, we compare the action of $\mathtt{ev}$ and $\delta^{n}$ on the top Dyck path $U^{n}R^{n}$.
We have $\mathtt{ev}(U^nR^n)=U^nR^n$ and $\delta^{n}(U^nR^n)=U^nR^n$. 
From these, we have $C=1$. This implies that $\mathtt{ev}=\delta^{n}$.
\end{proof}

\section{Non-crossing weighted partitions}
\label{sec:NCPTL}
\subsection{Non-crossing partitions and Dyck paths}
\label{sec:NCPDyck}
We introduce the notion of non-crossing partitions following \cite{Kre72}.
A set partition $\pi$ of the set $[N]:=\{1,2,\ldots,N\}$ is said to be {\it non-crossing} if 
$i$ and $k$, and $j$ and $l$ belong to the same blocks with $i<j<k<l$,
then all indices $i, j, k$ and $l$ belong to the same block.
For example, $\{\{1,2\},\{3,4\}\}$ is a non-crossing partition of $[4]$, and 
$\{\{1,3\},\{2,4\}\}$ is not non-crossing.
We denote by $\mathtt{NC}(N)$ the set of non-crossing partitions of $[N]$.
The cardinality of $\mathtt{NC}(N)$ is given by the $N$-th Catalan number.

Let $B_1,\ldots,B_m$ be the blocks of $\pi\in\mathtt{NC}(N)$, where 
$m$ is the number of blocks in $\pi$.
Suppose $l(B)$ be the size of the block $B$ in $\pi$, i.e., $l(B)=|B|$.
We associate a Dyck path $\mu(B)$ to a block $B$ in $\pi$ as follows.
If $l(B)=1$, we associate the Dyck path $\mu(B):=UR$ to $B$.
If $l:=l(B)\ge2$, we associate the Dyck path 
$\mu(B):=U(UR)^{l-1}R$.
We construct a Dyck path of size $N$ by gluing the 
Dyck paths $\mu(B_{i})$ for the blocks $B_{i}$, $1\le i\le m$ of $\pi$.
We introduce the order of blocks as follows:
\begin{align*}
B_{i}<B_{j} \Leftrightarrow \min(B_i)<\min(B_j).
\end{align*}
Note that the order is well-defined since each integer appears once 
in a block of $\pi$, and blocks $B_i$ and $B_j$ are non-crossing.

First, take Dyck paths $\mu(B_1)$ and $\mu(B_2)$. By the order of blocks,
we have $\min(B_1)=1$ and $\min(B_2)\ge2$.
If $\max(B_1)<\min(B_2)$, we concatenate two Dyck paths $\mu(B_1)$ and 
$\mu(B_2)$ from left to right.
If $\max(B_1)>\min(B_2)$ and $i:=\#\{j\in B_1| j<\min(B_2)\}$, we insert the path $\mu(B_2)$
at the vertex between $2i$-th and $2i+1$-th edges in $\mu(B_1)$.
Then, we obtain a large Dyck path corresponding to the union of two blocks
$\{B_1, B_2\}$. We denote the new Dyck path by $\mu(\{B_1,B_2\})$.
Similarly, we insert $\mu(B_3)$ into $\mu(\{B_1,B_2\})$. 
If $\max(B_1\cup B_2)<\min(B_3)$, we concatenate two Dyck paths $\mu(\{B_1,B_2\})$
and $\mu(B_3)$ from left to right.
If $\max(B_1\cup B_2)>\min(B_3)$ and $i=\#\{j\in B_1\cup B_2 | j<\min(B_3)\}$, we
insert the path $\mu(B_3)$ at the vertex between $2i$-th and $2i+1$-th edges in
$\mu(\{B_1,B_2\})$.
We denote by $\mu(\{B_1,B_2,B_3\})$ the new Dyck path.
We continue to insert the Dyck paths $\mu(B_i), 2\le i\le m$, one-by-one into 
$\mu(\{B_1,\ldots,B_{i-1}\})$. 
In this way, we obtain a Dyck path $\mu(\pi)$ of size $N$ from a non-crossing partition $\pi$.

\begin{defn}
\label{def:NCPDyck}
Let $p:=\mu(\pi)$ be a Dyck path as above.
We define a map from a non-crossing partition $\pi\in\mathtt{NC}(n)$ to a Dyck path $p$ of 
size $n$
as $\mathtt{NCPtoDyck}:\pi\mapsto p$.
\end{defn}

\begin{remark}
The map $\mathtt{NCPtoDyck}$ given in Definition \ref{def:NCPDyck} 
is essentially the same as the bijection $\phi_{n}$ 
studied in \cite[Section 2]{Stu13}. 
Therefore, $\mathtt{NCPtoDyck}$ is a bijection between Dyck paths and non-crossing partitions.
\end{remark}

For example, a Dyck path $U^3R^3U^2RU^2R^3$ corresponds to 
the following non-crossing partition of $[7]$:  
$\{\{1,3\},\{2\},\{4,5,7\},\{6\}\}$.
The pictorial presentation of this non-crossing partition 
is given in Figure \ref{fig:NC}.
\begin{figure}[ht]
\begin{tikzpicture}[scale=0.8]
\draw circle(3cm);
\foreach \a in {90,270/7,-90/7,-450/7,-810/7,-1180/7,-1540/7}
\filldraw [black](\a:3cm)circle(1.5pt);
\draw(90:3cm)node[anchor=south]{$1$};
\draw(270/7:3cm)node[anchor=south west]{$2$};
\draw(-90/7:3cm)node[anchor=north west]{$3$};
\draw(-450/7:3cm)node[anchor=north west]{$4$};
\draw(-810/7:3cm)node[anchor=north east]{$5$};
\draw(-1180/7:3cm)node[anchor=north east]{$6$};
\draw(-1540/7:3cm)node[anchor=south east]{$7$};
\draw(90:3cm)to[bend right=40](-90/7:3cm);
\draw(-450/7:3cm)to[bend right=40](-810/7:3cm)to[bend right=30](-1540/7:3cm)to[bend left=30](-450/7:3cm);
\end{tikzpicture}
\caption{The non-crossing partition $\{\{1,3\},\{2\},\{4,5,7\},\{6\}\}$}
\label{fig:NC}
\end{figure}
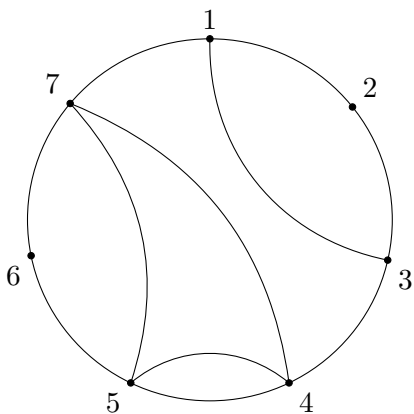
In the pictorial presentation, we connect integers in the same block by diagonals.
However, when a block consists of a single integer, we omit drawing a diagonal line.

We introduce another diagram which we call a {\it chord diagram} to represent a Dyck path.
A chord diagram consists of $2n$ labeled points in line and $n$ non-crossing arches which connect 
two labeled points.
The $2n$ labeled points have labels $1,1',2,2',\ldots,n,n'$ from left to right	.
Let $C$ be a chord diagram. 
We denote by $(i,j)$ if the $i$-th labeled point is connected by an arch with the $j$-th 
labeled point such that $i<j$. 
A Dyck path $P$ corresponding to $C$ is given as follows.
Suppose that $(i,j)$ is an arch. Then, the $i$-th (resp. $j$-th) step in $P$ 
is an up (resp. down) step.
Since we have $n$ arches, we obtain a Dyck path $P$ in this way.

For example, we have the following correspondence:
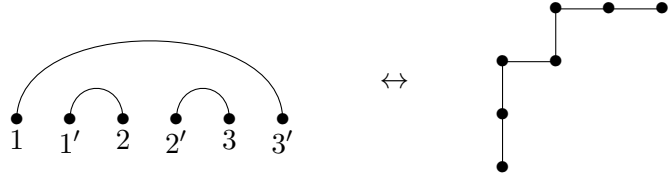
\begin{align*}
\tikzpic{-0.5}{[scale=0.7]
\foreach \x in {0,1,2,3,4,5} \draw(\x,0)node{$\bullet$};
\draw(0,0)node[anchor=north]{$1$}..controls(0,2)and(5,2)..(5,0)node[anchor=north]{$3'$}
(1,0)node[anchor=north]{$1'$}..controls(1,0.8)and(2,0.8)..(2,0)node[anchor=north]{$2$}
(3,0)node[anchor=north]{$2'$}..controls(3,0.8)and(4,0.8)..(4,0)node[anchor=north]{$3$};
}\qquad
\leftrightarrow \qquad
\tikzpic{-0.5}{[scale=0.7]
\foreach \x in {(0,0),(0,1),(0,2),(1,2),(1,3),(2,3),(3,3)}
\draw \x node{$\bullet$};
\draw(0,0)--(0,2)--(1,2)--(1,3)--(3,3); 
}
\end{align*}

Note that an arch in a chord diagram connects a label without a prime with another 
label with a prime.
We obtain a non-crossing partition $\pi$ from a chord diagram in the following way.
Suppose that $(i,j')$ is an arch in a chord diagram. We may have $j<i$ in general.
Then, $i$ and $j$ modulo $n$ belong to the same block. 
Since we have $n$ arches, we have $n$ 
conditions, and this defines the non-crossing partition $\pi$.
In the above example, we have three arches $(1,3'), (1',2)$ and $(2',3)$, which 
gives the non-crossing partition $\{\{1,2,3\}\}$.

Conversely, suppose we have a non-crossing partition $\pi$ consisting of 
$m$ blocks $B_{1},B_{2},\ldots,B_{m}$.
If $B_{i}$ consists of $p$ integers $i_1,\ldots,i_{p}$.
Then, we have $p$ arches which connects $i'_{q}$ and $i_{q+1}$ for $1\le q\le p$, 
where we define $i_{p+1}:=i_{1}$.
We obtain $n$ arches from the $m$ blocks, and the arches give a chord diagram.

As a summary, we have a correspondence:
\begin{align*}
\text{a non-crossing partition} \leftrightarrow \text{a chord diagram}
\leftrightarrow \text{ a Dyck path}.
\end{align*}

By an explicit construction of a chord diagram, the above map from a non-crossing partition
to a Dyck path via a chord diagram coincides with the map $\mathtt{NCtoDyck}$.

\subsection{Non-crossing weighted partitions and \texorpdfstring{$k$}{k}-Dyck paths}
We generalize the bijection between a Dyck path and non-crossing partition of $[N]$
to the case of a $(1,k)$-Dyck path.
The difference between $k\ge2$ case and Dyck paths is that we have to introduce 
the weight on a non-crossing partition.

For this purpose, we introduce the notion of weighted non-crossing partition.
Let $\pi$ and $\pi'$ be two non-crossing partitions of $[N]$.
Then, we write 
\begin{align*}
\pi\le \pi' \Leftrightarrow \pi' \text{ is a refinement of } \pi. 
\end{align*}
Here, ``refinement" means that a block in $\pi$ is obtained by 
merging several blocks of $\pi'$.
Let $\pi_1,\pi_2,\ldots,\pi_{k}$ be $k$ non-crossing partitions of $[N]$
such that 
\begin{align}
\label{eq:condpi}
\pi_1\le \pi_2\le \ldots \le \pi_{k}.
\end{align}
Then, we define $k$ non-crossing partitions $\pi$ of $[N]$ as an ordered set of 
$k$ non-crossing partitions $\pi_1,\ldots,\pi_{k}$ of $[N]$.
We regard each $\pi_{i}$ as the $i$-th layer of $\pi$.
Consider  a circular pictorial presentation of non-crossing partitions.
Suppose that $i$ and $j$ belong to the same block in $\pi_1$. 
Since all the layers satisfy the condition (\ref{eq:condpi}), 
a diagonal $d$ connecting the nodes $i$ and $j$ appear up to the $p$-th 
layer with $p\ge1$.
Then, we say that, in the pictorial presentation, the diagonal $d$ has weight $p$.
 
\begin{remark}
We define a $k$-chain of non-crossing partitions as in Eq. (\ref{eq:condpi}).
In some literature, we write $\pi\le \pi'$ when $\pi$ is a refinement of $\pi'$. 
However, our definition is more consistent with the graphical presentation, which 
we will introduce next.
\end{remark}  
 
\begin{example}
Consider the following $3$-non-crossing partition $\pi$ of $[4]$.
Suppose that the partition $\pi$ consists of three partitions
\begin{align*}
\pi_1=1234, \qquad \pi_2=234/1, \qquad \pi_3=4/23/1. 
\end{align*}
Note that the diagonal connecting $2$ and $3$ has the weight $3$,
and the diagonals connecting $2$ and $4$, and $3$ and $4$ are the weight $2$.
Other diagonals have weight $1$.
The pictorial presentation is given as follows.
The diagonals in black have the weight $1$, the ones in red have the weight $2$,
and the one in blue has the weight $3$.
\begin{figure}[ht]
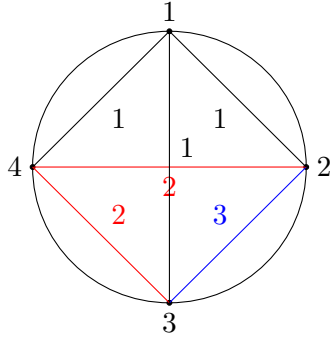

\tikzpic{-0.5}{[scale=0.6]
\draw circle(3cm);
\foreach \a in {0,90,180,270}
\filldraw [black](\a:3cm)circle(1.5pt);
\draw(90:3cm)node[anchor=south]{$1$};
\draw(180:3cm)node[anchor=east]{$4$};
\draw(270:3cm)node[anchor=north]{$3$};
\draw(0:3cm)node[anchor=west]{$2$};
\draw[blue](0:3cm)to[anchor=south east]node{$3$}(270:3cm);
\draw[red](0:3cm)to[anchor=north]node{$2$}(180:3cm);
\draw[red](270:3cm)to[anchor=south west]node{$2$}(180:3cm);
\draw(90:3cm)to[anchor=north east]node{$1$}(0:3cm);
\draw(90:3cm)to[anchor=north west]node{$1$}(180:3cm);
\draw(90:3cm)to[anchor=south west]node{$1$}(270:3cm);
}
\caption{An example of weighted non-crossing partition of $[4]$.}
\label{fig:wNC}
\end{figure}
The condition (\ref{eq:condpi}) uniquely 
fixes the weights of diagonals.
For example, the diagonal connecting $1$ and $3$ has the weight $1$, which 
is fixed by other weights of diagonals.
\end{example}

The total number of $k$-non-crossing partitions is given as follows.
\begin{prop}[\cite{Ede80,Ede82}]
The number of $k$-non-crossing partitions of size $n$ is given by 
the $n$-th Fuss--Catalan number, i.e.
\begin{align}
\#\{\pi_1\le \ldots\le \pi_{k}| \pi_i\in \mathtt{NC}(n)\}
=\genfrac{}{}{}{}{1}{kn+1}\genfrac{(}{)}{0pt}{}{(k+1)n}{n}.
\end{align}
\end{prop}

The number of $k$-non-crossing partitions of size $n$ is the same as the number of 
$k$-Dyck paths of size $n$. 
In Section \ref{sec:NCPDyck}, we introduce a bijection between a Dyck path
and a non-crossing partition of $[N]$.
In what follows, we generalize this correspondence to the case of 
a $(1,k)$-Dyck path and a weighted non-crossing partition.

Let $\pi$ be a weighted non-crossing partition of $[N]$, and 
consider a circular pictorial presentation of $\pi$.
If we ignore the weight of diagonals, we have several blocks 
$B_1,\ldots, B_{h}$ in $\pi$ with $\min(B_i)<\min(B_{i+1})$ for 
$1\le i\le h-1$.
Let $\mu_{i}$ be a $(1,k)$-Dyck path corresponding to 
the block $B_{i}$.
Take $\mu_1$ and $\mu_2$ and set $j:=\#\{p\in B_{1}| p<\min(B_{2})\}$. 
We insert $\mu_2$ between $(k+1)j$-th step and $(k+1)j+1$-th step 
of $\mu_1$. We denote by $\mu_{1\cup2}$ the new $(1,k)$-Dyck path.
Next, take $\mu_3$. Since $\mu_{1\cup2}$ corresponds to the 
merged block $B_1\cup B_{2}$, we have $\min(B_{1}\cup B_2)<\min(B_3)$.
Let $j':=\#\{p\in B_1\cup B_{2}| p<\min(B_3)\}$.
We insert $\mu_3$ between $(k+1)j'$-th step and $(k+1)j'+1$-th step 
of $\mu_{1\cup 2}$.
Then, we have a new $(1,k)$-Dyck path $\mu_{1\cup2\cup3}$.
We continue this process to obtain a new path $\mu_{1\cup\ldots\cup h}$.
We define this path $\mu_{1\cup\ldots\cup h}$ as $(1,k)$-Dyck path 
corresponding to $\pi$.

The remaining task is to assign a $(1,k)$-Dyck path to a weighted block 
$B_{i}$ in $\pi$.
Let $B$ be a block of $[r]$ which consists of diagonals with a unique weight $p\le k$.
The $(1,k)$-Dyck path $\mu(B)$ is defined by
\begin{align*}
\mu(B)=UR^{k-p}(UR^{k})^{r-1}R^{p}.
\end{align*}
Suppose that the block $B$ of $[r]$ can be obtained from $B'$ by increasing the weights of 
diagonals by one.
Let $\mu(B')$ be a $(1,k)$-Dyck path corresponding to $B'$. 
We denote $\mathbf{u}(B')$ be the step sequence, {\it i.e.}, the set of positions of up steps 
in $\mu(B')$.
Then, we define 
\begin{align*}
\mathbf{u}(B)=\{1\}\cup\{u-1: 2\le u\in\mathbf{u}(B')\}.
\end{align*}
The set $\mathbf{u}(B)$ uniquely fixes the $(1,k)$-Dyck path $\mu(B)$.

The map from $\pi$ to $\mu$ is invertible by its construction.
We decompose the paths $\mu$ into small paths, we assign a weighted non-crossing 
partition to these small paths.	

\begin{defn}
We define $\mathtt{NCPtoDyck}$ as a map from a non-crossing weighted partition
to a $k$-Dyck path constructed as above.
\end{defn}

The map $\mathtt{NCPtoDyck}$ depends on $k$, but we omit the dependence of $k$
by abuse of notation. Especially, when $k=1$, the above map coincides with 
the map from a non-crossing partition to a Dyck path given in Definition \ref{def:NCPDyck}.

\begin{example}
We consider the weighted partition for $k=3$ as in Figure \ref{fig:wNC}.
We first consider the partition $23$, which is in the third layer. 
This partition corresponds to the $(1,3)$-Dyck path $UR^2UR^4$.
Next, we consider the partition $\{234,23\}$, which is in the second and the third layers.
The partition $\{234,23\}$ corresponds to the path $URUR^4UR^4$.
Finally, the partition $\{1234,234/1,4/23/1\}$ corresponds to 
the path $UR^2URUR^4UR^5$.
\end{example}

\subsection{Maps on non-crossing partitions}
We introduce five maps on a non-crossing partition. 
They are 
\begin{enumerate}
\item rotation $\mathtt{Rot}$,
\item reflection $\mathtt{Ref}$,
\item Kreweras complement map $\mathtt{Kre}$,
\item Simion--Ullman involution $\mathtt{SU}$,
\item Lalanne--Kreweras involution $\mathtt{LK}$.
\end{enumerate}

\paragraph{(1) Rotation}
The map $\mathtt{Rot}$ is a $360/n$ degree counterclockwise rotation of 
the circular presentation of a non-crossing partition. 
For example, the rotation of the non-crossing partition in Figure \ref{fig:NC}
gives $\{\{1\},\{2,7\},\{3,4,6\},\{5\}\}$.

\paragraph{(2) Reflection}
Recall that a non-crossing partition $\pi$ consists of blocks which are a set of 
integers. The reflection of $\pi$ is to replace $i$ by $n+1-i$ in each block.
For example, the reflection of the non-crossing partition in Figure \ref{fig:NC}
gives $\{\{1,3,4\},\{2\},\{5,7\},\{6\}\}$.

\paragraph{(3) Kreweras complement map}
Let $\pi$ be a non-crossing partition and consider its circular presentation $C(\pi)$.
Recall that the points in $C(\pi)$ are enumerated by $1,2,\ldots,n$ clockwise.
We append new $n$ points on $C(\pi)$ by dividing the interval between two points labeled 
$i$ and $i+1$ for $1\le i\le n-1$, or $n$ and $1$.
We put a label $1'$ on the point between $1$ and $2$, and put labels $i'$ on the $n-1$
remaining new points clockwise.
In $C(\pi)$, the points are connected by diagonals. The diagonals divide the inside of $C(\pi)$
into small regions. When $i'$ and $j'$ are in the same region, we connect them by a diagonal.
In this way, we have a new non-crossing partition $\pi^{c}$ with labels with a prime.
We define $\mathtt{Kre}(\pi):=\pi^{c}$.
For example, the Kreweras complement of the non-crossing partition in Figure \ref{fig:NC}
gives $\{\{1,2\},\{3,7\},\{4\},\{5,6\}\}$.

\paragraph{(4) Simion--Ullman involution}
The map $\mathtt{SU}$ is constructed in a similar manner to $\mathtt{Kre}$.
The difference is we put a new label $1'$ between the points labeled $n$ and $1$, 
and put labels $i'$ counterclockwise.
For example, the Simion--Ullman involution of the non-crossing partition in Figure \ref{fig:NC}
gives $\{\{1,5\},\{2,3\},\{4\},\{6,7\}\}$.

\paragraph{(5) Lalanne--Kreweras involution}
Recall that $\mathtt{Dyck}_2$ is an Lalanne--Kreweras involution on a Dyck path, 
$\widehat{RSK}$ is a bijection between $321$-avoiding permutation and a Dyck path,
and $\mathtt{NCPtoDyck}$ is a bijection between a non-crossing partition and 
a Dyck path.
We define an involution $\mathtt{LK}$ on non-crossing partitions as a 
composition of the bijections:
\begin{align}
\label{eq:defLK}
\mathtt{LK}:=\mathtt{NCPtoDyck}^{-1}\circ\widehat{RSK}\circ\mathtt{Dyck}_2\circ\widehat{RSK}^{-1}\circ\mathtt{NCPtoDyck}
\end{align}
For example, the action of the Lalanne--Kreweras involution on the non-crossing partition in Figure \ref{fig:NC}
is given by $\{\{1,2\},\{3\},\{4,7\},\{5,6\}\}$.

We summarize the properties of the maps.
\begin{prop}
\label{prop:Xinv}
We have 
\begin{gather*}
\mathtt{Kre}^{2}=\mathtt{Rot}, \\
\mathtt{Rot}^{n}=\mathtt{Ref}^{2}=\mathtt{SU}^{2}=\mathtt{LK}^{2}=\mathrm{id}.
\end{gather*}
\end{prop}
\begin{proof}
By construction, we have $\mathtt{Kre}^{2}=\mathtt{Rot}$, and 
$\mathtt{Rot}^{n}=\mathtt{Ref}^{2}=\mathrm{id}$. 
Similarly, we have $\mathtt{SU}^2=\mathrm{id}$ since we put the primed integers in the 
circular presentation counterclockwise. 
Finally, since $\mathtt{Dyck}_2$ is an involution on a Dyck path, we have 
$\mathtt{LK}^2=\mathrm{id}$.
\end{proof}

The relations among five maps are given as follows.
\begin{prop}
\label{prop:5relX}
We have 
\begin{gather}
\mathtt{SU}\circ\mathtt{Rot}=\mathtt{Rot}^{-1}\circ\mathtt{SU}, \\
\mathtt{LK}\circ\mathtt{Rot}=\mathtt{Rot}^{-1}\circ\mathtt{LK}, \\
\mathtt{LK}\circ\mathtt{SU}=\mathtt{Rot},\\
\mathtt{Kre}=\mathtt{Ref}\circ\mathtt{SU}=\mathtt{Ref}\circ\mathtt{LK}\circ\mathtt{Rot}. \\
\mathtt{Kre}\circ\mathtt{SU}=\mathtt{SU}\circ\mathtt{Kre}^{-1}, \\
\mathtt{Kre}\circ\mathtt{LK}=\mathtt{LK}\circ\mathtt{Kre}^{-1}, 
\end{gather}
\end{prop}

To prove Proposition \ref{prop:5relX}, we first study the relations between 
maps on non-crossing partitions and maps on Dyck paths.

Let $X$ be an operation on a non-crossing partition and $Y$ be an operation 
on a Dyck path. 
\begin{defn}
We write a relation between $X$ and $Y$ as  
\begin{align*}
X\sim Y \Leftrightarrow X= \mathtt{NCPtoDyck}^{-1}\circ Y\circ \mathtt{NCPtoDyck}.
\end{align*}
\end{defn}

The next proposition shows that the four maps $\mathtt{Rot}$, $\mathtt{Kre}$, $\mathtt{SU}$, and $\mathtt{LK}$ 
are essentially equivalent to operations on Dyck paths.
\begin{prop}
\label{prop:NCpro}
We have 
\begin{gather}
\mathtt{Kre}\sim\partial, \quad \mathtt{Rot}=\mathtt{Kre}^{2}\sim\partial^{2}, \\
\mathtt{SU}\sim\mathtt{ev}\circ \partial, \\
\mathtt{LK}\sim \mathtt{ev}\circ \partial^{-1}.
\end{gather}
\end{prop}
\begin{proof}
We first show that $\mathtt{Kre}\sim\partial$.
Recall that we have a chord diagram $C(P)$ for a given Dyck path $P$.
On one hand, an arch $(i,j')$ in $C(P)$ is a pair of an integer $i$ and a primed integer $j'$.
The correspondence between a chord diagram and a non-crossing partition, 
the arch $(i,j')$ implies that the two integers $i$ and $j$ belong to the same 
block, and there is no integer $i<k<j$ such that $k$ belongs to the same block
as $i$ and $j$.
If we act the promotion $\partial$ on the Dyck path, or equivalently on the chord diagram, 
the arch $(i,j')$ becomes the arch $((i-1)',j)$ where $i-1$ is understood modulo $n$.  

On the other hand, in the circular representation of a non-crossing partition, 
if $i$ and $j$ with $i<j$ belong to the same block, and there is no $k$ such that $i<k<j$ and 
$i,j$ and $k$ belong to the same block, then $i$ and $j-1$ belong to the same block after the 
action of $\partial$, i.e., we have an arch $((j-1)',i)$ in the chord diagram. 
Here, if $j=1$, then we interpret $j-1$ as $n$. 
Similarly, if $i<j$ belong to the same block and there is no $k$ such that $j<k$ and 
$i,j$ and $k$ belong to the same block, then we have an arc $((i-1)',j)$ in the chord diagram 
after the action of $\partial$.
From these, it is clear that the promotion $\partial$ on Dyck paths coincides with 
the map $\mathtt{Kre}$, which gives $\mathtt{Kre}\sim\partial$.

Since the rotation $\mathtt{Rot}$ is equal to $\mathtt{Kre}^2$, we have $\mathtt{Rot}\sim\partial^2$.
By circular representation of $\mathtt{SU}$, it is clear that $\mathtt{SU}\sim\mathtt{ev}\circ\partial$.

We will prove $\mathtt{LK}\sim\mathtt{ev}\circ\partial^{-1}$.
By definition (\ref{eq:defLK}), $\mathtt{Dyck}_2=\widehat{\delta}\circ\delta$, and 
Propositions \ref{prop:cd1} and \ref{prop:cd2},
we have $\mathtt{LK}\sim\mathtt{ev}\circ\partial^{-1}$.
This completes the proof.
\end{proof}

\begin{proof}[Proof of Proposition \ref{prop:5relX}]
From Proposition \ref{prop:NCpro}, we have 
\begin{align*}
\mathtt{SU}\circ\mathtt{Rot}\sim\mathtt{ev}\circ\partial^{3}=\partial^{-2}\circ\mathtt{ev}\circ\partial\sim\mathtt{Rot}^{-1}\circ\mathtt{SU}, \\
\mathtt{LK}\circ\mathtt{Rot}\sim\mathtt{ev}\circ\partial=\partial^{-2}\circ\mathtt{ev}\circ\partial^{-1}\sim\mathtt{Rot}^{-1}\circ\mathtt{LK}.
\end{align*}
Other relations are proven in a similar way by using $\mathtt{Ref}=\mathtt{ev}$.
\end{proof}

The next proposition shows the relation between the maps on Dyck paths and maps on 
non-crossing partitions.
\begin{prop}
We have 
\begin{align}
\label{eq:D1Kre}
\mathtt{Dyck}_{1}
=\widehat{\mathtt{RSK}}^{-1}\circ\mathtt{NCPtoDyck}\circ\mathtt{Kre}^{-1}\circ\mathtt{NCPtoDyck}^{-1}\circ\widehat{\mathtt{RSK}}, \\
\label{eq:rvD3}
\delta\circ\mathtt{Dyck}_3
=\widehat{\mathtt{RSK}}^{-1}\circ\mathtt{NCPtoDyck}\circ\mathtt{SU}\circ\mathtt{NCPtoDyck}^{-1}\circ\widehat{\mathtt{RSK}},
\end{align}
where $\delta$ is the rowmotion on Dyck paths.
\end{prop}
\begin{proof}
From Propositions \ref{prop:rmD1} and \ref{prop:cd1}, we have 
$\widehat{RSK}\circ\mathtt{Dyck}_1\circ\widehat{RSK}^{-1}\sim\partial^{-1}=\mathtt{Kre}^{-1}$.

From Proposition \ref{prop:cd1} and Corollary \ref{cor:cd3}, we have 
$\widehat{RSK}\circ\delta\circ\mathtt{Dyck}_3\circ\widehat{RSK}^{-1}=\partial^{-1}\circ\mathtt{ev}
=\mathtt{ev}\circ\partial$.
From Proposition \ref{prop:NCpro}, we have $\mathtt{SU}\sim\mathtt{ev}\circ\partial$, 
which implies Eq. (\ref{eq:rvD3}).
\end{proof}

\begin{example}
We consider the Dyck path $P=URU^2RUR^2UR$.
We have 
\begin{align*}
P\xrightarrow{\mathtt{Dyck}_1}U^2RU^2R^2UR^2,
\end{align*}
and 
\begin{align*}
P&\xrightarrow{\widehat{RSK}}U^2RU^2RUR^3\xrightarrow{\mathtt{NCPtoDyck}^{-1}}125/3/4 \\
&\xrightarrow{\mathtt{Kre}^{-1}}1/2/345\xrightarrow{\mathtt{NCPtoDyck}}URURU^2RUR^2
\xrightarrow{\widehat{RSK}^{-1}}24153=U^2RU^2R^2UR^2.
\end{align*}
Similarly, we have 
\begin{align*}
P\xrightarrow{\mathtt{Dyck}_3}URU^3R^3UR\xrightarrow{\delta}U^2R^2URU^2R^2,
\end{align*}
and
\begin{align*}
P&\xrightarrow{\widehat{RSK}}U^2RU^2RUR^3\xrightarrow{\mathtt{NCPtoDyck}^{-1}}125/3/4 \\
&\xrightarrow{\mathtt{SU}}1/234/5\xrightarrow{\mathtt{NCPtoDyck}}URU^2RUR^2UR
\xrightarrow{\widehat{RSK}^{-1}}21354=U^2R^2URU^2R^2.
\end{align*}

\end{example}

It is well-known that the set of non-crossing partitions form a graded lattice.
Let $\pi$ be a non-crossing partition with $m$ blocks. 
The rank function $\mathtt{rk}:\mathtt{NC}(n)\rightarrow \mathbb{Z}_{\ge0}$ 
is given by $\mathtt{rk}(\pi)=n-m$.
We write the cover relation $\pi\lessdot\pi'$ if 
$\pi$ is a refinement of $\pi'$ and $\mathtt{rk}(\pi')=\mathtt{rk}(\pi)+1$.

\begin{prop}
\label{prop:Xorev}
The maps $\mathtt{Kre}, \mathtt{SU}$ and $\mathtt{LK}$ are order-reversing. 
More precisely, we have 
\begin{align}
\label{eq:Xpi1}
\pi\lessdot\pi'
\Leftrightarrow
X(\pi)\gtrdot X(\pi'),
\end{align} 
where $X\in\{\mathtt{Kre},\mathtt{SU},\mathtt{LK}\}$.
Further, we have 
\begin{align}
\label{eq:Xpi2}
\mathtt{rk}(\pi)+\mathtt{rk}(X(\pi))=n-1.
\end{align}
\end{prop}
\begin{proof}
Let $P=P_1\cdots P_{2n}\in\{U,R\}^{2n}$ be a Dyck path, and $\pi$ be the non-crossing partition 
corresponding to $P$.
From Proposition \ref{prop:DyckPMev}, the evacuation $\mathtt{ev}$ on $P$
gives a Dyck path $P'=P'_1\cdots P'_{2n}$ such that $P'_{i}=\overline{P_{2n+1-i}}$ where 
$\overline{U}=R$ and $\overline{R}=U$. 
In terms of non-crossing partitions, $\mathtt{ev}$ corresponds to exchanging 
$i$ and $n+1-i$ in $\pi$.
This implies $\mathtt{ev}$ is order-preserving.
From Proposition \ref{prop:NCpro}, it is enough to show that 
Eqs. (\ref{eq:Xpi1}) and (\ref{eq:Xpi2}) hold for $\mathtt{Kre}\sim\partial$, or 
equivalently $\mathtt{SU}$.
Equations (\ref{eq:Xpi1}) and (\ref{eq:Xpi2}) for $\mathtt{Kre}$ follow 
from \cite{Kre72,SimUll91}.
\end{proof}

\subsection{Maps on non-crossing weighted partitions}
Let $\pi$ be a non-crossing weighted partition, and $\pi_{i}$, $1\le i\le k$, 
be a non-crossing partition of the $i$-th layer.
By definition of non-crossing weighted partitions, $\pi_{i+1}$ is a refinement of 
$\pi_{i}$ for all $i\in[1,k-1]$.
\begin{defn}
\label{defn:XNCWP1}
We define Kreweras complement map $\mathtt{Kre}$, Simion--Ullman involution $\mathtt{SU}$, 
and Lalanne--Kreweras involution $\mathtt{LK}$ on non-crossing weighted partitions by
\begin{align*}
\mathtt{X}:\pi=(\pi_1,\ldots,\pi_{k})\mapsto (\mathtt{X}(\pi_k),\ldots,\mathtt{X}(\pi_1)),
\end{align*}
where $\pi_{i}$ is the non-crossing partition in the $i$-th layer and 
$\mathtt{X}\in\{\mathtt{Kre, Su, LK}\}$.
\end{defn}

\begin{defn}
\label{defn:XNCWP2}
We define refection $\mathtt{Ref}$ and rotation $\mathtt{Rot}$ on non-crossing weighted 
partitions by 
\begin{align*}
\mathtt{X}:\pi=(\pi_1,\ldots,\pi_{k})\mapsto (\mathtt{X}(\pi_1),\ldots,\mathtt{X}(\pi_k)),
\end{align*}
where $\pi_{i}$ is the non-crossing partition in the $i$-th layer and 
$\mathtt{X}\in\{\mathtt{Ref, Rot}\}$.
\end{defn}

\begin{remark}
The action of $\mathtt{X}\in\{\mathtt{Kre, SU, LK}\}$ on a non-crossing weighted partition can be decomposed into the action 
of $\mathtt{X}$ on non-crossing partitions. 
The above definition is well-defined since the action of $\mathtt{X}$ reverses the order of 
layers of the non-crossing partitions from Proposition \ref{prop:Xorev}. 
\end{remark}

By definitions \ref{defn:XNCWP1} and \ref{defn:XNCWP2}, it is obvious that 
the maps $\{\mathtt{Ref, Rot, Kre, SU, LK}\}$ have the same relations as in
Proposition \ref{prop:Xinv} and Proposition \ref{prop:5relX}. 

In the case of non-crossing weighted partitions, we have the 
following relations, which are a generalization of Proposition \ref{prop:NCpro}.
\begin{prop}
\label{prop:kDyckNC}
We have 
\begin{gather}
\label{eq:kDyckNC1}
\mathtt{Rot}\sim\partial^{k+1}, \\
\label{eq:kDyckNC2}
\mathtt{SU}\sim\mathtt{ev}\circ \partial, \\
\label{eq:kDyckNC3}
\mathtt{LK}\sim\mathtt{ev}\circ \partial^{-k}.
\end{gather}
\end{prop}
\begin{proof}
We first show Eq. (\ref{eq:kDyckNC1}) by induction on $n$ and $k$.
From Proposition \ref{prop:NCpro} and simple calculations, 
we have Eq. (\ref{eq:kDyckNC1}) holds for $n=1$ and general $k\ge2$, 
and $k=1$ and general $n\ge2$.
We assume that Eq. (\ref{eq:kDyckNC1}) holds up to $n-1$ and $k-1$.
Let $\pi:=(\pi_1,\cdots\pi_{k})$ be an increasing $k$-chain of non-crossing partitions.
We consider two cases: 1) $\pi_{k}\neq12\cdots n$, and 
2) $\pi_{k}=12\cdots n$. 

Case 1). Since $\pi_{k}\neq12\ldots n$, $\pi_{k}$ consists of at least 
two blocks. 
We write $\pi_{k}=B_1/B_2$ such that $B_{1}$ is a block containing 
the integer $1$, and $B_2$ is a non-crossing partition in $[1,n]\setminus B_1$.
By construction of a weighted non-crossing partition, we insert a $k$-Dyck path corresponding 
to $B_2$ into the $k$-Dyck path corresponding to $B_{1}$.
This implies that the $k$-Dyck path corresponding to $\mathtt{Rot}(\pi)$ can be obtained from 
$\mathtt{Rot}(B_{1})$ and $\mathtt{Rot}(B_{2})$ by inserting the latter into the former.
Further, since $\min(B_2)\ge2$, we have $\min(\mathtt{Rot}(B_2))=\min B_2-1\ge1$. 
When we insert $B_2$ into $B_1$ after the $(k+1)(\min(B_2)-1)$-th edge, we insert 
$\mathtt{Rot}(B_2)$ into $\mathtt{Rot}(B_1)$ after the $(k+1)(\min(\mathtt{Rot}(B_2)-1)$-th edge.
By induction hypothesis, we have $\mathtt{Rot}(B_i)=\partial^{k+1}(P_i)$ for $i=1,2$ where
$P_i$ is a $k$-Dyck path corresponding to $B_{i}$.
From these observations, we have $\mathtt{Rot}(\pi)=\partial^{k+1}(P)$ where $P$ is the 
$k$-Dyck path corresponding to $\pi$.

Case 2). Since $\pi_{k}=12\ldots n$, we have $\pi_{j}=12\ldots n$ for all $1\le j\le k-1$. 
Let $\pi'=(\pi_1,\pi_2,\ldots,\pi_{k-1},1/2/\ldots/n)$ be 
a non-crossing weighted partition of $k$ layers, 
and we denote  a non-crossing weighted partition of $k-1$ layers by $\pi_{-}:=(\pi_1,\ldots,\pi_{k-1})$.
Since $\pi_k=12\ldots n$, the Dyck path $P$ corresponding to $\pi$ can be obtained from the $k$-Dyck 
path $P'$ corresponding to $\pi'$ by adding one unit cell to each row of $P'$.
Recall that a $k$-Dyck path $P$ defines a Young diagram $Y^{(k)}(P)$ above it.
Since $\mathtt{Rot}(12\ldots n)=12\ldots n$, the $k$-Dyck path corresponding to $\mathtt{Rot}(\pi)$ 
is obtained from the $k$-Dyck path corresponding to $\mathtt{Rot}(\pi')$ by adding one unit cell to each row 
of $Y^{(k)}(\mathtt{Rot}(\pi'))$.
By construction of $k$-Dyck path from a non-crossing weighted partition, $Y^{(k)}(\pi')$ is obtained 
from $Y^{(k-1)}(\pi_{-})$ by adding $n-i$ boxes in the $i$-th row from top for $1\le i\le n$.
By induction hypothesis for $\pi_{-}$, we have $\mathtt{Rot}(\pi_{-})=\partial^{k}(P_{-})$ where 
$P_-$ is a $k-1$-Dyck path corresponding to $\pi_-$.
Then, $Y^{(k)}(\mathtt{Rot}(\pi'))$ is obtained from $Y^{(k-1)}(\mathtt{Rot}(\pi_-))$ by adding $n-i$ 
boxes in the $i$-th row from top.
These observations imply that $Y^{(k)}(\mathtt{Rot}(\pi'))$ defines the $k$-Dyck path 
which is given by $\partial^{k+1}(P')$.
It is clear that $k$-Dyck path corresponding to $\mathtt{Rot}(\pi)$ can be obtained from 
$\partial^{k+1}(P')$ by adding one box in each row of $P'$.
From these, $\mathtt{Rot}(\pi)=\partial^{k+1}(P)$.

Secondly, we show Eq. (\ref{eq:kDyckNC2}) by induction on $n$ and $k$.
Let $\pi=(\pi_1,\ldots,\pi_{k})$ be a non-crossing weighted partition,  
and $\pi_{-}=(\pi_{1},\ldots,\pi_{k-1})$.
Let $P$ (resp. $P_{-}$) be $k$-Dyck (resp. $k-1$-Dyck) paths corresponding to $\pi$ and $\pi_{-}$,
and $P'$ (resp. $P'_{-}$) be paths corresponding to $\mathtt{SU}(\pi)$ (resp. $\mathtt{SU}(\pi_{-})$).
We consider two cases: 1) $\pi_{k}=12\cdots n$, and 2) $\pi_{k}\neq12\cdots n$

Case 1). We have $\mathtt{SU}(\pi_{k})=1/2/\ldots/n$. 
Since $\pi_{k}$ is a refinement of $\pi_{i}$, $1\le i\le k-1$, all $\pi_{i}=12\ldots n$.
By the correspondence between $k$-Dyck paths and $k$-chains of non-crossing partitions,
the step sequence of $P$ is given by 
$(1,2,3+k,4+2k,\ldots,n+(n-2)k)$.
The promotion on $P$ gives a $k$-Dyck path whose step sequence is 
$(1,2+k,3+2k,\ldots,n+(n-1)k)$.
The action of $\mathtt{ev}$ on this $k$-Dyck path gives the same Dyck path.
On the other hand, $\mathtt{SU}(\pi)$ corresponds to the Dyck path whose 
step sequence is $(1,2+k,3+2k,\ldots,n+(n-1)k)$.
From these, we have $P'=\mathtt{ev}\circ\partial(P)$.

Case 2). 
We prove the claim by induction on $n$. For $n=1,2$, the claim follows from simple calculations.
We assume that the claim holds up to $n-1$ for $k\ge1$.
Since $\pi_{k}\neq 12\cdots n$, $\mathtt{SU}(\pi_{k})$ has at least two blocks.
Since we have $\mathtt{SU}\circ\mathtt{Rot}=\mathtt{Rot}^{-1}\circ\mathtt{SU}$ by Proposition 
\ref{prop:5relX}, we can assume without of loss of generality 
that a block $B$ of $\pi_{k}$ contains integers $1$ and $i$ such that an integer $j$  
is not contained in $B$ where $i\ge2$ and $1<j<i$.
All layers in $\pi$ contain $1$ and $i$ in the same block, and by the definition of $\mathtt{NCPtoDyck}^{-1}$,
the first two steps in $P$ are up steps, and first $(k+1)(i-1)+1$ steps contain $i$ up steps.
Recall that the perfect matching $S(P)$ consists of $n$ blocks, and each block contains $k+1$ integers.
Further, the minimal integer in a block of $S(P)$ corresponds to the position of an up step. 
In $S(P)$, a block, which contains the $i'$-th up step with $1\le i'\le i$, contains only the integers 
in $[1,(k+1)(i-1)+1]$.  
The action of the promotion $\partial$ on $P$ changes $P$ to $Q$ such that $Q$ touches the line $y=x/k$ 
at the vertex $((i-1)k,i-1)$. This means that $Q$ is a concatenation of two $k$-Dyck paths $Q_1$ and $Q_2$
of size $i-1$ and $n-i+1$.
The action of the evacuation $\mathtt{ev}$ changes $Q$ to $Q'$ such that $Q'$ is expressed as a concatenation 
of two $k$-Dyck paths $Q'_2$ and $Q'_1$ of size $n-i+1$ and $i-1$, i.e., $Q'=Q'_2\circ Q'_1$. 

On the other hand, the action of $\mathtt{SU}$ on $\pi_{k}$ implies that the integers $n+2-i,\cdots,n$ form
a non-crossing weighted partition. 
This means that the $k$-Dyck path $P'$ is written as a concatenation of two $k$-Dyck paths $P'_1$ and $P'_2$
of size $n+i-1$ and $i-1$, i.e., $P'=P'_1\circ P'_2$.
By induction hypothesis, we have $P'_1=Q'_2=\mathtt{ev}\circ\partial(Q_2)$ and 
$P'_2=Q'_1=\mathtt{ev}\circ\partial(Q_1)$, which implies that $P'=\mathtt{ev}\circ\partial(P)$.

Finally, from Proposition \ref{prop:5relX} and Eqs. (\ref{eq:kDyckNC1}) and (\ref{eq:kDyckNC2}), 
we have 
\begin{align}
\mathtt{LK}=\mathtt{Rot}\circ\mathtt{SU}
\sim\partial^{k+1}\circ\mathtt{ev}\circ\partial
=\mathtt{ev}\circ\partial^{-k},
\end{align}
which completes the proof.
\end{proof}

\begin{remark}
Proposition \ref{prop:kDyckNC} states the relations between the maps 
$\{\mathtt{Rot},\mathtt{SU},\mathtt{LK}\}$ and $\{\mathtt{ev},\partial\}$.
However, the Kreweras complement map $\mathtt{Kre}$ for $k\ge2$ cannot be simply expressed 
in terms of $\{\mathtt{ev},\partial\}$.
To see this explicitly, we consider the case $k=2$. 
Then, we have $\mathtt{Kre}^2=\mathtt{Rot}$, and $\mathtt{Rot}\sim\partial^{3}$, which implies 
$\mathtt{Kre}$ cannot be expressed in terms of $\partial$.
\end{remark}

\begin{example}
Set $(n,k)=(4,3)$.
Let $P$ be the $3$-Dyck path whose step sequence is $(1,2,4,7)$.
The non-crossing weighted partition for $P$ is 
$\pi=(1234, 14/23, 14/2/3)$.
The perfect matching of $P$ is given by $\{\{1,14,15,16\},\{2,3,12,13\},\{4,5,6,11\},\{7,8,9,10\}\}$.

The rotation gives $\mathtt{Rot}(\pi)=(1234,12/34,1/2/34)$.
On the other hand, we have the $3$-Dyck path with step sequence $(1,3,8,10)$ as $\partial^{4}(P)$. 

The Simion--Ullman involution gives $\mathtt{SU}(\pi)=(1/234,1/24/3,1/2/3/4)$.
$\mathtt{SU}(\pi)$ corresponds to the $3$-Dyck path $P'$ with step sequence $(1,5,7,8)$.
It is easy to verify that $P'=\mathtt{ev}\circ\partial(P)$. 

We compute the action of the Lalanne--Kreweras involution on $P$ 
by use of Eq. (\ref{eq:defLK}).
As a result, we have $\mathtt{LK}(\pi)=(123/4,13/2/4,1/2/3/4)$.
$\mathtt{LK}(\pi)$ corresponds to the $3$-Dyck path $P''$ with 
step sequence $(1,3,4,13)$.
The path $P''$ is obtained as $P''=\mathtt{ev}\circ\partial^{-3}(P)$.
\end{example}

\subsection{Lift of non-crossing weighted partitions}
We introduce a map on non-crossing weighted partitions which we 
call {\it lift of a non-crossing weighted partition} and 
denote it by $\mathtt{Lift}$.
Let $\pi$ be a non-crossing weighted partitions with at most $k$ layers. 
We first increase all the weights by one. We denote it by $\pi^{\uparrow}$
Then, if there are no edges with weight $k+1$, we define $\mathtt{Lift}(\pi)$
as a new non-crossing weighted partition $\pi^{\uparrow}$.

Below, we consider the case where some edges in $\pi^{\uparrow}$ have the weight $k+1$.
Let $E(\pi^{\uparrow})$ be a set of edges which are not in the same block as edges 
with weight $k+1$.
Suppose that $e$ is an edge with weight $k+1$.
By taking another point, we consider a non-crossing partition $\nu$ of size $3$ 
which contains the edge $e$.
By relabeling the labels of points by $1,2$ and $3$, $\nu$ is isomorphic to a non-crossing partition
of size $3$. Further, we consider only the case where $\nu$ contains only one edge with weight 
$k+1$.
We call the points in $\nu$ $1,2$ and $3$ in increasing order.
We consider the following three cases for $\nu$:
1) the two points $1$ and $2$ are connected by an edge with weight $k+1$,
2) the two points $2$ and $3$ are connected by an edge with weight $k+1$,
and 
3) the two points $1$ and $3$ are connected by an edge with weight $k+1$.
In the case of 1), we delete the edge connecting the two points $1$ and $3$.
In the case of 2), we delete the edge connecting $1$ and $2$.
In the case of 3), we delete the edge connecting $2$ and $3$.

By the above operation, we delete some edges from $\pi^{\uparrow}$.
Then, we delete all the edges with weight $k+1$, and delete some 
edges in $E(\pi^{\uparrow})$ if they are not compatible with the remaining edges
and deleted these edges.

An example of $\mathtt{Lift}$ is shown in Figure \ref{fig:Lift}.
The left non-crossing partition is $\{1234,3/24/1\}$. Then, we increase 
all the weight by one and obtain $\{1234,1234,3/24/1\}$ (see the middle picture).
Consider two non-crossing partitions $\{124,124,24/1\}$ and $\{234,234,3/24\}$.
We delete two edges $12$ and $34$. The diagonal $13$ is not in the same block 
as $24$. Since $12$ and $34$ are deleted, $13$ is also deleted.
Finally, we obtain a non-crossing partition $\{23/14,23/14\}$.

\begin{figure}[ht]
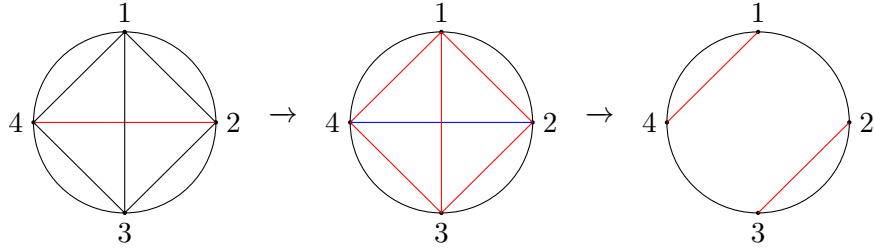

\begin{align*}
\tikzpic{-0.5}{[scale=0.4]
\draw circle(3cm);
\foreach \a in {0,90,180,270}
\filldraw [black](\a:3cm)circle(1.5pt);
\draw(90:3cm)node[anchor=south]{$1$};
\draw(180:3cm)node[anchor=east]{$4$};
\draw(270:3cm)node[anchor=north]{$3$};
\draw(0:3cm)node[anchor=west]{$2$};
\draw(0:3cm)to[anchor=south east](270:3cm);
\draw[red](0:3cm)to[anchor=north](180:3cm);
\draw(270:3cm)to[anchor=south west](180:3cm);
\draw(90:3cm)to[anchor=north east](0:3cm);
\draw(90:3cm)to[anchor=north west](180:3cm);
\draw(90:3cm)to[anchor=south west](270:3cm);
}
\xrightarrow{}
\tikzpic{-0.5}{[scale=0.4]
\draw circle(3cm);
\foreach \a in {0,90,180,270}
\filldraw [black](\a:3cm)circle(1.5pt);
\draw(90:3cm)node[anchor=south]{$1$};
\draw(180:3cm)node[anchor=east]{$4$};
\draw(270:3cm)node[anchor=north]{$3$};
\draw(0:3cm)node[anchor=west]{$2$};
\draw[red](0:3cm)to[anchor=south east](270:3cm);
\draw[blue](0:3cm)to[anchor=north](180:3cm);
\draw[red](270:3cm)to[anchor=south west](180:3cm);
\draw[red](90:3cm)to[anchor=north east](0:3cm);
\draw[red](90:3cm)to[anchor=north west](180:3cm);
\draw[red](90:3cm)to[anchor=south west](270:3cm);
}\xrightarrow{}
\tikzpic{-0.5}{[scale=0.4]
\draw circle(3cm);
\foreach \a in {0,90,180,270}
\filldraw [black](\a:3cm)circle(1.5pt);
\draw(90:3cm)node[anchor=south]{$1$};
\draw(180:3cm)node[anchor=east]{$4$};
\draw(270:3cm)node[anchor=north]{$3$};
\draw(0:3cm)node[anchor=west]{$2$};
\draw[red](0:3cm)to[anchor=south east](270:3cm);
\draw[red](90:3cm)to[anchor=north west](180:3cm);
}
\end{align*}
\caption{Lift on a non-crossing weighted partition with at most $2$ layers.
A diagonal in black, red and blue are in the first, second and third layers respectively.}
\label{fig:Lift}
\end{figure}

\begin{prop}
\label{prop:Lift}
We have $\mathtt{Lift}\sim \partial$.
\end{prop}
\begin{proof}
Let $\pi:=(\pi_1\le \pi_2\le\cdots\le\pi_{k})$ be a non-crossing weighted partition.
We denote by $P$ the $k$-Dyck path corresponding to a non-crossing weighted partition $\pi$.
We consider the two cases 1) $\pi_{k}=12\cdots n$ and 2) $\pi_{k}\neq12\cdots n$.

Case 1). Since $\pi_{i}\le \pi_{k}$ for all $i\in[1,k-1]$, we have 
$\pi_{i}=12\cdots n$ for $i\in[1,k-1]$.
The step sequence $\mathbf{u}(\pi)=(u_1,\ldots,u_{n})$ of $\pi$ is given by 
\begin{align}
u_{i}=\begin{cases}
1, & i=1, \\
2+(k+1)(i-2), & i\ge2.
\end{cases}
\end{align}
The lift of $\pi$ is given by $\pi':=\mathtt{Lift}(\pi)=(\pi_{0},\cdots,\pi_{0})$ where 
$\pi_{0}=1/2/\cdots/n$. The step sequence $\mathbf{u}(\pi')$ of $\pi'$ is given 
by $\mathbf{\pi'}=(1,(k+1)+1,2(k+1)+1,\cdots,(n-1)(k+1)+1)$.
The action of the promotion $\partial$ on $\mathbf{u}(\pi)$ is nothing but 
$\mathbf{u}(\pi')$. Therefore, we have $\mathtt{Lift}\sim\partial$.

Case 2). Since $\pi_k\neq 12\cdots n$, the non-crossing partition $\pi_{k}$ has 
at least two blocks. Since the map $\mathtt{Lift}$ is defined in terms of the circular 
representation of non-crossing weighted partitions, we have 
$\mathtt{Rot}\circ\mathtt{Lift}=\mathtt{Lift}\circ\mathtt{Rot}$.
Without loss of generality, we can assume that $\pi_{k}$ contains a block $B_1$ such that 
$B_1$ contains the integer $1$.
We have two cases: a) $\pi_{k}=1/2/\ldots/n$, and b) the size of $B_1$ is strictly larger than one.

Case 2a). Since the size of all blocks in $\pi_{k}$ is one, we have no edges with label $k+1$ 
after the action of $\mathtt{Lift}$. All the labels in $\pi$ are increased by one,
it is obvious that $\mathtt{Lift}$ is equivalent to the promotion $\partial$. 

Case 2b).
We denote by $\pi_{-}$ the non-crossing weighted partition in $\pi_{k}\setminus B_1$.
Let $i_{\max}$ be the maximal element in the block $B_1$.
Since $B_1$ contains the integer $1$, the step sequence $\mathbf{u}(\pi)=(u_1,\cdots,u_{n})$ 
satisfies $u_1=1$ and $u_2=2$.
Recall the map $\mathtt{PM}$ defined in Definition \ref{defn:PM}. Each up step in a Dyck path 
$P$ has $k$ right steps in the representation of the perfect matching $S(P)$.
The first to the $(k+1)(i_{\max}-1)+1$-th steps in $P$ have $i_{\max}$ up steps and 
$k(i_{\max}-1)$ down steps. This means that the $i$-th up step  with $i\in[2,i_{\max}]$ possesses 
$k$ right steps before the $(k+1)(i_{\max}-1)+2$-th step in $P$.
The $(k+1)(i_{\max}-1)+2$-th step in $P$ is a right step and belongs to the same matching 
as the integer $1$ in $S(P)$.

The non-crossing partition in $\pi_{-}$ may consist of several blocks.
If $\pi_{-}$ contains a block of size larger than one, one can apply the argument similar 
to the case of $B_1$, and identify the positions of up and down steps.

We apply the promotion $\partial$ on $P$. Let $P'=\partial(P)$. 
The above observations imply that the step sequence $\mathbf{u}(P')=(u'_1,\ldots,u'_{n})$
is given by 
\begin{align}
\label{eq:ssudash}
u'_{i}=\begin{cases}
u_{i+1}, & 1\le i\le i_{\max}-1, \\
(k+1)(i_{\max}-1)+2, & i=i_{\max}, \\
u_{i}, & i\ge i_{\max}+1.
\end{cases}
\end{align}

We consider the non-crossing weighted partition $\pi'$ corresponding to the $k$-Dyck path $P'$. 
The step sequence given by Eq. (\ref{eq:ssudash}) implies that 
if $i\in[1,i_{\max}]$ is not contained in $B_1$, then $i-1$ and $i$ belong to the same block
at $k$-th layer. This corresponds to the action of the map $\mathtt{Lift}$ on $\pi$.
Similarly, since $j\in[i_{\max}+1,n]$ is not contained in $B_1$, the action of map $\mathtt{Lift}$
on $\pi$ induces that the integer $j$ belongs to the same block in $i_{\max}$. 
Since the $(k+1)(i_{\max}-1)+2$-th step in $P'$ is an up step by Eq. (\ref{eq:ssudash}), the number of layers 
which contains $j$ and $i_{\max}$ is increased by one compared to that of layers in $\pi$.

From these observations, the $k$-Dyck path $P'$ corresponds to the non-crossing 
weighted partition $\mathtt{Lift}(\pi)$, which completes the proof.
\end{proof}

\begin{example}
Consider the non-crossing partition $\{1234,3/24/1\}$ for $k=2$
in Figure \ref{fig:Lift}.
By the map $\mathtt{NCPtoDyck}$, we have a $2$-Dyck path $p_1$ whose 
step sequence is $1356$. 
The right non-crossing partition corresponds to the $2$-Dyck path $p_2$
whose step sequence is $1245$. 
It is straightforward to see that  $p_2=\partial(p_1)$.
\end{example}

The following corollary is a direct consequence of 
Proposition \ref{prop:NCpro} and Proposition \ref{prop:Lift}.
\begin{cor}
We have $\mathtt{Kre}=\mathtt{Lift}$ on non-crossing partitions if $k=1$.
\end{cor}

\section{Promotion and rowmotion in rational Dyck paths}
\label{sec:prorow}
\subsection{Matching map for \texorpdfstring{$k$}{k}-Dyck paths}
\label{sec:Matk}
We introduce a map called {\it matching map} on $k$-Dyck paths.
Let $p\in\mathtt{Dyck}_{(1,k)}(n)$.
A valley of $p$ is a pattern $RU$, that is, a right step followed 
by an up step.
We denote by $\overline{i}:=N+1-i$, $1\le i\le n$, where $N=(k+1)n$.

In this section, we consider only $k$-Dyck paths, i.e., $(a,b)=(1,k)$.
However, for later, we define an integer sequence $K(p)$
fro an $(a,b)$-Dyck path $p$.
\begin{defn}
\label{defn:Kp}
We define an integer sequence $K(p):=(k_1,\ldots,k_{an})$ for an $(a,b)$-Dyck path $p$ as follows.
\begin{enumerate}
\item Set $i=1$.
\item If the $i$-th row contains a valley, then go to (a). Otherwise, go to (b).
\begin{enumerate}
\item Suppose that the valley is in the $j$-th column from left.
We define $k_{i}:=bn-j+1$.
\item We  define $k_{i}:=\overline{i}$.
\end{enumerate}
\item Increase $i$ by one, and go to (2). Algorithm stops when $i=an$.
\end{enumerate}
\end{defn}
By definition, $K(p)$ is an integer sequence consisting of integers and bared integers.
The length of $K(p)$ is $an$.

We construct a $k$-Dyck path $q$ from $K(p)$ via a perfect matching.	
A perfect matching $S(q):=\{S_1,\ldots,S_{n}\}$, where $S_{i}$ is the set of $k+1$ integers, 
is defined as follows.
\begin{enumerate}
\item Set $i=1$ and $I=[1,(k+1)n]$. 
\item Recall that $k_i$ is an integer with or without a bar.
Take $k+1$ successive decreasing (resp. increasing) integers in $I$ starting from $k_{1}$ 
if $k_1$ is an integer without (resp. with) a bar. 
Here, integers are considered modulo $(k+1)n$.
Then, define $S_{i}$ as the set of these $k+1$ integers.
\item Replace $I$ by $I\setminus S_{i}$, and increase $i$ by one. Then, 
go to (2). The algorithm stops when $i=n$.	
\end{enumerate}
From construction, $S(q)$ is a perfect matching for a $k$-Dyck path.

By applying $\mathtt{PM}^{-1}$ to $S(P)$, we obtain an $(a,b)$-Dyck path $q$.
As a summary, we define as follows.
\begin{defn}
We define a map $\mathtt{Mat}:\mathtt{Dyck}_{(1,k)}(n)\rightarrow\mathtt{Dyck}_{(1,k)}(n)$, 
$p\mapsto q$.	
We call $\mathtt{Mat}$ matching map.
\end{defn}

\begin{remark}
The matching map $\mathtt{Mat}$ coincides with the Armstrong--Stump--Thomas map \cite{ArmStuTho13} on Dyck paths when $k=1$.
\end{remark}

\begin{example}
Consider a path $p$ whose step sequence is $147$ for $(a,b)=(1,2)$ and $n=3$.
We have two valleys in the first and second row, and there is no valley in the third row.
The valley in the first (resp. second) row is in the fourth (resp. second) column.
Thus, we have $K(p)=(3,5,\overline{3})=(3,5,7)$.
Then, we have $S(q):=\{\{3,2,1\},\{5,4,9\},\{7,8,6\}\}$
The perfect matching gives a $2$-Dyck path $q$ whose step sequence is $146$. 
As a consequence, we have $\mathtt{Mat}(147)=146$.
\end{example}

The inverse map $\mathtt{Mat}^{-1}:\mathtt{Dyck}_{(1,k)}\rightarrow\mathtt{Dyck}_{(1,k)}, q\mapsto p$ 
can be constructed as follows.
We consider a perfect matching $P(q)$ obtained from $q$ by $\mathtt{PM}$.
We replace $N+1-i$ in $P(q)$ by $\overline{i}$ for $1\le i\le n$.
We define a function $H$ from $[1,kn]\cup\{\overline{i}:1\le i\le n\}$ to $\mathbb{Z}$ by 
\begin{align}
\label{eq:Hi}
\begin{aligned}
H(i)&:=i, \quad \text{ if } i\in[1,kn], \\
H(\overline{i})&:=ki,\quad \text{ if } i\in[1,n].
\end{aligned}
\end{align}
Recall that $P(q)$ is a collection of sets of integers in $[1,kn]\cup\{\overline{i}:1\le i\le n\}$.
From each set in $P(q)$, we take the integer $i$ such that $H(i)$ is maximum.
Here, if $H(i)=H(\overline{j})$ for some $i\in[1,kn]$ and $j\in[1,n]$, we take $\overline{j}$ rather
than $i$.
We take $n$ elements from $P(q)$, and these $n$ integers are regarded as the set $K(p)$.
It is obvious that $p$ can be recovered from $K(p)$. 
In this way, we have the inverse map $\mathtt{Mat}^{-1}:q\mapsto p$.

\begin{example}
Suppose $(k,n)=(2,3)$ and $q=146$. 
The perfect matching is given by 
\begin{align*}
\{\{1,2,3\},\{4,5,9\},\{6,7,8\}\}=\{\{1,2,3\},\{4,5,\overline{1}\},\{6,\overline{3},\overline{2}\}\}.
\end{align*}
Form the first set, we take $3$  since we have $H(3)=3>H(2)=2>H(1)=1$.
Form the second set, we take $5$ since we have $H(5)=5>H(4)=4>H(\overline{1})=2$.
Finally, from the third set, we take $\overline{3}$ since we have $H(\overline{3})=H(6)=6>H(\overline{2})$.
As a consequence, we have $\{3,5,\overline{3}\}$.
It is easy to see that $p=147$ if $K(p)=\{3,5,\overline{3}\}$.  
\end{example}

\begin{prop}
\label{prop:Matinvk}
The two maps $\mathtt{Mat}$ and $\mathtt{Mat}^{-1}$ are inverse to each other.
\end{prop}
\begin{proof}
Let $P$ be a $k$-Dyck path and $S(P):=\mathtt{PM}(P)$ is the perfect matching of $P$.
It is enough to show that a perfect matching $S(P)$ defines the integer sequence 
$K(P)$ where $K(P)$ consists of integers with or without a bar, 
and conversely $K(P)$ defines $S(P)$.

Suppose that a block $B$ of $S(P)$ contains a bared integer $\overline{i}$, $1\le i\le n$,
and an integer $j$ such that $j\ge ki+1$.
Since the block $B$ contains $\overline{i}$, we already have $i-1$ blocks.
Since each block contains $k+1$ integers, we used $(i-1)(k+1)$ integers for blocks in $S(P)$.
Note that we have $i+j-2\ge i(k+1)-1$. If the block $B$ is represented by the integer $\overline{i}$,
the maximal integer in $B$ is at most $ki$. 
The condition $j\ge ki+1$ means that $\overline{i}$ cannot be a representative of the block $B$, and 
the maximal integer without a bar can be a representative.
This implies that $K(P)$ is obtained from $S(P)$ by use of Eq. (\ref{eq:Hi}). 
By reversing the above correspondence, it is easy to show that $K(P)$ defines $S(P)$. 
This completes the proof.
\end{proof}

\subsection{Matching map for \texorpdfstring{$(a,b)$}{(a,b)}-Dyck paths}
We consider the case where $a$ and $b$ are coprime.
Define $\widetilde{k}:=\lfloor b/a\rfloor$ where $\lfloor x\rfloor$ is the floor function.
For an $(a,b)$-Dyck path $p$, let $K(p)$ be an integer sequence given by Definition \ref{defn:Kp}.
We construct an $(a,b)$-Dyck path $q$ of length $(a+b)n$ from the sequence $K(p)$.

We first introduce the notion of admissibility for a path.
Let $\mathbf{k}=(k_1,\ldots, k_{l})$ be a consecutive increasing integer sequence such that
$1\le k_1$ and $k_l\le (a+b)n$.
We consider the following $(a,b)$-Dyck path $P$ for $\mathbf{k}$: the first step is an up step, and 
the $l$-th step is either an up step or a right step.
If $l=1$ and $k_1\le an+\lfloor b(an-1)/a\rfloor$, 
then we associate an up step to $\mathbf{k}$.
Suppose that the path $P$ has the following expression:
\begin{align*}
Uw_1R^{r_1}w_2R^{r_2}\ldots w_{t}R^{r_t},
\end{align*}
where $w_1,\ldots,w_t$ are $(a,b)$-Dyck paths.
Let $n_i$ (resp. $m_i$) be the number of up (resp. right) steps in $w_i$, $1\le i\le t$.
We say that an $(a,b)$-Dyck path $P$ for $\mathbf{k}$ is a {\it admissible} if 
\begin{align}
\label{eq:adm1}
\begin{aligned}
&\sum_{i=1}^{u}(r_i+m_i)\le \left\lceil b(\sum_{i=1}^{u}n_i+1)/a \right\rceil-1, \quad 1\le u\le t-1, \\
&\sum_{i=1}^{t}(r_i+m_i)=\left\lfloor b(\sum_{i=1}^{t}n_i+1)/a\right\rfloor, \\
&w_i \text{ is admissible for } 1\le i\le t, 
\end{aligned}
\end{align}
where $\lceil x\rceil$ is the ceiling function.
Especially, $P$ is above the line with a slope of $a/b$.

We consider an admissible path which has a single up step.
If $a>b$, then $P=U$ is admissible. If $a\le b$, then $P=UR^{\tilde{k}}$ 
with $\tilde{k}=\lfloor b/a\rfloor$ is also admissible.
Other up-right paths with a single up step are not admissible.

Similarly, we consider an integer sequence 
$\mathbf{k}=(k_1,\ldots,k_{r},k_{r+1},\ldots,k_{l})$ such that 
$k_1=1$, $k_{l}=(a+b)n$, and subsequences $(k_1,\ldots,k_r)$ and $(k_{r+1},\ldots,k_{l})$
are consecutive increasing integer sequences. 
We associate an $(a,b)$-Dyck path $P'$, which consists of two connected components, 
to the sequence $\mathbf{k}$.
The first component starts from the first step and ends at the $k_r$-th step. It is expressed as
\begin{align*}
w_0Uw_1R^{r_1}w_2R^{r_2}\ldots w_{t}R^{r_t},
\end{align*}
where $w_1,\ldots,w_t$ are admissible in the sense of Eq. (\ref{eq:adm1}).
The second component starts from the $k_{r+1}	$-th step and ends with the $k_l$-th step.
It is expressed as 
\begin{align*}
R^{r_{t+1}}w_{t+1}R^{r_{t+2}}w_{t+2}\ldots R^{r_s}w_{s},
\end{align*}
where the paths $w_{t+1},\ldots,w_{s-1}$ are admissible in the sense of Eq. (\ref{eq:adm1}).
Let $n_i$ (resp. $m_i$) be the number of up (resp. right) steps in $w_{i}$, $0\le i\le s$.
We say that the path $P'$ consisting of two components is {\it admissible} if 
the following conditions are satisfied: 
\begin{align}
\label{eq:adm2}
\begin{aligned}
&\sum_{i=1}^{u}(r_i+m_i)\le \left\lceil b(1+\sum_{i=1}^{u}n_i)/a\right\rceil-1, \quad 1\le u\le t, \\
&\sum_{i=1}^{u}r_i+\sum_{i=1}^{u-1}m_i+\left\lfloor b(an-\sum_{i=0}^{s}n_i-1)/a\right\rfloor
\le \left\lceil b(an-n_{0}-\sum_{i=u}^{s}n_i)/a \right\rceil-1, \quad t+1\le u\le s-1,\\
&\sum_{i=1}^{s}r_i+\sum_{i=1}^{s-1}m_i=\left\lfloor b(an-n_0-n_{s})/a\right\rfloor
-\left\lfloor b(an-\sum_{i=0}^{s}n_i-1)/a\right\rfloor, \\
&\text{The path for }(k_1,\ldots,k_{n_0+m_0},k_{l-n_s-m_s+1},\ldots,k_{l}) \text{ is admissible},
\end{aligned}
\end{align}
and $P'$ is above the line with a slope of $a/b$.

We map the integer sequence $K(p)=(k_1,\ldots, k_{an})$ to a perfect matching $S(q):=\{S_1,\ldots,S_{an}\}$.
In what follows, we regard integers modulo $(a+b)n$. That is, the maximal integer which is smaller than $1$
is $(a+b)n$, and the minimal integer which is larger than $(a+b)n$ is $1$.
The perfect matching $S(q)$ is obtained from $K(p)$ as follows.
\begin{enumerate}
\item Set $i=1$ and $I=[1,(a+b)n]$. 
\item We have two cases: (a) $k_i$ has no bar, and (b) $k_i$ has a bar. 
\begin{enumerate}
\item $k_i$ has no bar. 
We take $\tilde{k}+1$ consecutive decreasing integers $\mathbf{k}$ starting from $k_i$ in $I$. 
Let $k_{\min}$ be the minimal integer in the sequence $\mathbf{k}$. 
We have three cases:
\begin{enumerate}
\item 
If $k_i-k_{\min}=\tilde{k}$ and $k_{\min}-1\in I$, 
we define a block of the perfect matching $S_{i}=\{k_{\min},k_{\min}+1,\ldots,k_{i}\}$. 
\item 
 $k_{i}-k_{\min}=\tilde{k}$ and $k_{\min}-1\not\in I$. 
This means that $k_{\min}-1$ belong to a block $S_{j}$ of the perfect matching for some 
$j$ satisfying $1\le j\le i-1$.
In each block $S_{l}$, $1\le l\le i-1$, of the perfect matching, we associate an up step 
to the minimal integer, and right steps to the remaining integers.
Similarly, we associate an up step to the minimal integer $k_{\min}$ and right steps 
to the other remaining integers in $\mathbf{k}$.
If the up-right path for the integers in  
$\mathbf{k}$ and blocks $S_{l}$, $1\le l\le i-1$, is not admissible 
in the sense of Eq. (\ref{eq:adm1}) or (\ref{eq:adm2}), 
we add an integer $k'=\max\{j<\min\{\mathbf{k}\}: j\in I\}$ to $\mathbf{k}$.
We consider the up-right path for the integers in the new $\mathbf{k}$ and blocks $S_l$. 
If this path is not admissible, we continue to add integers one-by-one to $\mathbf{k}$ until we have an admissible 
$(a,b)$-Dyck path. 
By the above procedures, we have the integer sequence $\mathbf{k}$ corresponding to an admissible path.
Finally, even if the up-right path for the integer sequence $\mathbf{k}$ is admissible, we 
try adding an integer $k'$ to $\mathbf{k}$, and check if an up-right path is admissible or not.
We continue this process as far as we have an admissible path. 
Let $q'$ be the admissible path of maximum size.
We take $\mathbf{k}$ as a set of integers which yields the path $q'$.
We define $S_{i}=\mathbf{k}$.
\item 
$k_i-k_{\min}>\tilde{k}$. As in the case (ii), we add integers one-by-one to $\mathbf{k}$
until we have an admissible path of maximum size. Then, we define $S_i=\mathbf{k}$.
\end{enumerate}
\item $k_i$ has a bar. We have $1\le k_i\le an$. Let $k'_{i}=(a+b)n-k_i+1$.
We take $\tilde{k}+1$ consecutive increasing integers $\mathbf{k}$ starting from $k'_i$ in $I$.
As in the case (a), we prolong the sequence $\mathbf{k}$ by adding integers if necessary until 
we have an admissible sequence of maximum size. 
The difference from case (a) is that we take consecutive increasing integers in $I$ instead of decreasing integers. 
We define $S_{i}=\mathbf{k}$ where $\mathbf{k}$ is an integer sequence for the admissible path of maximum size.
\end{enumerate}
\item 
In (2), we define a block of the perfect matching as $S_i=\mathbf{k}$.
We replace $I$ by $I\setminus S_{i}$, and increase $i$ by one. 
Then, go to (2). The algorithm stops when $I=\emptyset$.
\end{enumerate}

We use the admissibility conditions (\ref{eq:adm1}) and (\ref{eq:adm2}) to obtain a block 
of the perfect matching $S(q)$.
The meaning of the admissibility is to take a block with maximum size starting from 
the integer $k_i$ or $\overline{k_i}$.

\begin{defn}
We define the matching map $\mathtt{Mat}:\mathtt{Dyck}_{(a,b)}(n)\rightarrow\mathtt{Dyck}_{(a,b)}(n)$ 
by $p\mapsto q$.
\end{defn}

\begin{example}
Consider the $(2,3)$-Dyck path $p$ of size $2$ whose step sequence is given by
$\mathbf{u}(p)=(1,3,5,6)$. Note that $\widetilde{k}=1$.
Then, we have $K(p)=(\overline{1},5,6,\overline{4})=(10,5,6,7)$.
First, we take $2$ consecutive increasing integers from $I=[1,10]$, which are 
$\{10,1\}$. This sequence is not admissible by Eq. (\ref{eq:adm2}). 
We prolong $\{10,1\}$ to $\{10,1,2\}$. 
The perfect matching $\{10,1,2\}$ is admissible by Eq. (\ref{eq:adm2}), 
and $\{10,1,2,3\}$ is not.
Therefore, we have $S_{1}=\{10,1,2\}$.
Secondly, we take $2$ consecutive decreasing integers starting from $5$ in $I=[1,10]\setminus S_1$, 
which are $\{5,4\}$. This sequence is admissible by Eq. (\ref{eq:adm1}), and 
$\{5,4,3\}$ is not admissible.
We have $S_{2}=\{5,4\}$.
Thirdly, we take $2$ consecutive decreasing integers starting from $6$ in $I\setminus\{S_1,S_2\}$, 
which are $\{6,3\}$. The integers in the perfect matching $S_2$ are in $[3,6]$.
We consider an $(2,3)$-Dyck path $UURR$ corresponding to $[3,6]$. 
This Dyck path is not admissible, hence we prolong the sequence by adding an integer. 
The new sequence is $\{6,3,9\}$, and the path for $[1,6]\cup[9,10]$ is admissible by Eq. (\ref{eq:adm2}).
The sequence $\{6,3,9,8\}$ is not admissible by Eq. (\ref{eq:adm2}).
Hence, we have $S_{3}=\{6,3,9\}$.
Fourth, we take $2$ consecutive increasing integers from $\bar{4}=7$ 
in $I\setminus\{S_1,S_2,S_3\}$. This gives an admissible sequence $\{7,8\}$.
As a result, we have a perfect matching $S(q)=\{\{1,2,10\},\{3,6,9\},\{4,5\},\{7,8\}\}$.
The step sequence for the $(2,3)$-Dyck path $q$ is $\mathbf{u}(q)=(1,3,4,7)$, which 
means $\mathtt{Mat}((1,3,5,6))=(1,3,4,7)$.
\end{example}

\begin{example}
We have an obvious bijection between $(a,b)$-Dyck paths and $(b,a)$-Dyck paths.
We consider $(1,2)$- and $(2,1)$-Dyck paths for $n=2$.
The bijection is given by
\begin{align*}
(1,4)\leftrightarrow(1,2,4,5),\qquad (1,3)\leftrightarrow(1,2,3,5), \qquad (1,2)\leftrightarrow(1,2,3,4),
\end{align*}
where we express Dyck paths in terms of step sequences.
 
The map $\mathtt{Mat}$ does not compatible with this bijection in general.
The matching map on $(1,2)$-Dyck paths gives 
\begin{align*}
(1,4)\rightarrow(1,4), \qquad (1,3)\rightarrow(1,2)\rightarrow(1,3).
\end{align*}
On the other hand, the matching map on $(2,1)$-Dyck paths gives
\begin{align*}
(1,2,4,5)\rightarrow(1,2,3,4)\rightarrow(1,2,3,5)\rightarrow(1,2,4,5).
\end{align*}
\end{example}

The inverse map $\mathtt{Mat}^{-1}:\mathtt{Dyck}_{(a,b)}\rightarrow\mathtt{Dyck}_{(a,b)}, q\mapsto p$ 
can be constructed in a similar manner to the case $(a,b)=(1,k)$.
We consider a perfect matching $P(q)$ obtained from $q$ by $\mathtt{PM}$.
We replace $N+1-i$ in $P(q)$ by $\overline{i}$ for $1\le i\le an$.
We define a function $H$ from $[1,bn]\cup\{\overline{i}:1\le i\le an\}$ to $\mathbb{Z}$ by 
\begin{align}
\label{eq:Hab}
\begin{aligned}
H(i)&:=i, \quad \text{ if } i\in[1,bn], \\
H(\overline{i})&:=\lceil bi/a\rceil,\quad \text{ if } i\in[1,an],
\end{aligned}
\end{align}
where $\lceil x\rceil$ is the ceiling function.
Recall that $P(q)$ is a collection of sets of integers in $[1,bn]\cup\{\overline{i}:1\le i\le an\}$.
From each set in $P(q)$, we take the integer $i$ such that $H(i)$ is maximum.
Here, if $H(i)=H(\overline{j})$ for some $i\in[1,bn]$ and $j\in[1,an]$, we take $\overline{j}$ rather
than $i$.
We take $n$ elements from $P(q)$, and these $n$ integers are regarded as the set $K(p)$.
It is obvious that $p$ can be recovered from $K(p)$. 
In this way, we have the inverse map $\mathtt{Mat}^{-1}:q\mapsto p$.

\begin{example}
Suppose that $q$ is a $(2,3)$-Dyck path of size $2$ whose step sequence is 
$\mathbf{u}(q)=(1,3,4,7)$.
By applying $\mathtt{PM}$, we have 
\begin{align*}
S(q)=\{\{1,2,10\},\{3,6,9\},\{4,5\},\{7,8\}\}
=\{\{1,2,\overline{1}\},\{3,6,\overline{2}\},\{4,5\},\{\overline{4},\overline{3}\}\}.
\end{align*}
We take $\overline{1}$ from the first set since $H(\overline{1})=H(2)=2>H(1)=1$.
We take $6$ from the second set since $H(6)=6>H(\overline{2})=H(3)=3$.
We take $5$ and $\overline{4}$ from the third and fourth sets since $H(5)=5>H(4)=4$ and 
$H(\overline{4})=6>H(\overline{3})=5$.
This gives $K(p)=\{\overline{1},6,5,\overline{4}\}$, which yields the 
$(2,3)$-Dyck path $p$ with the step sequence $\mathbf{u}(p)=(1,3,5,6)$.
\end{example}

By an argument similar to the proof of Proposition \ref{prop:Matinvk}, we have the 
following proposition.
\begin{prop}
The maps $\mathtt{Mat}$ and $\mathtt{Mat}^{-1}$ on $(a,b)$-Dyck paths are inverse to each other.
\end{prop}

\subsection{Equivalence of promotion and rowmotion on rational Dyck paths}
In this section, we show that the matching map $\mathtt{Mat}$ for
$(a,b)$-Dyck paths satisfies the commutative diagram similar to Proposition \ref{prop:cd1}.
From Theorem \ref{thrm:RSKMat}, the map $\widehat{RSK}$ can be expressed in terms 
of the matching map. For general $(a,b)$, we have no analogue of $\widehat{RSK}$.
However, the matching map $\mathtt{Mat}$ plays a role similar to $\widehat{RSK}$ for
$(a,b)$-Dyck paths.

\begin{defn}
Let $X$ and $Y$ be operations on rational Dyck paths.
We write a relation between $X$ and $Y$ as 
\begin{align*}
X\cong Y \Leftrightarrow X=\mathtt{Mat}\circ Y\circ\mathtt{Mat}^{-1}.
\end{align*}
If $X\cong Y$, we say that $X$ is congruent to $Y$ up to $\mathtt{Mat}$.
\end{defn}

On one hand, we have two operations on Dyck paths, the promotion $\partial$
and the evacuation $\mathtt{ev}$.
On the other hand, we have two kinds of operations on the poset of Dyck paths, 
the rowmotion $\delta$ and the rowvacuations $\mathtt{Rvac}$ and 
$\mathtt{DRvac}$.
These two families of operations are related through the matching 
map.

\begin{theorem}
\label{thrm:rowMat}
The rowmotion is congruent to promotion up to $\mathtt{Mat}$, 
that is, $\mathtt{Mat}\circ \delta=\partial^{-1}\circ\mathtt{Mat}$.
\end{theorem}
\begin{proof}
We prove the claim by induction on the number $N$ of up steps.
For $N=1$, we have at most one $(a,b)$-Dyck path and the claim follows.

We first consider the case $N=2$.
An $(a,b)$-Dyck path $P$ for $N=2$ consists of two up steps and several down steps.
A path $P$ define a Young diagram $Y(P)$ above $P$ and below $U^{2}R^{m}$ with $m=\lfloor 2b/a\rfloor$.
We denote by $|Y(P)|$ the number of boxes in $Y(P)$.
We consider two cases: 1) $|Y(P)|=0$, and 2) $1\le |Y(P)|\le \lfloor b/a\rfloor$. 

Case 1). Since $|Y(P)|=0$, the step sequence of $P$ is $(1,2)$. 
We traces the step sequences of $(a,b)$-Dyck paths.
On one hand, we have 
\begin{align*}
(1,2)\xrightarrow{\delta}(1,\lfloor b/a\rfloor+2)\xrightarrow{\mathtt{Mat}}(1,2).
\end{align*}
On the other hand, we have 
\begin{align*}
(1,2)\xrightarrow{\mathtt{Mat}}(1,\lfloor b/a\rfloor+2)\xrightarrow{\partial^{-1}}(1,2).
\end{align*}
From these observations, the claim follows.

Case 2). Let $l=|Y(P)|+2$. The integer $l$ satisfies $3\le l\le \lfloor b/a\rfloor+2$.
As in the Case 1), we trace the step sequences of $(a,b)$-Dyck paths.
We have 
\begin{align*}
(1,l)\xrightarrow{\delta}(1,l-1)\xrightarrow{\mathtt{Mat}}(1,\lfloor b/a\rfloor-l+3),
\end{align*}
and 
\begin{align*}
(1,l)\xrightarrow{\mathtt{Mat}}(1,\lfloor b/a\rfloor-l+2)\xrightarrow{\partial^{-1}}(1,\lfloor b/a\rfloor-l+3).
\end{align*}
From these, the claim follows.

We assume that the claim holds up to $N-1$.
Let $P$ be an $(a,b)$-Dyck path with $N$ up steps, and $K(P):=(k_1,\ldots,k_N)$ be a sequence 
of integers with or without a bar defined in Definition \ref{defn:Kp}.
We consider the two cases: A) $k_1=\overline{1}$, and B) $k_{1}=r$ with some integer $r\ge2$.

Case A). Since $k_1=\overline{1}$, $K(\delta(P)):=(k'_1,\ldots,k'_{N})$ is given by 
$k'_1=\lfloor Nb/a\rfloor-\lfloor (N-1)b/a\rfloor+1$,
and we write $(k'_2,\ldots,k'_{N})=\delta((k_2,\ldots,k_{N}))$.
The action of $\mathtt{Mat}$ on $K(\delta(P))$ gives 
a matching block corresponding to $k'_1$. 
For other matching blocks, we write $\mathtt{Mat}\circ\delta(k_2,\ldots,k_{N})$.

The action of $\partial^{-1}$ on $K(P)$ increases (resp. decreases) the integers with (resp. without)
a bar. The newly obtained integer sequence $K(\partial^{-1}(P))$ may not be admissible as 
a sequence for an $(a,b)$-Dyck path $\partial^{-1}(P)$. 
In this case, by considering the perfect matching of $\partial^{-1}(P)$, we pick 
a maximal element for each matching block by Eq. (\ref{eq:Hab}).
By definition of the matching map, the matching block corresponding to $k_1=\overline{1}$
contains $z:=\lfloor Nb/a\rfloor-\lfloor(N-1)b/a\rfloor+1$ increasing integers starting from $\overline{1}$.
By increasing integers by one, we have a matching block which contains the integers $[1,z]$ without a bar.
As a result, the integer sequence $K(\partial^{-1}(P)):=(k''_1,\ldots,k''_{N})$ can be written as 
$k''_1=z$ and $(k''_2,\ldots,k''_{N})$ is expressed as $\partial^{-1}(k_2,\ldots,k_{N})$.

The induction hypothesis implies that we have $\mathtt{Mat}\circ\delta(k_2,\ldots,k_{N})=\partial^{-1}(k_2,\ldots,k_{N})$.
From the previous two paragraphs, $\mathtt{Mat}\circ\delta(P)$ and $\partial^{-1}\circ\mathtt{Mat}(P)$ have 
the same matching block consisting of $z$ integers $[1,z]$, and the remaining blocks are the same by induction.
From these, the claim follows.  

Case B). We consider the three cases: 
i) $\lfloor Nb/a\rfloor-\lfloor (N-1)b/a\rfloor+1\le r\le \lfloor Nb/a\rfloor-\lfloor (N-2)b/a\rfloor$,
ii) $r\ge\lfloor Nb/a\rfloor-\lfloor (N-2)b/a\rfloor+1$ and $k_2$ does not have a bar, and
iii)  $r\ge\lfloor Nb/a\rfloor-\lfloor (N-2)b/a\rfloor+1$ and $k_2=\overline{2}$.

Case i). The sequence $K(\delta(P))=(k'_1,\ldots,k'_{N})$ satisfies
$(k'_1,k'_2)=(\overline{1},r+1)$ if $r$ is maximal and $k_2=\overline{2}$, or $k'_1=r+1$ otherwise. 
Note that the integer $r+1$ without a bar appears in the sequence $K(\delta(P))$.
The other elements $(k'_1,\ldots,k'_{N})\setminus \{r+1\}$ are given by $K(Q)$ where an $(a,b)$-Dyck path $Q$ 
of size $N-1$. If we write $(k'_1,\ldots,k'_{N})\setminus \{r+1\}=(\widetilde{k'_1},\ldots,\widetilde{k'_{N-1}})$ and 
$K(Q)=(l_{1},\ldots,l_{N-1})$, we have 
\begin{align}
l_i=\begin{cases}
\overline{i}, & \text{ if } \widetilde{k'_{i}} \text{ has a bar}, \\
k'_{i+1}-\lfloor Nb/a\rfloor+\lfloor (N-1)b/a\rfloor, & \text{ if }  \widetilde{k'_{i}} \text{ does not have a bar}. 
\end{cases}
\end{align}
The map $\mathtt{Mat}$ on $\delta(P)$ gives a matching block $B$ starting from $k'_{1}=r+1$, and 
other matching blocks are obtained as $\mathtt{Mat}\circ\delta(Q)$. Then, to have matching blocks for 
$\mathtt{Mat}\circ\delta(P)$, we insert the matching block $B$ into those of $\mathtt{Mat}\circ\delta(Q)$.

The promotion $\partial^{-1}$ on $K(P)$ increases the integers in $K(P)$ by one, and decreases 
the integers with a bar in $K(P)$ by one. Since $k_1=r$, the action of $\partial^{-1}$ yields
the matching block starting from $r+1$. Then, it is clear that the elements of $\partial^{-1}(K(P))$
except the first one are given by $K(\partial^{-1}Q)$.
By induction hypothesis, we have $\mathtt{Mat}\circ\delta(Q)=\partial^{-1}\circ\mathtt{Mat}(Q)$, which implies
that $\mathtt{Mat}\circ\delta=\partial^{-1}\circ\mathtt{Mat}$ on an $(a,b)$-Dyck path $P$.

Case ii). The sequence $K(\delta(P))=(k'_1,\ldots,k'_{N})$ satisfies $k'_1=r+1$. Since $k_2$ does 
not have a bar, $K(\mathtt{Mat}\circ\delta(P))\setminus\{r+1\}$ is given by $K(\mathtt{Mat}\circ\delta(Q))$ 
where $Q$ is an $(a,b)$-Dyck path as in the case of i).
By an argument similar to Case i), we have $\mathtt{Mat}\circ\delta=\partial^{-1}\circ\mathtt{Mat}$ 
on an $(a,b)$-Dyck path $P$.

Case iii).
We assume that the integers $k_2$ to $k_{m}$ have a bar for $m\ge2$. 
By the action of $\delta$ on $P$, the integer sequence $K(\delta(P))=(k'_1,\ldots,k'_{N})$ has an integer $p$
such that $k'_{p}=r+1$, $2\le p\le m$, and the integers $k'_{q}$ has a bar for $1\le q<p$. 
Especially, we have $k'_1=\overline{1}$.
By the matching map $\mathtt{Mat}$, we have a matching block which contains $\overline{1}$.

We express the other integers $K':=(k'_2,\ldots,k'_{N})$ in terms of an $(a,b)$-Dyck path of size $N-1$.
More precisely, $K'$ is identified with the path $\mathtt{Mat}\circ\delta(P')$
where $K(P'):=(\widetilde{k_1},\ldots,\widetilde{k_{N-1}})$ satisfies
\begin{align}
\label{eq:KPdash}
\widetilde{k_i}=\begin{cases}
r-M, & \text{ if } i=1,\\
\overline{i}, & \text{ if } i\ge2 \text{ and } k_{i+1} \text{ has a bar},\\
k_{i+1}-M, & \text{ if } i\ge2 \text{ and } k_{i+1} \text{ does not have a bar},	
\end{cases}
\end{align}
where $M=\lfloor Nb/a\rfloor-\lfloor (N-1)b/a\rfloor$.

We consider the action of $\partial^{-1}$ on $K(P)$. Since $k_1=r$, then we have $L=\partial^{-1}(K(P))=(l_1,\ldots,l_{N})$ 
such that $l_1=r+1$. Since $k_2=\overline{2}$, the action of $\partial^{-1}$ yields a matching block which contains $\overline{1}$.
Then the remaining integers $(l_1,l_3,\ldots,l_{N})$ are expressed as 
$\partial^{-1}(k_1,k_3,\ldots,k_{N})$. 
We claim that $\partial^{-1}(k_1,k_3,\ldots,k_{N})$ is $\partial^{-1}(P')$ where $P'$ satisfies Eq. (\ref{eq:KPdash}).

To show the claim, note that $\partial^{-1}$ increases the integer without a bar by one, and decreases the integer
with a bar by one. Especially, $\overline{1}$ is mapped to $1$ and $\lfloor Nb/a\rfloor$ is mapped to $\overline{N}$. 
Because of increments or decrements of integers, the integer $\partial^{-1}(k_{i})$ may not be a maximal 
integer in a matching block with respect to the function $H$ given by Eq. (\ref{eq:Hab}).
We replace $\partial^{-1}(k_{i})$ with the maximal element of each matching block. 
Then, we have an $(a,b)$-Dyck path $P'$ of size $N-1$ which satisfies Eq. (\ref{eq:KPdash}).  

The integer sequences $K(\mathtt{Mat}\circ\delta(P))$ and $K(\partial^{-1}\circ\mathtt{Mat}(P))$ 
have the same matching block containing the integer $\overline{1}$ with a bar. 
The other matching blocks are the same 
by the induction hypothesis, which implies that the two integer sequences coincide with each other.
As a result, we have $\mathtt{Mat}\circ\delta(P)=\partial^{-1}\circ\mathtt{Mat}(P)$,
which completes the proof.
\end{proof}

\begin{example}
We consider the $(2,3)$-Dyck path whose step sequence is $(1,3,5,6)$.
Then, we have the following diagram for the $(2,3)$-Dyck paths:
\begin{align}
\tikzpic{-0.5}{[scale=0.8]
\node (0) at (0,0){$(1,3,5,6)$};
\node (1) at (4,0){$(1,2,4,8)$};
\node (2) at (0,-2){$(1,3,4,7)$};
\node (3) at (4,-2){$(1,2,5,8)$};
\draw[->,anchor=south] (0) to node {$\delta$}(1);
\draw[->,anchor=south] (2) to node {$\partial^{-1}$}(3);
\draw[->,anchor=east] (0) to node {$\mathtt{Mat}$}(2);
\draw[->,anchor=west] (1) to node {$\mathtt{Mat}$}(3);
}
\end{align}
\end{example}

\subsection{Rowmotion on \texorpdfstring{$k$}{k}-Dyck paths}
Recall that a $k$-Dyck path is a path from $(0,0)$ to $(kn,n)$.
A $k$-Dyck path $P$ defines a Young diagram above $P$ and below the top $k$-Dyck path.
To define the rowmotion on $k$-Dyck paths, we need to introduce a poset structure
on Young diagrams.
Let $Y=((n-1)k,(n-2)k,\ldots,k)$ be a Young diagram above the lowest $k$-Dyck path.
We define the rank of the leftmost cell in the top row of $Y$ is $(n-1)k-1$.
Then, we define the rank $\mathtt{rk}(c)$ of the cell $c\in Y$ as 
$\mathtt{rk}(c)=\mathtt{rk}(c')-1$ where $c'$ is a cell just above or right to $c$.
\begin{figure}[ht]
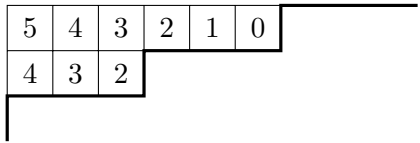

\tikzpic{-0.5}{[scale=0.6]
\draw[very thick](0,0)--(0,1)--(3,1)--(3,2)--(6,2)--(6,3)--(9,3);
\draw(0,1)--(0,3)--(7,3)(1,1)--(1,3)(2,1)--(2,3)(3,2)--(3,3)(4,2)--(4,3)(5,2)--(5,3)(0,2)--(5,2);
\foreach \x in {4,3,2} \draw(4.5-\x,1.5)node{$\x$};
\foreach \x in {5,4,3,2,1,0} \draw(5.5-\x,2.5)node{$\x$};
}
\caption{The rank of cells for $(1,3)$-Dyck paths of size three.}
\label{fig:rkcell}
\end{figure}
Figure \ref{fig:rkcell} gives the rank of cells for $3$-Dyck paths of size three.
Suppose two cells $c_1,c_2\in Y$. 
The cell $c_2$ covers $c_1$ if and only if $\mathtt{rk}(c_2)=\mathtt{rk}(c_1)+1$,
and two cells $c_2$ and $c_1$ shares an edge.
Then, one can apply the general theory developed in Section \ref{sec:GTrow} to this poset.

Recall that the rowmotion $\delta$, the rowvacuation $\mathtt{Rvac}$, and 
the dual rowvacuation $\mathtt{DRvac}$ are expressed by use of filter toggles 
$\mathbf{t}_{0},\mathbf{t}_1,\ldots,\mathbf{t}_{r}$.
For example, let $P$ be a $3$-Dyck path whose step sequence is given by $(1,3,6)$.
Then, we have 
\begin{align*}
\mathbf{t}_{i}(P)=
\begin{cases}
P, & \text{ if } i=0,1,5, \\
(1,3,7), & \text{ if } i=2, \\
(1,4,5), & \text{ if } i=3, \\
(1,2,6), & \text{ if } i=4. 
\end{cases}
\end{align*}
In this way, we can apply $\delta,\mathtt{Rvac}$ and $\mathtt{DRvac}$ on $k$-Dyck paths.

\subsection{Equivalence of operations on \texorpdfstring{$k$}{k}-Dyck paths}
In this subsection, we consider the case $(a,b)=(1,k)$.
In the previous subsection, we show that the promotion is congruent to the 
rowmotion for the generic $(a,b)$.

In the case of $(a,b)=(1,k)$, the evacuation and the rowvacuation are related as follows.
\begin{theorem}
\label{thrm:evrvac}
The evacuation $\mathtt{ev}$ is congruent to $\delta^{-k}\circ\mathtt{Rvac}$,
that is, $\mathtt{ev}\cong \delta^{-k}\circ\mathtt{Rvac}$.
\end{theorem}

Before proceeding to the proof of Theorem \ref{thrm:evrvac}, we introduce two lemmas regarding the 
action of $\widehat{\mathtt{ev}}:=\mathtt{Mat}^{-1}\circ\mathtt{ev}\circ\mathtt{Mat}$ on two 
$k$-Dyck paths.

Let $P_{\max}$ (resp. $P_{\min}$) be the top (resp. bottom) $k$-Dyck path, i.e., 
$P_{\max}$ (resp. $P_{\min}$) has an integer sequence $K(P_{\max})=(\overline{1},\overline{2},\ldots,\overline{n})$
(resp. $K(P_{\min})=(k+1,2k+1,\ldots,(n-1)k+1,\overline{n})$), where $K(P)$ is defined in Definition \ref{defn:Kp}.

We first study the action of $\widehat{\mathtt{ev}}$ on $\delta^{-k}(P_{\max})$ and 
$\delta^{-2}(P_{\min})$.

\begin{lemma}
\label{lemma:evrvac1}
The action of $\widehat{\mathtt{ev}}$ on $Q:=\delta^{-k}(P_{\max})$ gives the set of integers with or without a bar
\begin{align}
\label{eq:KQmax}
K(Q)=\{k+1,n+2k,n+3k,\ldots,n+(n-r_{\max}-1)k,\overline{k},\overline{2k},\ldots,\overline{r_{\max}k},\overline{n}\}
\end{align}
where $r_{\max}:=\lfloor n/k-1\rfloor$.
The $k$-Dyck path $Q$ is obtained from $K(Q)$ by sorting the entries in the right order, and applying 
$\mathtt{Mat}^{-1}$.	
\end{lemma}
\begin{proof}
The $k$-Dyck path $\delta^{-k}(P_{\max})$ has a sequence $K(\delta^{-k}(P_{\max}))$ of integers with or without a bar:
\begin{align*}
K(\delta^{-k}(P_{\max}))
=\begin{cases}
(k(n-1)+1,k(n-1)+2,\ldots, kn,\overline{k+1},\overline{k+2},\ldots,\overline{n}), & \text{ for } k\le n-1, \\
(k(n-1)+1,k(n-2)+2,\ldots,(n-1)k+n-1,\overline{n}), & \text{ for } k\ge n.
\end{cases}
\end{align*}
We consider two cases: A) $k\le n-1$, and B) $k\ge n$.

Case A).
The integer sequence $K(\delta^{-k}(P_{\max}))$ gives $n$ matching blocks
\begin{align*}
&\{(n-1)k+1,k(n-1),\ldots,(n-2)k+1\}, \\
&\{(n-1)k+r,(n-r)k,(n-r)k-1,\ldots,(n-r-1)k+1\}, \text{ for } 2\le r\le k, \\
&\{\overline{k+1},\overline{k},\ldots,\overline{1}\}, \\
&\{\overline{k+r},(r-2)k+1,(r-2)k+2\ldots,(r-1)k\}, \text{ for } 2\le r\le n-k.
\end{align*}
The action of $\mathtt{ev}$ on these $n$ matching blocks gives 
\begin{align*}
&\{n+k,n+k+1,\ldots,n+2k\}, \\
&\{n+k-r+1,n+kr+1,n+kr+2,\ldots,n+(r+1)k\}, \text{ for } 2\le r\le k, \\
&\{1,2,\ldots,k+1\}, \\
&\{k+r,n+(n-r+2)k,n+(n-r+2)k-1,\ldots,n+(n-r+1)k+1\}, \text{ for } 2\le r\le n-k.
\end{align*}
The map $\mathtt{Mat}^{-1}$ gives the following integer sequence $K(Q)$ 
for a $k$-Dyck path $Q$.
We have the following two subcases:
\begin{enumerate}[(a)]
\item Suppose 
$n\ge k(k+1)/(k-1)$. Then, we have 
\begin{align*}
K(Q)=&\{k+1,n+2k,n+3k,\ldots,n+(n-r_{\min}+2)k,
\overline{k},\overline{2k},\ldots,\overline{(r_{\min}-3)k},\overline{n}\}
\end{align*}
where 
\begin{align*}
r_{\min}:=\min\left\{r\in\mathbb{N} : 2+\genfrac{}{}{}{}{n}{k}\le r \right\}.
\end{align*}

\item Suppose $n< k(k+1)/(k-1)$.
We have 
\begin{align*}
K(Q)=\{k+1,n+2k,n+3k,\ldots,n+(r_{\max}+1)k,
\overline{k},\overline{2k},\ldots,\overline{(n-r_{\max}-2)k},\overline{n}\},
\end{align*}
where 
\begin{align*}
r_{\max}:=\max\left\{r\in\mathbb{N}: r\le n-1-\genfrac{}{}{}{}{n}{k}\right\}.
\end{align*}
\end{enumerate}
From a) and b), we have 
\begin{align*}
K(Q)=\{k+1,n+2k,\ldots, n+(n-r_{\max}-1)k,
\overline{k},\overline{2k},\ldots,\overline{r_{\max}k},\overline{n}
\}
\end{align*}
where $r_{\max}=\lfloor n/k-1\rfloor$.

Case B).
The integer sequence $K(\delta^{-k}(P_{\max}))$ gives $n$ matching blocks
\begin{align*}
&\{(n-1)k+1,k(n-1),\ldots, (n-2)k+1\}, \\
&\{(n-1)k+r,(n-r)k,(n-r)k-1,\ldots,(n-r-1)k+1\}, \text{ for } 2\le r\le n-1, \\
&\{\overline{n},\overline{n-1},\ldots,\overline{1},(n-1)k+n,(n-1)k+n+1,\ldots,nk\}.
\end{align*}
The action of $\mathtt{ev}$ on these $n$ matching blocks gives 
\begin{align*}
&\{n+k,n+k+1,\ldots,n+2k\}, \\
&\{n+k-r+1,n+kr+1,n+kr+2,\ldots,n+(r+1)k\}, \text{ for } 2\le r\le n-1, \\
&\{1,2,\ldots,n,n+1,\ldots,k+1\}
\end{align*}
The map $\mathtt{Mat}^{-1}$ gives the following integer sequence $K(Q)$ for a $k$-Dyck path $Q$.
We have the following two subcases:
\begin{enumerate}[(a)]
\item If $n=2$, we have $K(Q)=(k+1,\overline{2})$.
\item If $n\ge3$, we have 
\begin{align*}
K(Q)=\{n+2k,n+3k,\ldots,n+(r_{\max}+1)k,\overline{n},\overline{(n-r_{\max}-2)k},\overline{(n-r_{\max}-3)k},\ldots,\overline{k},k+1\},
\end{align*}
where $r_{\max}:=\lfloor (n-1)(k-1)/k\rfloor$.
Since $n\ge k$, $r_{\max}=n-2$. 
As a result, we have 
\begin{align*}
K(Q)=\{k+1,n+2k,n+3k,\ldots,n+(n-1)k,\overline{n}\}.
\end{align*}
\end{enumerate}
From Cases A) and B),  we obtain the set $K(Q)$ of integers with or without a bar
as in Eq. (\ref{eq:KQmax}). 
\end{proof}

\begin{lemma}
\label{lemma:evrvac2}
Let $Q=\widehat{\mathtt{ev}}(\delta^{-2}(P_{\min}))$. The sequence $K(Q)$ of integers with or without
a bar is obtained from the following set of integers:
\begin{align*}
S=\{n+k+1,n+2k+1,\ldots,n+(n-r_{0}+1)k+1, 
\overline{k-1},\overline{2k-1},\ldots,\overline{(r_{0}-2)k-1},\overline{n}\},
\end{align*}
with the integer $r_{0}$ such that $(n+k+1)/k\le r_{0}\le(n+2k)/k$.
$K(Q)$ is obtained from $S$ by sorting the entries in the right order.
\end{lemma}
\begin{proof}
The $k$-Dyck path $\delta^{-2}(P_{\min})$ has a sequence of integers with or without a bar
\begin{align*}
K(\delta^{-2}(P_{\min}))=
(nk-1,\overline{2},\ldots,\overline{n}).
\end{align*}
This integer sequence give the following $n$ matching blocks
\begin{align*}
&\{nk,nk-2,\ldots, (n-1)k\}, \\
&\{\overline{2},\overline{1},1,2,\ldots,k-1\}, \\
&\{\overline{r},(r-2)k,(r-2)k+1,\ldots,(r-1)k-1\}, \text{ for } 3\le r\le n.
\end{align*}
The action of $\mathtt{ev}$ on these $n$ matching blocks gives 
\begin{align*}
&\{n+1,n+2,\ldots,n+k+1\}, \\
&\{1,2,n+nk,n+nk-1,\ldots,n+(n-1)k+2\}, \\
&\{r,n+(n-r+2)k+1,n+(n-r+2)k,\ldots,n+(n-r+1)k+2\}, \text{ for } 3\le r\le n.
\end{align*}
The action of $\mathtt{Mat}^{-1}$ on $n$ matching blocks gies the 
set $K(Q)$ of integers with or without a bar for a $k$-Dyck path $Q$.
We have 
\begin{align*}
K(Q)=\{n+k+1,n+2k+1,\ldots,n+(n-r_{0}+1)k+1, 
\overline{k-1},\overline{2k-1},\ldots,\overline{(r_{0}-2)k-1},\overline{n}\},
\end{align*}
where $r_{0}$ is the integer such that 
\begin{align*}
\genfrac{}{}{}{}{n}{k}+1+\genfrac{}{}{}{}{1}{k}\le r_{0}\le\genfrac{}{}{}{}{n}{k}+2.
\end{align*}
\end{proof}

\begin{proof}[Proof of Theorem \ref{thrm:evrvac}]
Let $\widehat{\mathtt{ev}}:=\mathtt{Mat}^{-1}\circ\mathtt{ev}\circ\mathtt{Mat}$.
We have $\widehat{\mathtt{ev}}^{2}=\mathtt{id}$, 
and $\delta\circ\widehat{\mathtt{ev}}=\widehat{\mathtt{ev}}\circ\delta^{-1}$
by Theorem \ref{thrm:rowMat}.
By Proposition \ref{prop:Rvac}, the rowvacuation $\mathtt{Rvac}$ satisfies $\mathtt{Rvac}^2=\mathtt{id}$,
and $\delta\circ\mathtt{Rvac}=\mathtt{Rvac}\circ\delta^{-1}$.
From these, $\widehat{\mathtt{ev}}$ is written as $\widehat{\mathtt{ev}}=c\circ\mathtt{Rvac}$ 
such that $c$ satisfies $c\circ\mathtt{Rvac}=\mathtt{Rvac}\circ c^{-1}$.
Then, $c$ is written as $\delta^{X}$ with some integer $X$.
In what follows, we show that $X=-k$.

We make use of the induction on $n$. For $n=1,2$, it is straightforward to see $X=-k$.
The previous paragraph implies that $\widehat{\mathtt{ev}}=\delta^{X}\circ\mathtt{Rvac}$.
To fix $X$, we consider the $k$-Dyck path $P$ such that $P=\delta^{-k}(P_{\max})$.
Then, $\widehat{\mathtt{ev}}(P)=\delta^{X+k}\circ\mathtt{Rvac}(P_{\max})$.
By the definition of $\mathtt{Rvac}$, we have 
\begin{align*}
\mathtt{Rvac}(P_{\max})=\mathtt{Rvac}(\delta^{k-2}(P'_{\min}))\oplus Q,
\end{align*}
where $P'_{\min}$ is the lowest $k$-Dyck path of size $n-1$, $Q$ is the $k$-Dyck path of size one,
and $\oplus$ stands for a concatenation of two $k$-Dyck path. 
We have 
\begin{align*}
\mathtt{Rvac}(\delta^{k-2}(P'_{\min}))=&\delta^{-(k-2)}\circ\mathtt{Rvac}(P'_{\min}), \\
&=\delta^{2}\circ\widehat{\mathtt{ev}}(P'_{\min}), \\
&=\widehat{\mathtt{ev}}(\delta^{-2}(P'_{\min})),
\end{align*}
where we have used the induction hypothesis in the second line.
From these, we have 
\begin{align}
\label{eq:evPXk}
\mathtt{ev}(P)=\delta^{X+k}\left(
\widehat{\mathtt{ev}}(\delta^{-2}(P'_{\min}))\oplus Q
\right).
\end{align}
From Lemmas \ref{lemma:evrvac1} and \ref{lemma:evrvac2}, it is easy to see that 
\begin{align*}
\mathtt{ev}(P)=\widehat{\mathtt{ev}}(\delta^{-2}(P'_{\min}))\oplus Q.
\end{align*}
By comparing this with Eq. (\ref{eq:evPXk}), we have $X+k=0$, which implies $X=-k$.
This completes the proof.
\end{proof}

Since we consider $k$-Dyck paths, which are bijective to non-crossing weighted 
partitions, we consider operations congruent to $\mathtt{Kre}$, $\mathtt{SU}$ and 
$\mathtt{LK}$.

\begin{prop}
\label{prop:Matdelta}
We have the following congruences:
\begin{enumerate}
\item $\partial^{-nk+k-1}\cong\mathtt{DRvac}\circ\mathtt{Rvac}$, 
\item $\mathtt{SU}\cong\delta^{-(k-1)}\circ\mathtt{Rvac}$,
\item $\mathtt{LK}\cong\delta^{-2k}\circ\mathtt{Rvac}$,
\item $\mathtt{Kre}^2\cong\delta^{-(k+1)}$.
\end{enumerate}
\end{prop}
\begin{proof}
(1) follows from (4) in Proposition \ref{prop:Rvac}, $r=(n-1)k-1$, and Theorem \ref{thrm:rowMat}.

(2), (3), and (4) follows from Proposition \ref{prop:kDyckNC}, Theorems \ref{thrm:rowMat} and
\ref{thrm:evrvac}.
\end{proof}

\subsection{The map \texorpdfstring{$\widehat{RSK}^{-1}$}{RSK-1} on \texorpdfstring{$k$}{k}-Dyck paths}
The map $\widehat{RSK}^{-1}$ can be expressed in terms of the map $\mathtt{DT}$ as in Proposition \ref{prop:RSKDT}.
We use this correspondence to define $\widehat{RSK}^{-1}$ since we have no analogue of $\widehat{RSK}$ for 
$(a,b)\neq(1,1)$.
In the case of $(a,b)=(1,k)$, one can define the cover-inclusive $k$-Dyck tilings as a generalization 
of the cover-inclusive Dyck tilings.
We briefly recall the definition of $k$-Dyck tilings following \cite{JosVerKim16,Shi21}.

A $k$-Dyck tiling is a tiling of the region which is above a $k$-Dyck path and below the top $k$-Dyck 
path $U^{n}R^{kn}$.
The difference between a $k$-Dyck tiling and a Dyck tiling is that we use a $k$-Dyck tile instead 
of a Dyck tile. A $k$-Dyck tile is a tile consisting of several unit boxes such that the centers of the 
boxes form a $k$-Dyck path.
The cover-inclusiveness of a $k$-Dyck tiling is the same as that of a Dyck tiling, i.e., the sizes of  
$k$-Dyck tiles are weakly decreasing from bottom to top.
A cover-inclusive $k$-Dyck tiling is said to be maximal if it consists of tiles with the largest size.

Let $P$ be a $k$-Dyck path and $D_{\max}(P)$ be the maximal $k$-Dyck tiling above $P$.
We denote by $\mathcal{D}(P)$ the set of $k$-Dyck tiles in $D_{\max}(P)$.
As in the case of a Dyck tiling studied in Section \ref{sec:cPMRSK}, we associate a transposition $t_{i,j}$ with a $k$-Dyck 
tile in $\mathcal{D}(P)$.
That is, if the south-most (resp. rightmost) edge of a $k$-Dyck tile $d$ is $i$-th (resp. $j$-th)
up (resp. right) step, we associate $t_{i,j}$ to $d$.
Let $w:=w(P)$ be the product of the transpositions $t_{i,j}$ such that the order of $t_{i,j}$ 
is compatible with the order of $k$-Dyck tiles.

Let $S(P)$ be the perfect matching corresponding to $P$. 
The action of $t_{i,j}$ on $S(P)$ is given by acting a transposition on $S(P)$ 
as the exchange of $i$ and $j$.
Then, we have the action of $w$, which is a permutation obtained from $P$, on $S(P)$, 
and obtain a new collection $\widehat{S}(P)$ of sets of integers. 
Note that $\widehat{S}(P)$ is not a perfect matching in general. 
By construction, the integers $1,\ldots,n$ belong to distinct blocks in $\widehat{S}(P)$.

Since we consider the maximal $k$-Dyck tiling above $P$, the integers $i\in[1,n]$ belong 
to distinct blocks in $\widehat{S}(P)$.
Let $S_{i}$, $1\le i\le n$, be the block which the integer $i$ belongs to, and 
$S'_{i}:=S_{i}\setminus\{i\}$.
We define a sequence of integers $\kappa(P):=(\kappa_1,\ldots,\kappa_{n})$ by 
\begin{align*}
\kappa_{i}:=\#\left\{   
j\in \bigcup_{1\le k\le n-i}S'_{k} \quad \Big| \quad 
j<\min S'_{n+1-i}
 \right\}.
\end{align*}
The sequence $\kappa(P)$ corresponds to a kind of the inversion number of $\widehat{S}(P)$.
We define the sequence of integers $\widehat{Y}=(y_1,y_2,\ldots,y_{n})$ by
\begin{align*}
y_{i}=(n-i)k-\kappa_{i}, 
\end{align*}
for $1\le i\le n$.
Note that, in general, the sequence $\widehat{Y}$ does not give a Young diagram.
Let $Y_{\max}$ be the maximal Young diagram inside $\widehat{Y}$.
More precisely, $Y_{\max}=(y_1^{\max},\ldots,y_{n}^{\max})$ is given by 
$y_{1}^{\max}=y_{1}$ and $y_{i}^{\max}=\min\{y_{i-1}^{\max},y_{i}\}$ for $2\le i\le n$.
The Young diagram $Y_{\max}$ characterizes 
a $k$-Dyck path $Q$ such that the region above $Q$ and below the top $k$-Dyck path is 
the Young diagram $Y_{\max}$.

\begin{defn}
We define the above map  as $\widehat{RSK}^{-1}:P\mapsto Q$.
\end{defn}

\begin{example}
We consider two examples.
Let $P$ be a $2$-Dyck path whose step sequence is $(1,4,7)$. 
The maximal $2$-Dyck tiling above $P$ is given by
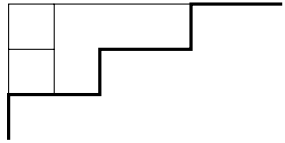
\begin{center}
\tikzpic{-0.5}{[scale=0.6]
\draw[very thick](0,0)--(0,1)--(2,1)--(2,2)--(4,2)--(4,3)--(6,3);
\draw(1,1)--(1,3)--(4,3)(0,1)--(0,3)--(1,3)(0,2)--(1,2);
}
\end{center}
Then, this $2$-Dyck tiling gives $w=t_{3,4}t_{2,3}t_{3,7}$.
Here, the transpositions $t_{3,4}$, $t_{2,3}$ and $t_{3,7}$ correspond to
the top cell, the cell below the top cell, and the $2$-Dyck tile of size one.

The perfect matching 
is $S(P)=\{\{1,2,3\},\{4,5,6\},\{7,8,9\}\}$. The action of $w$ on $S(P)$ is given by
\begin{align*}
S(P)&\xrightarrow{t_{3,7}}\{\{1,2,7\},\{4,5,6\},\{3,8,9\}\}
\xrightarrow{t_{2,3}}\{\{1,3,7\},\{4,5,6\},\{2,8,9\}\}, \\
&\xrightarrow{t_{3,4}}\{\{1,4,7\},\{3,5,6\},\{2,8,9\}\}=\widehat{S}(P).
\end{align*}
We have $\kappa(P)=(1,2,0)$, and $\widehat{Y}=(3,0,0)$.
The maximal Young diagram inside $\widehat{Y}$ is $(3,0,0)$.
The Young diagram which characterizes a $2$-Dyck path is $Y=(3,0,0)$, which 
implies that we have the $2$-Dyck path whose step sequence is $126$.
Therefore, $\widehat{RSK}^{-1}(147)=126$.

Let $P'$ be a $2$-Dyck path with the step sequence $127$, and we have $\kappa(P')=(4,0,0)$.
From this, we have $\widehat{Y}=(0,2,0)$. Note that $\widehat{Y}$ is not a Young diagram.
The maximal Young diagram inside $\widehat{Y}$ is $(0,0,0)$, and this maximal Young diagram
gives the $2$-Dyck path with the step sequence $123$. We have $\widehat{RSK}^{-1}(127)=(123)$.
\end{example}

The above algorithm for $\widehat{RSK}^{-1}$ with $(a,b)=(1,k)$ can be seen as 
a generalization of the map $\widehat{RSK}^{-1}$ for $(a,b)=(1,1)$ defined 
in Section \ref{sec:MatRSK}.
We will characterize the integer sequence $\kappa(P)$ for a $k$-Dyck path 
in terms of the maximal Dyck tiling $D_{\max}(P)$.
For this purpose, we introduce the notion of an {\it Hermite history}
of a Dyck tiling.
A Dyck tiling consists of Dyck tiles, and Dyck tiles consist of 
north, south, east and west edges which surround the tile.
Let $d$ be a $k$-Dyck tile.
We put a line from the west edge at the bottom to the east edge at the top 
in $d$.
We concatenate the lines of Dyck tiles $d_1$ and $d_2$ if the east edge at the top of $d_1$ 
is the west edge at the bottom of $d_2$.
By definition, the lines start from the up steps in the top $k$-Dyck path. Therefore, we have 
$n$ lines in $D_{\max}(P)$, the bottom line is of length zero.
We call the set of these lines an Hermite history.
Let $H(P):=(h_1,\ldots,h_{n})$ be the numbers of $k$-Dyck tiles in the $i$-th line from top 
in the Hermite history of $D_{\max}(P)$.

\begin{lemma}
We have $\kappa(P)=H(P)$.
\end{lemma}
\begin{proof}
In $D_{\max}(P)$, a Dyck tile corresponds to a transposition $t_{i,j}$.
To obtain $\widehat{S}(P)$, we apply the transpositions $t_{i,j}$ to $S(P)$.
Then, we have a sequence of collections of sets:
\begin{align}
\label{eq:seqSP}
S(P)\xrightarrow{(i_1,j_1)}S_{1}(P)\xrightarrow{(i_2,j_2)}S_{2}(P)\xrightarrow{(i_3,j_3)}
\cdots \xrightarrow{(i_r,j_r)}S_{r}(P)=\widehat{S}(P),
\end{align}
where $(i,j)$ stands for the transposition $t_{i,j}$.
By construction, $j_{l}$ in Eq. (\ref{eq:seqSP}) is a minimum element 
in a block of $S_{l-1}(P)$.
If a $k$-Dyck tile is of size zero, the transposition exchanges $i_{l}$ and $i_{l}+1$ in $S_{l-1}(P)$.
These imply that if $D_{\max}(P)$ does not contain a $k$-Dyck tile of size greater than zero,
then $\kappa(P)=H(P)$.

We consider the case where $D_{\max}(P)$ contains a $k$-Dyck tile $d$ of size greater than zero.
The $k$-Dyck tile $d$ corresponds to a transposition $(i,j)$ with $j\neq i+1$.
Suppose that $j$ is the minimum element of a block and the $t$-th smallest element among 
the minimum elements in blocks. If the $k$-Dyck tile $d$ has the size $s$, then, by the transposition $(i,j)$, 
$i$ becomes the minimum element of a block and the $t-s$-th smallest element among the minimum elements in blocks.
Note that the width of a $k$-Dyck tile of size $s$ is $sk+1$, and each block in $S_{l}(P)$ consists of $k+1$ integers.
The transposition $(i,j)$ corresponds to the change of the order of blocks. 
Then, it is easy to see that we have $\kappa(P)=H(P)$.
\end{proof}

\begin{prop}
\label{prop:RSKk}
We have the following commutative diagram:
\begin{align}
\label{eq:cd6}
\tikzpic{-0.5}{[scale=0.8]
\node (0) at (0,0){$P$};
\node (1) at (3,0){$P'$};
\node (2) at (0,-2){$Q$};
\node (3) at (3,-2){$Q'$};
\draw[->,anchor=south] (0) to node {$\delta$} (1);
\draw[->,anchor=south] (2) to node {$\partial^{-1}$} (3);
\draw[->,anchor=east] (0) to node {$\widehat{RSK}$}(2);
\draw[->,anchor=west] (1) to node {$\widehat{RSK}$}(3);
}
\end{align}
\end{prop}
\begin{proof}
We will show that $\widehat{RSK}^{-1}\circ\partial^{-1}=\delta\circ\widehat{RSK}^{-1}$.
Let $P$ be a $k$-Dyck path and $S(P)$ be a perfect matching of $P$.
Let $l$ be the minimum element in a block $B$ of $S(P)$ such that $B$ contains the integer $(k+1)n$.
We consider two cases: 1) $l=(k+1)(n-1)+1$, and 2) $l\le(k+1)(n-1)$.

Case 1). We consider the maximal $k$-Dyck tiling $D_{\max}(P)$ above $P$.
Let $\kappa(P)=(\kappa_1,\ldots,\kappa_{n})$ be the numbers of $k$-Dyck tiles in the Hermite history of $D_{\max}(P)$.
Let $\alpha$ be the minimum integer such that $\kappa_{\alpha}\neq0$ and $\kappa_{\beta}=0$ for $\beta<\alpha$.
We consider the path $Q=\partial^{-1}(P)$.  Since $l=(k+1)(n-1)+1$, 
$P$ is expressed as a concatenation of two words $P'\circ UR^{k}$ where $P'$ is a $k$-Dyck 
path of size $n-1$. We have a $k$-Dyck tile of size $n-\alpha$ in $P$. 
Then, by promotion, we have $Q=U\circ P'\circ R^{k}$. 
The sequence $\kappa(Q)=(\kappa'_1,\ldots,\kappa'_{m})$ is given by 
\begin{align}
\label{eq:kappaQ}
\kappa'_{i}=
\begin{cases}
0, & \text{ if } i\le \alpha, \\
\kappa_{i}+1, & \text{ otherwise}.
\end{cases}
\end{align}
We consider the action of $\delta$ on $T=\widehat{RSK}^{-1}(P)$.
Since $D_{\max}(P)$ is maximal, we have 
$\kappa_{i}<\kappa_{i-1}+k$ for $\alpha+1\le i$.
Further, we have $\kappa_{\alpha}=(\alpha-1)k$ due to the existence 
of the $k$-Dyck tile of size $n-\alpha$.
These imply that we have a valley in the $j$-th row with $\alpha+1\le j$,
and the $\alpha$-th step sequence is zero in $T$.
Therefore, we can add a unit box in the $j$-th row as an action of $\delta$,
and we have $\kappa(\delta(T))=\kappa(Q)$.

Case 2).
Let $Q=\partial^{-1}(P)$, and $T=\widehat{RSK}^{-1}(P)$.
Since the $k$-Dyck path $P$ is written as $P=P'\circ UR^{a}\circ P''\circ R^{k-a}$, where 
$P'$ and $P''$ are $k$-Dyck path of smaller size, 
$Q$ is expressed as $Q=U\circ P'\circ R^{a+1}\circ P''\circ R^{k-a-1}$ where $a\in[0,k-1]$.
We denote by $\kappa(P)=(\kappa_1,\ldots,\kappa_{n})$ and $\kappa(Q)=(\kappa'_1,\ldots,\kappa'_{n})$.
We will show that $\kappa(Q)=\kappa(\delta(T))$.
Recall that $\kappa_{i}$ is the number of $k$-Dyck tiles in the $i$-th line in the Hermite history
of $D_{\max}(P)$. 
To obtain $\kappa'_{i}$ from $\kappa_{i}$, we consider the following reconnection of the lines in 
the Hermite history.
We start from the $i$-th up step in the top $k$-Dyck path, and go right or up along Dyck tiles.
If we arrive at the unit box which is in the $\alpha-p$-th column from left and $p$-th row 
from bottom where $p\in[1,\alpha]$, 
we go up one unit, and move to the $i+1$-th line in the Hermite history.
We continue to go right or up along the Dyck tiles in the $i+1$-th line.
Then, $\kappa'_{i}$ is the number of $k$-Dyck tiles in the $i$-th new line of the Hermite history.
This reconnection is well-defined since the unit box in the $\alpha-p$-th column and $p$-th row
forms a $k$-Dyck tile of size zero.

We consider the sequence $\kappa(\delta(T))$.
The action of $\delta$ on $T$ is characterized by the position of valleys in $T$.
Let $\widehat{\kappa}'_{i}=(n-i)k+\kappa_{i}$.
We consider the integer sequence $\widehat{\kappa}:=(\widehat{\kappa}_1,\ldots,\widehat{\kappa}_{n})$
such that 
\begin{align}
\widehat{\kappa}_{i}=\max\{\widehat{\kappa}'_{j} : j\ge i\}.
\end{align}
We have a valley at $i+1$-th row from bottom if $\widehat{\kappa}_{i}>\widehat{\kappa}_{i+1}$.
Then, the action of $\delta$ on the $i$-th row in $T$ is given by $\widehat{\kappa}_{i+1}\mapsto\widehat{\kappa}_{i+1}+1$.
The position of the valley in the $i$-th row can be detected by $\kappa_{i+1}$ which corresponds to the $i+1$-th row.
The reconnection of lines in the Hermite history for $P$ corresponds to the detection of the positions of the valleys 
in $T$. From these, we have $\kappa(\delta(T))=\kappa(Q)$.

From Cases 1) and 2), we have $\widehat{RSK}^{-1}\circ\partial^{-1}=\delta\circ\widehat{RSK}^{-1}$.
\end{proof}

The map $\widehat{RSK}$ can be expressed in terms of the matching map $\mathtt{Mat}$ and 
the promotion $\partial$.
\begin{prop}
\label{prop:RSKMatk}
We have $\widehat{RSK}=\partial^{-(n-1)}\circ\mathtt{Mat}$.
\end{prop}
\begin{proof}
From Theorem \ref{thrm:rowMat} and Proposition \ref{prop:RSKk}, the map $\widehat{RSK}$ is  
written as $\widehat{RSK}=C\circ\mathtt{Mat}$ such that 
$\partial\circ C=C\circ\partial$. This implies that $C=\partial^{X}$ with some $X$.
We will show that $X=-(n-1)$.

We compare the actions of $\widehat{RSK}^{-1}$ and $\delta^{-(n-1)}\circ\mathtt{Mat}^{-1}$ on
a $k$-Dyck path $P$ whose step sequence $\mathbf{u}(P)$ is 
\begin{align*}
\mathbf{u}(P)=(1,k+1,2k+1,\ldots,(n-1)k+1).
\end{align*}
Then, we have a sequence of $k$-Dyck paths:
\begin{align}
\label{eq:Matdel}
P\xrightarrow{\mathtt{Mat}^{-1}}P'\xrightarrow{\delta^{-(n-1)}}P'',
\end{align}
where the step sequences of $P'$ and $P''$ are given by 
\begin{align*}
\mathbf{u}(P')&=(1,2,\ldots,n), \\
\mathbf{u}(P'')&=(1,3,5,\ldots,2n-1).
\end{align*}
On the other hand, we have 
\begin{align}
\label{eq:RSKkinv}
P\xrightarrow{\widehat{RSK}^{-1}}P''.
\end{align}
By comparing Eq. (\ref{eq:Matdel}) with Eq. (\ref{eq:RSKkinv}),
we have $\widehat{RSK}=\partial^{-(n-1)}\circ\mathtt{Mat}$.
\end{proof}

\begin{remark}
From Theorem \ref{thrm:RSKMat}, Proposition \ref{prop:cd1} and Corollary \ref{cor:rmev},
we have $\widehat{RSK}=\partial^{-(n-1)}\circ\mathtt{Mat}$ for $(a,b)=(1,1)$.
Recall that we have no analogue of RSK correspondence for $(a,b)=(1,k)$ with $k\ge2$, 
however, one can define $\widehat{RSK}^{-1}$ by use of the maximal $k$-Dyck tilings.
Proposition \ref{prop:RSKMatk} implies that the definition of $\widehat{RSK}^{-1}$ 
through the $k$-Dyck tiling is compatible with $\widehat{RSK}=\partial^{-(n-1)}\circ\mathtt{Mat}$
for $(a,b)=(1,1)$. Therefore, it is natural to define $\widehat{RSK}$ by $\partial^{-(n-1)}\circ\mathtt{Mat}$
for general $(a,b)$.
\end{remark}

The evacuation $\mathtt{ev}$ can be expressed in terms of the rowmotion 
and the rowvacuation through the map $\widehat{RSK}$.
\begin{cor}
We have the following commutative diagram:
\begin{align}
\tikzpic{-0.5}{[scale=0.8]
\node (0) at (0,0){$P$};
\node (1) at (2,0){$P''$};
\node (4) at (4,0){$P'$};
\node (2) at (0,-2){$Q$};
\node (3) at (4,-2){$Q'$};
\draw[->,anchor=south] (0) to node {$\delta^{X}$} (1);
\draw[->,anchor=south] (1) to node {$\mathtt{Rvac}$} (4);
\draw[->,anchor=south] (2) to node {$\mathtt{ev}$} (3);
\draw[->,anchor=east] (0) to node {$\widehat{RSK}$}(2);
\draw[->,anchor=west] (4) to node {$\widehat{RSK}$}(3);
}
\end{align}
where $X=2(n-1)+k$.
\end{cor}
\begin{proof}
The claim directly follows from Theorem \ref{thrm:evrvac} and Proposition \ref{prop:RSKMatk}.
\end{proof}

\section{Example}
\label{sec:ex}
We write the step sequence $(u_1,\ldots,u_n)$ as $u_1\ldots u_{n}$ in one-line notation. 
We consider $(a,b)=(1,2)$ and $n=3$, and have twelve $2$-Dyck paths.

We have the bijection between the non-crossing weighted partitions and $2$-Dyck paths.
\begin{align*}
&147\leftrightarrow\{1/2/3,1/2/3\},\quad 146\leftrightarrow\{1/2/3,1/23\}, \quad 
145\leftrightarrow\{1/23,1/23\}, \quad 137\leftrightarrow\{1/2/3,12/3\}, \\
&136\leftrightarrow\{1/2/3,123\},\quad 135\leftrightarrow\{1/23,123\}, \quad
134\leftrightarrow\{1/2/3,13/2\},\quad 127\leftrightarrow\{12/3,12/3\}, \\
&126\leftrightarrow\{12/3,123\},\quad 125\leftrightarrow\{123,123\}, \quad
124\leftrightarrow\{13/2,13/2\},\quad 123\leftrightarrow\{13/2,123\}.
\end{align*}

The action of the promotion $\partial$ on these $2$-Dyck paths is given by
\begin{gather*}
\partial: 147\rightarrow 136\rightarrow 125\rightarrow 147, \\
\partial: 146\rightarrow 135\rightarrow 124\rightarrow 137\rightarrow 126\rightarrow 145
\rightarrow 134\rightarrow 123\rightarrow 127\rightarrow 146.
\end{gather*}
The action of $\mathtt{ev}$ is given by 
\begin{align*}
\mathtt{ev}: 
147\leftrightarrow 147, \quad 146\leftrightarrow 127, \quad 145\leftrightarrow 137, \quad
136\leftrightarrow 125, \quad 135\leftrightarrow 123, \quad 134\leftrightarrow 124, \quad
126\leftrightarrow 126.
\end{align*}

The action of $\mathtt{Kre}$ on non-crossing partitions is given by
\begin{align*}
\mathtt{Kre}: 1/2/3\rightarrow 123\rightarrow 1/2/3, \quad 12/3\rightarrow 1/23\rightarrow 13/2\rightarrow 12/3. 
\end{align*}
Then, the action of $\mathtt{Kre}$ on $2$-Dyck paths is given by
\begin{align*}
&\mathtt{Kre}: 147\rightarrow 125\rightarrow 147, \ 136\rightarrow 136, \\ 
&\mathtt{Kre}: 146\rightarrow 123\rightarrow 137\rightarrow 135\rightarrow 134\rightarrow 126 \rightarrow 146, \
145\rightarrow 124\rightarrow 127 \rightarrow 145.
\end{align*}

The action of $\mathtt{SU}$ on non-crossing partitions of size $3$ is given by
\begin{align*}
\mathtt{SU}: 1/2/3\leftrightarrow 123, \quad 12/3\leftrightarrow 12/3, \quad 1/23\leftrightarrow 13/2, \quad
13/2\leftrightarrow 1/23.
\end{align*}
Therefore, the action of $\mathtt{SU}$ on $2$-Dyck paths are given by
\begin{align*}
\mathtt{SU}: 
147\leftrightarrow 125, \quad 146\leftrightarrow 123, \quad 145\leftrightarrow 124, \quad
137\leftrightarrow 126, \quad 136\leftrightarrow 136, \quad 134\leftrightarrow 135, \quad 
127\leftrightarrow 127.
\end{align*}

The action of $\mathtt{LK}$ on non-crossing partitions of size $3$ is given by
\begin{align*}
\mathtt{LK}: 1/2/3\leftrightarrow 123, \quad 12/3\leftrightarrow 13/2, \quad 1/23\leftrightarrow 1/23.
\end{align*}
From these, we have the following action of $\mathtt{LK}$ on $2$-Dyck paths.
\begin{align*}
\mathtt{LK}: 147\leftrightarrow 125, \quad 146\leftrightarrow 135, \quad 145\leftrightarrow 145, \quad
137\leftrightarrow 123, \quad 136\leftrightarrow 136, \quad 134\leftrightarrow 126, \quad
127\leftrightarrow 124.
\end{align*}

The action of the rowmotion $\delta$ gives two obits:
\begin{gather*}
\delta: 145\rightarrow 137 \rightarrow 126 \rightarrow 145, \\
\delta: 147\rightarrow 136 \rightarrow 125 \rightarrow 134 \rightarrow 127 \rightarrow
146\rightarrow 135 \rightarrow 124 \rightarrow 123 \rightarrow 147.
\end{gather*}

The action of $\mathtt{Rvac}$ on $2$-Dcyk paths is given by 
\begin{gather*}
\mathtt{Rvac}: 
147\leftrightarrow 134 \quad 146\leftrightarrow 124 \quad 145\leftrightarrow 137 \quad
136\leftrightarrow 125 \quad 135\leftrightarrow 135 \quad 127\leftrightarrow 123 \quad 
126\leftrightarrow 126 
\end{gather*}
Similarly, the action of $\mathtt{DRvac}$ on the paths is given by 
\begin{gather*}
\mathtt{DRvac}: 
147\leftrightarrow 123 \quad 146\leftrightarrow 134 \quad 145\leftrightarrow 145 \quad
137\leftrightarrow 126 \quad 136\leftrightarrow 124 \quad 135\leftrightarrow 125 \quad 
127\leftrightarrow 127 
\end{gather*}
The action of $\mathtt{Mat}$ on these $2$-Dyck paths is given by
\begin{align*}
&\mathtt{Mat}: 
147\rightarrow 146\rightarrow 126 \rightarrow 125 \rightarrow 123 \rightarrow 135 \rightarrow 137 \rightarrow 147, \\
&\mathtt{Mat}:  
145\rightarrow 136\rightarrow 127\rightarrow 145, \qquad 134\rightarrow 134, \qquad 124\rightarrow 124. 
\end{align*}

The action $\widehat{RSK}^{-1}$ on $2$-Dyck paths is given by
\begin{align*}
&\widehat{RSK}^{-1}: 
147\rightarrow 126\rightarrow 134\rightarrow 136\rightarrow 137\rightarrow 127\rightarrow 123\rightarrow 147, \\
&\widehat{RSK}^{-1}: 146\rightarrow 124\rightarrow 146, \quad 145\rightarrow 125\rightarrow 145, \quad
135\rightarrow 135, \quad 
\end{align*}

\bibliographystyle{amsplainhyper} 
\bibliography{biblio}

\end{document}